\documentclass[12pt,oneside,reqno]{amsart}
\usepackage{amssymb}
\usepackage{}
\usepackage{amsmath}
\usepackage{graphicx}
\usepackage{mathrsfs}
\usepackage{amsfonts}
\usepackage{enumerate,amsmath,amssymb,amsthm}
\usepackage{hyperref}
\hypersetup{backref,
colorlinks=true,
linkcolor=blue,
anchorcolor=blue,
citecolor=blue}

\newcommand{\norm}[1]{\left\lVert #1 \right\rVert}

\numberwithin{equation}{section}

\newcommand{\be}{\begin{eqnarray}}
\newcommand{\ee}{\end{eqnarray}}
\newcommand{\ce}{\begin{eqnarray*}}
\newcommand{\de}{\end{eqnarray*}}
\newtheorem{theorem}{Theorem}[section]
\newtheorem{lemma}[theorem]{Lemma}
\newtheorem{remark}[theorem]{Remark}
\newtheorem{definition}[theorem]{Definition}
\newtheorem{proposition}[theorem]{Proposition}
\newtheorem{Examples}[theorem]{Example}
\newtheorem{corollary}[theorem]{Corollary}
\newtheorem{assumption}{Assumption}[section]

\def\e{{\mathrm{e}}}

\def\e{\epsilon}

\def\[{{\Big[}}
\def\]{{\Big]}}
\def\<{{\langle}}
\def\>{{\rangle}}
\def\({{\Big(}}
\def\){{\Big)}}

\def\bx{{\mathbf{x}}}

\def\dif{{\mathord{{\rm d}}}}

\def\={&\!\!=\!\!&}
\def\bt{\begin{theorem}}
\def\et{\end{theorem}}
\def\bl{\begin{lemma}}
\def\el{\end{lemma}}
\def\br{\begin{remark}}
\def\er{\end{remark}}
\def\bas{\begin{assumption}}
\def\eas{\end{assumption}}
\def\bd{\begin{definition}}
\def\ed{\end{definition}}
\def\bp{\begin{proposition}}
\def\ep{\end{proposition}}
\def\bc{\begin{corollary}}
\def\ec{\end{corollary}}
\def\bx{\begin{Examples}}
\def\ex{\end{Examples}}

\def\cB{{\mathcal B}}

\def\cF{{\mathcal F}}

\def\mD{{\mathbb D}}
\def\mE{{\mathbb E}}
\def\mF{{\mathbb F}}

\def\mN{{\mathbb N}}

\def\mR{{\mathbb R}}

\def\mW{{\mathbb W}}

\def\geq{\geqslant}
\def\leq{\leqslant}

\allowdisplaybreaks
\usepackage{color}

\title[LDP and MDP for DDSPDEs with jumps]{\bf{Large and moderate deviation principles for McKean-Vlasov SDEs with jumps}}

\author{Wei Liu}
\address{Wei Liu, School of Mathematics and Statistics, Wuhan University, Wuhan, Hubei 430072, PR China;
Hubei Key Laboratory of Computational Science, Wuhan University, Wuhan, Hubei 430072, P.R. China.}
\thanks{Wei Liu is supported by NSFC (No. 12071361, 11731009), the Fundamental Research Funds for the Central Universities
(No. 2042020kf0031, 2042020kf0217).}
\email{wliu.math@whu.edu.cn}

\author{Yulin Song}
\address{Yulin Song, Department of Mathematics, Nanjing University, Nanjing, 210093, P.R. China}
\thanks{Yulin Song is supported by NSFC (No. 11971227, 11790272)}
\email{ylsong@nju.edu.cn}

\author{Jianliang Zhai}
\address{Jianliang Zhai, CAS Key Laboratory of Wu Wen-Tsun Mathematics, School of Mathematical Science, University of Science and Technology of China, Hefei 230026, P.R. China}
\thanks{Jianliang Zhai's research is supported by NSFC (No. 11971456, 11671372, 11721101), School Start-up Fund (USTC) KY0010000036, the Fundamental Research Funds for the Central Universities (No. WK3470000016).}
\email{zhaijl@ustc.edu.cn}

\author{Tusheng Zhang}
\address{Tusheng Zhang, School of Mathematics, University of Manchester, Oxford Road, Manchester, M13 9PL, UK}
\thanks{}
\email{Tusheng.Zhang@manchester.ac.uk}
\date{\today}
\begin{document}

\maketitle

\begin{abstract}
\noindent
In this paper, we consider McKean-Vlasov stochastic differential equations (MVSDEs) driven by L\'evy noise. By identifying the right equations satisfied by the solutions of the MVSDEs with  shifted driving L\'evy noise, we build up a framework to fully apply the weak convergence method to establish large and moderate deviation principles for MVSDEs. In the case of ordinary SDEs, the rate function is calculated by using the solutions of the corresponding skeleton equations simply replacing the noise by the elements of the Cameron-Martin space. It turns out that the correct rate function for MVSDEs is defined through the solutions of skeleton equations replacing the noise by smooth functions and replacing the distributions involved in the equation by the distribution of the solution of the corresponding deterministic equation (without the noise). This is somehow surprising. With this approach, we obtain large and moderate deviation principles for much wider  classes of MVSDEs in comparison with the existing literature.

\vspace{1mm}
\end{abstract}
\tableofcontents

\section{Introduction}
The purpose  of this paper is to establish large and moderate deviation principles for the following McKean-Vlasov stochastic differential equations (MVSDEs) driven by  L\'evy noise as the parameter $\epsilon$ tends to $0$,
\begin{align}\label{MVSDE-BLE zhai 3}
\dif X^{\epsilon}(t)
=&b_\epsilon(t,X^{\epsilon},\mu^{\epsilon})\dif t
+\sqrt{\epsilon}\sigma_\epsilon(t,X^{\epsilon},\mu^{\epsilon})\dif W(t)\nonumber\\
&+\epsilon \int_Z G_\epsilon(t,X^{\epsilon},\mu^{\epsilon},z)\widetilde{N}^{\epsilon^{-1}}(\dif z,\dif t),\ t\in[0,T], \epsilon\in(0,1],
\end{align}
where $X^\epsilon:=\{X^\epsilon(s),s\in[0,T]\}$, $\mu^{\epsilon}$ is the law of $X^{\epsilon}$, $W$ is a Brownian motion (BM in short), $N^{\epsilon^{-1}}$ is a Poisson random measure (PRM  in short) on $[0,T]\times Z$ with a $\sigma$-finite intensity measure $\epsilon^{-1}{{\rm Leb}}_T\otimes\nu$, and $\widetilde{N}^{\epsilon^{-1}}([0,t]\times B)=N^{\epsilon^{-1}}([0,t]\times B)-\epsilon^{-1}t\nu(B)$, $\forall B\in \mathcal{B}(Z)$ with $\nu(B)<\infty$, is the compensated PRM.
The precise assumptions on (\ref{MVSDE-BLE zhai 3}) will be given in Section 2. We notice that $b_\epsilon,~\sigma_\epsilon,~G_\epsilon$ are functionals of the path $X^\epsilon$ and the law of $X^\epsilon$. (\ref{MVSDE-BLE zhai 3}) in particular includes the MVSDE:
\begin{align}\label{MVSDE-BLE zhai 3-1}
\dif X^{\epsilon}(t)
=&b_\epsilon(t,X^{\epsilon}(t),\mu^{\epsilon}_t)\dif t
+\sqrt{\epsilon}\sigma_\epsilon(t,X^{\epsilon}(t),\mu^{\epsilon}_t)\dif W(t)\nonumber\\
+&\epsilon \int_Z G_\epsilon(t,X^{\epsilon}(t),\mu^{\epsilon}_t,z)
\widetilde{N}^{\epsilon^{-1}}(\dif t,\dif z),\ t\in[0,T], \epsilon\in(0,1],
\end{align}
where $b_\epsilon(t, \cdot,\cdot),~\sigma_\epsilon(t,\cdot,\cdot)$ and $~G_\epsilon(t,\cdot,\cdot,z)$ depend only on the value and the law of the process $X^{\epsilon}$ at time  $t$. We stress  that the general framework (\ref{MVSDE-BLE zhai 3}) covers stochastic differential equations (SDEs),  stochastic partial differential equations (SPDEs), and SDES/SPDEs with delay/memory etc.

%In particular, when the coefficients $b_\epsilon,~\sigma_\epsilon,~G_\epsilon$  do not depend on the distribution of the solutions , the framework here covers
%a very large class of SPDEs driven by multiplicative L\'evy noise, including stochastic porous medium equation,
%stochastic p-Laplace equation, stochastic Burgers type equations, stochastic 2D Navier-Stokes
%equations and many other stochastic hydrodynamical systems. We refer to \cite{BLZ,Liu R} for more details.
MVSDEs were first suggested by Kac \cite{[K1],[K2]} as a stochastic toy model for the Vlasov kinetic equation of plasma, and then introduced by Mckean \cite{[MC]} to model plasma dynamics. These equations describe limiting behaviors
of individual particles in an interacting particle system of mean-field type when the number of particles goes to infinity (so-called propagation of chaos). For this reason MVSDEs are also referred as mean-field SDEs. The theory and applications of MVSDEs and associated interacting particle systems have been extensively studied by a large number of researchers under various settings due to their wide range of applications in several fields, including physics, chemistry, biology, economics, financial mathematics etc. One can refer to \cite{[BLPR],[HSS],[HW],[MV],[RZ]} and the references therein for the existence and uniqueness of solutions to MVSDEs, \cite{[ADF],[DEGZ],[JW],[MS],[MSSZ],[SZ]} for propagation of chaos,  \cite{[EA],[EGZ],[GLWZ1],[GLWZ2],[HS2020],[LMW],[LWZ],[Mal1],[Mal2],[Song]} for exponential ergodicity and functional inequalities, \cite{ARRST,[RST],[HIP], [LW],[SY]} and the references therein for large deviation principles.

%and \cite{ARRST,[HIP],[RST],[SY]} for large deviation principles (LDPs) and moderate deviation principles (MDPs).

On the other hand, real world  models in finance, physics, biology, etc.,
 sometimes can not be well represented by Gaussian noise. And from the point of view of particle systems, in many scenarios, the individual particles and the related whole population will demonstrate some sudden jumps. {L\'evy-type perturbations come to the stage to reproduce the performance of these natural phenomena, and a PRM is a good and natural model to express these jumps. Thus,  it is natural to consider MVSDEs driven by BM and PRM as follows:
\begin{equation}\label{MVSDE-BLE}
\dif X(t)=b(t,X(t),\mu_t)\dif t+\sigma(t,X(t),\mu_t)\dif W(t)
+\int_Z G(t,X(t-),\mu_t,z)\widetilde{N}(\dif z, \dif t).
\end{equation}

Compared with the MVSDEs driven by BM, the MVSDEs with L\'evy noise have been much less studied. Recently in  \cite{[JMW]}, Jourdain et al. studied the existence, uniqueness and particle approximations for MVSDEs driven by L\'evy noise. In analogy to the case of Gaussian noise (see \cite{[BM],[HRW]}), nonlinear and nonlocal integral Fokker-Planck PDEs can be related to  MVSDEs with L\'evy noise (see \cite{[HL],[JMW],[Li]}).

%%%%%%%%%%%%%%%%%%%%%%%
Large and moderate deviation principles can provide an exponential estimate for tail probability (or the probability of a rare event) in terms of some explicit rate function. In the case of stochastic processes, the heuristics underlying large and moderate deviations theory is to identify a deterministic path around which the diffusion is concentrated with overwhelming probability, so that the stochastic motion can be seen as a small random perturbation of this deterministic path.

%the idea is to find a deterministic path around which the diffusion is concentrated with overwhelming probability. As a consequence, the stochastic motion can be interpreted as a small perturbation of the deterministic path.

Large and moderate deviation principles for classical stochastic evolution equations and SPDEs driven by BM and/or PRM have been extensively investigated in recent years. Among the approaches to deal with these problems, the weak convergence method based on a variational representation for positive measurable functionals of a BM and/or PRM (see \cite{BDG,Budhiraja-Dupuis-Maroulas.,BD 2000, BDM 2008,Budhiraja-Chen-Dupuis}) is proved to be a powerful tool to establish large and moderate deviation principles for various dynamical systems driven by Gaussian noise and/or PRM. The reader is referred to \cite{BDG,Budhiraja-Dupuis-Maroulas.,BD 2000, BDM 2008,Budhiraja-Chen-Dupuis,BPZ,Brz+Manna+Zhai_2018,BGT 2017,DXZZ 2017, DZZ 2017,DWZZ 2020,MSZ,RZ 2008,WZZ 2015,Yang-Zhai-Zhang,Zhai-Zhang} and the references therein.
The key components of the variational representation are the controlled BM and the controlled PRM. The controlled BM basically shifts the mean, while the controlled PRM plays the role of a thinning function.
We  refer to \cite{BD2019} for an excellent review of the advances on the weak convergence method during the past decade.

%%%%%%%%%%%%%%%%%%%%%%%
Assume that there is a unique strong solution $X^\epsilon$ to MVSDE (\ref{MVSDE-BLE zhai 3}). Then, there exists a measurable map  $\mathcal{G}^\epsilon$ such that the solution $X^\epsilon$ can be represented as
\begin{eqnarray}\label{eq zhai 1}
X^\epsilon=\mathcal{G}^\epsilon(\sqrt{\epsilon}W,\epsilon N^{\epsilon^{-1}}).
\end{eqnarray}
One key step to establish the LDP is to prove the weak convergence of the perturbations   $X^{\epsilon,u_\epsilon}:=\mathcal{G}^\epsilon(\sqrt{\epsilon}W+\int_0^\cdot\phi_\epsilon(s)\dif s,\epsilon N^{\epsilon^{-1}\psi_\epsilon})$ as $\epsilon\rightarrow 0$, here $u_\epsilon=(\phi_\epsilon,\psi_\epsilon)$.  It is therefore important to identify the correct equation satisfied by $X^{\epsilon,u_\epsilon}$ in this setting. It would be natural to think that
$X^{\epsilon,u_\epsilon}$ is the solution to
the following controlled SDE:
\begin{align}\label{MVSDE-BLE zhai 1}
\dif X^{\epsilon,u_\epsilon}(t)
=&b_\epsilon(t,X^{\epsilon,u_\epsilon},\mu^{\epsilon,u_\epsilon})\dif t
+\sqrt{\epsilon}\sigma_\epsilon(t,X^{\epsilon,u_\epsilon},\mu^{\epsilon,u_\epsilon})\dif W(t)
+\sigma_\epsilon(t,X^{\epsilon,u_\epsilon},\mu^{\epsilon,u_\epsilon})\phi_\epsilon (t)\dif t\nonumber\\
&+\int_Z G_\epsilon(z,X^{\epsilon,u_\epsilon},\mu^{\epsilon,u_\epsilon},z)
\Big(
\epsilon N^{\epsilon^{-1}\psi_\epsilon}(\dif s,\dif z)-\nu(\dif z)\dif s
\Big)
,\ \ \ t\in[0,T],
\end{align}
where $\mu^{\epsilon,u_\epsilon}$ is the distribution of $X^{\epsilon,u_\epsilon}$.
\vskip 0.3cm
Indeed, this was the claim in \cite{[CHM]}, in which Cai et al. attempted to apply fully weak convergence method to obtain the LDP for MVSDEs with L\'evy noise. Unfortunately this claim is wrong, which leads to a wrong rate function of the LDP.
\vskip 0.3cm
One of the contributions of this paper is to find the {\it correct equation} satisfied by  $X^{\epsilon,u_\epsilon}$. In fact we find out that  $X^{\epsilon,u_\epsilon}$ is actually the  unique solution to the following SDE:
\begin{align}\label{MVSDE-BLE zhai 2}
\dif X^{\epsilon,u_\epsilon}(t)
=&b_\epsilon(t,X^{\epsilon,u_\epsilon},\mu^{\epsilon})\dif t
+\sqrt{\epsilon}\sigma_\epsilon(t,X^{\epsilon,u_\epsilon},\mu^{\epsilon})\dif W(t)
+\sigma_\epsilon(t,Y^{\epsilon,u_\epsilon},\mu^{\epsilon})\phi_\epsilon (t)\dif t\nonumber\\
&+\int_Z G_\epsilon(z,X^{\epsilon,u_\epsilon},\mu^{\epsilon},z)
\Big(\epsilon N^{\epsilon^{-1}\psi_\epsilon}(\dif s,\dif z)-\nu(\dif z)\dif s\Big)
,\ \ \ t\in[0,T],
\end{align}
where $\mu^{\epsilon}$ is the distribution of the solution $X^{\epsilon}$ to (\ref{MVSDE-BLE zhai 3}). The reason is that when perturbing the BM and PRM in the arguments of the map  $\mathcal{G}^\epsilon(\cdot, \cdot)$, $\mu^{\epsilon}$ is already deterministic and hence it is not affected by the perturbation, as the following example shows.

%%%%%%%%%%%%%%%%%%%%%%%%%%%%%%%%%%%%

\bx
Consider the following simple MVSDE:
\begin{eqnarray}\label{eq simple EX}
X^{\epsilon}(t)=x_0+\int_0^t \mathbb{E}(X^{\epsilon}(s))\dif s
+\sqrt{\epsilon}W(t),\ t\in[0,T], \epsilon\in(0,1].
\end{eqnarray}
Here $x_0\in\mathbb{R}$, $W$ is a one-dimensional BM.
Due to the existence and uniqueness of the strong solution, there exists a map $\mathcal{G}^\epsilon$ such that $X^{\epsilon}=\mathcal{G}^\epsilon(\sqrt{\epsilon}W)$. To see the equation satisfied by $Y^\epsilon:=\mathcal{G}^\epsilon(\sqrt{\epsilon}W+\int_0^\cdot\phi_\epsilon (s)\dif s)$, we take  expectation on both sides of (\ref{eq simple EX}) to get
\begin{eqnarray*}
\mathbb{E}(X^{\epsilon}(t))=x_0+\int_0^t \mathbb{E}(X^{\epsilon}(s))\dif s,\ t\in[0,T], \epsilon\in(0,1].
\end{eqnarray*}
Hence $\mathbb{E}(X^{\epsilon}(t))=x_0e^t$.
Thus we have
$$
X^{\epsilon}(t)=
x_0+\int_0^t x_0e^s\dif s+\sqrt{\epsilon}W(t)=\mathcal{G}^\epsilon(\sqrt{\epsilon}W)(t),\ t\in[0,T], \epsilon\in(0,1].
$$
Therefore, $Y^\epsilon:=\mathcal{G}^\epsilon(\sqrt{\epsilon}W+\int_0^\cdot\phi_\epsilon (s)\dif s)$ is the solution of the equation:
\begin{align}
Y^{\epsilon}(t)
=&x_0+\int_0^t x_0e^s\dif s+\sqrt{\epsilon}W(t)+\int_0^t\phi_\epsilon (s)\dif s\nonumber\\
=& x_0+\int_0^t\mathbb{E}[X^{\epsilon}(s)]\dif s+\sqrt{\epsilon}W(t)
+\int_0^t\phi_\epsilon (s)\dif s                                       ,\ t\in[0,T], \epsilon\in(0,1].
\end{align}
$Y^{\epsilon}$ does {\bf NOT} satisfy the following controlled SDE:
\begin{eqnarray}\label{eq simple EX Y}
Y^{\epsilon}(t)
=x_0+\int_0^t \mathbb{E}(Y^{\epsilon}(s))\dif s
+\sqrt{\epsilon}W(t)+\int_0^t\phi_\epsilon (s)\dif s,\ t\in[0,T], \epsilon\in(0,1].
\end{eqnarray}

\ex

%until now, and we do not know how to prove their result.Since finding the maps $\Gamma^\epsilon$ satisfying (\ref{eq zhai 1}) and proving $\Gamma^\epsilon(\sqrt{\epsilon}W+\int_0^\cdot\psi_\epsilon(s)ds,\epsilon N^{\epsilon^{-1}\varphi_\epsilon})$ is the unique solution to (\ref{MVSDE-BLE zhai 2}) are basic steps to applying the weak convergence method, in this paper we give a rigorous proof of these results. And this is one of main contributions of this paper.

%$C([0,T],K)\times M_{FC}\big([0,T]\times Z\big)\rightarrow\mD$

%%%%%%%%%%%%%%%%%%%%%%%%%%%%%%%%%%%%%%%%%%%%%%%%%%%%%%%%%%%%%%%%%%%%%%%%%%%%%%%%%%%%%%%%%%%%%%%%%%%%%%

Large and moderate deviations for  MVSDEs and MVSPDEs, especially driven by L\'evy noise, have been much less studied. For the MVSDE \eqref{MVSDE-BLE zhai 3} without jumps,  Herrmann et al. \cite{[HIP]} obtained the large deviation principle (LDP) in path space equipped with the uniform norm, assuming  the  superlinear growth of the drift but imposing coercivity condition, and a constant diffusion coefficient. Dos Reis et al. \cite{[RST]} obtained LDPs in path space topologies under the assumption that  coefficients $b$ and $\sigma$ have some extra regularity with respect to time. The approach in \cite{[HIP]} and \cite{[RST]} is to  first replace the distribution $\mu_t^{\epsilon}$ of $X_t^{\epsilon}$ in the coefficients with a Dirac measure $\delta_{X^0(t)}$, where $X^0$ is the solution to the following ordinary differential equation
\begin{equation}\label{MVSDE-ODE}
\dif X^0(t)= b(t,X^0(t),\delta_{X^0(t)})\dif t,
\end{equation}
 and then to use discretization, approximation and exponential equivalence arguments.
 Carrying out similar arguments, Adams et al. \cite{ARRST} studied a class of reflected McKean-Vlasov diffusions.
 These techniques
 require more stringent conditions on the coefficients.

Recently Suo and Yuan \cite{[SY]} obtained a moderate deviation principle (MDP) for MVSDEs driven by BM, assuming the  Lipschitz conditions on coefficients $b,\sigma$ and on the Lyons derivative of the coefficients $b$.
They used the weak convergence approach to first establish the MDP for the following SDEs
%prove that the solution $X^{\epsilon}$ to (\ref{MVSDE-BLE zhai 3}) is exponentially equivalent to the solution $Y^{\epsilon}$ for the following SDEs:
\begin{equation}\label{eq Y zhai}
\dif Y^{\epsilon}(t)
=b_\epsilon(t,Y^{\epsilon}(t),\delta_{X^0(t)})\dif t
+\sqrt{\epsilon}\sigma_\epsilon(t,Y^{\epsilon}(t),\delta_{X^0(t)})\dif W(t),\ \ \ t\in[0,T],
\end{equation}
where $X^0$ is given in \eqref{MVSDE-ODE}, and then proved the exponential equivalence of $X^{\epsilon}$ and $Y^{\epsilon}$. However stronger conditions on the coefficients are required in this approach.

\vskip 0.5cm
 One of the main contributions in this article  is the identification of  the correct controlled equations for MVSDEs when perturbing the driving BM and PRM.  This is  the key for us to fully use the  weak convergence method to establish LDPs and MPDs, which leads to the very natural Lipschitz conditions on the coefficients without the extra assumptions appearing in the literatures mentioned above.  The  discretization and approximation techniques can not be applied to the case of L\'evy driving noise and also require stronger conditions on the coefficients even in the Gaussian case.
%Another contribution is to give a rigorous proof/description of the resulting MVSDES when the driving Gaussian noise and PRM are shifted. The proof is involved %and long. It is put in the appendix.

\vskip 0.4cm
The paper is organized as follows. In Section 2, we introduce notations, cylindrical Brownian motion, Poisson random measures as well as various  associated spaces. In Section 3, we introduce the setup for McKean-Vlasov stochastic differential equations on Banach spaces. The framework is sufficiently general to include also SPDEs. In Section 4, we formulate two abstract results on large deviation and moderate deviation principles for the MVSDEs. In Section 5, we consider MVSDES in $\mathbb{R}^d$. We obtain a large deviation principle and a moderate deviation principle for the MVSDEs under much weaker conditions than the one existing in the literature.

\section{Framework}
Set $\mathbb{N}:=\{1,2,3,\cdots\}$, $\mathbb{R}:=(-\infty,\infty)$ and $\mathbb{R}_+:=[0,\infty)$. For a metric space $S$, the Borel $\sigma$-field on $S$ will be written as
$\mathcal{B}(S)$. We denote by $C_c(S)$ the space of real-valued continuous functions with compact supports. Let $C([0,T],S)$ be the space of continuous functions $f:[0,T]\rightarrow S$ endowed with the uniform convergence topology. Let $D([0,T],S)$ be the space of all c\`adl\`ag functions $f:[0,T]\rightarrow S$ endowed with the Skorokhod topology. For an S-valued measurable map
$X$ defined on some probability space $(\Omega,\mathcal{F},P)$ we will denote the measure induced by $X$ on $(S,\mathcal{B}(S))$ by $Law(X)$.
For a measurable space $(U,\mathcal{U})$, let $Pr(U)$ be the class of probability measures on this space. We use the symbol $``\Rightarrow"$
to denote the convergence in distribution.

For a locally compact Polish space $S$, the space of all Borel measures on $S$ is denoted by $M(S)$, and $M_{FC}(S)$ denotes the set of all  $\mu\in M(S)$ with $\mu(O)<\infty$
for each compact subset $O\subseteq S$. We endow $M_{FC}(S)$ with the weakest topology such that for each $f\in C_c(S)$ the mapping
$\mu\in M_{FC}(S)\rightarrow \int_Sf(s)\mu(\dif s)$ is continuous. This topology is metrizable such that $M_{FC}(S)$ is a Polish space, see \cite{Budhiraja-Dupuis-Maroulas.} for more details.

We fix $T>0$ throughout this paper. Let $K$ be a separable Hilbert space with inner product $\langle\cdot,\cdot\rangle_K$  and norm $\|\cdot\|_K$. Assume that $Z$ is a locally compact Polish space with a $\sigma$-finite measure $\nu\in M_{FC}(Z)$. The  probability space $(\Omega, \cF, {\mathbb F}:=\{\cF_t\}_{t\in [0,T]},P)$ is specified as
\begin{align*}
  \Omega:=C\big([0,T],K\big)\times M_{FC}\big([0,T]\times Z
  \times \mathbb{R}_+\big),\qquad \cF:=\cB(\Omega).
\end{align*}
We introduce the coordinate mappings
\begin{align*}
&W\colon \Omega \rightarrow C\big([0,T],K\big),
\qquad W(\alpha,\beta)(t)=\alpha(t),\ t\in[0,T];\\
& N\colon \Omega \rightarrow M_{FC}\big([0,T]\times Z\times \mathbb{R}_+\big),\qquad  N(\alpha,\beta)=\beta.
\end{align*}
Define for each $t\in [0,T]$ the $\sigma$-algebra
\begin{align*}
\mathcal{G}_{t}:=\sigma\left(\left\{\big(W(s), \, N((0,s]\times A)\big):\,
0\leq s\leq t,\,A\in \mathcal{B}\big(Z\times \mathbb{R}_+\big)\right\}\right).
\end{align*}
For the given $\nu$, it follows from \cite[Sec.I.8]{Ikeda-Watanabe} that there exists a unique probability measure $P$
 on $(\Omega,\mathcal{F})$ such that:
\begin{enumerate}
\item[(a)] $W$ is a $K$-cylindrical BM;
\item[(b)] $N$ is a PRM on $[0,T]\times Z\times\mathbb{R}_+$ with intensity measure $\text{Leb}_T\otimes\nu\otimes \text{Leb}_\infty$, where
$\text{Leb}_T$ and $\text{Leb}_\infty$ stand for the Lebesgue measures on $[0,T]$ and $\mathbb{R}_+$ respectively;
\item[(c)] $W$ and $N$ are independent.
\end{enumerate}
We denote by $\mathbb{F}:=\{{\mathcal{F}}_{t}\}_{t\in[0,T]}$ the
$P$-completion of $\{\mathcal{G}_{t}\}_{t\in[0,T]}$ and
$\mathcal P$ the $\mathbb{F}$-predictable $\sigma$-field
on $[0,T]\times \Omega$. The cylindrical BM $W$ and the PRM $N$
will be defined on  the (filtered) probability space $(\Omega, \cF, \mathbb{F}:=\{\cF_t\}_{t\in [0,T]},P)$.
%; see e.g. \cite[Appendix]{BPZ}.
The corresponding compensated PRM will be denoted by $\widetilde{N}$.

Denote
\begin{align*}
{\mathcal R_+}
=\left\{\varphi\colon [0,T]\times \Omega\times Z\to \mathbb{R}_+: \varphi\ \text{is}
\, (\mathcal{P}\otimes\mathcal{B}(Z))/\mathcal{B}(\mathbb{R}_+)\text{-measurable}\right\}.
\end{align*}
For any $\varphi\in{\mathcal R_+}$, $N^{\varphi}:\Omega\rightarrow M_{FC}([0,T]\times Z)$ is  a
counting process  on $[0,T]\times {Z}$ defined by
   \begin{align}\label{Jump-representation}
      N^\varphi((0,t]\times A)=\int_{(0,t]\times A\times \mathbb{R}_+}1_{[0,\varphi(s,z)]}(r)\, N(\dif s, \dif z, \dif r),\ 0\le t\le T,\ A\in\mathcal{B}(Z).
   \end{align}
$N^\varphi$ can be viewed as a controlled random measure, with $\varphi$ selecting the intensity.
%Let us observe that
%$$ N^\varphi:\Omega\rightarrow M_{FC}([0,T]\times Z). $$

Analogously, $\widetilde{N}^\varphi$ is defined by replacing $N$ with $\widetilde{N}$ in (\ref{Jump-representation}). When $\varphi\equiv c>0$, we write $N^\varphi=N^c$ and $\widetilde{N}^\varphi=\widetilde{N}^c$.
%\begin{align}\label{Eq Def Nphi}
 %     \widetilde{N}^\varphi((0,t]\times A):=\int_{(0,t]\times A\times \mathbb{R}_+}1_{[0,\varphi(s,z,\cdot)]}(r)\, \widetilde{N}(\dif s, \dif z, \dif r).
  % \end{align}

%The following results are taken from \cite{BPZ}(see Propositions 3.1 and 3.2 there). %{\red if these two props are used later or not?if not, then delete them.}
%\begin{proposition}\label{Prop 1}
%For any constant $c>0$, $N^c$
%$$ N^c:\Omega\rightarrow M_{FC}([0,T]\times Z) $$
%is a PRM on $[0,T]\times Z$ with intensity measure $c Leb_T\otimes \nu$,
%and $\widetilde{N^c}$ is just the corresponding compensated PRM.
%\end{proposition}

%\begin{proposition}\label{Prop 2}
%For any functions $\psi,\varphi\in\mathcal{R}_+$, if there exists some $T_0>0$ and $A\in\mathcal{B}(Z)$
%such that
%$$
%\psi(s,z,\omega)=\varphi(s,z,\omega),\ \text{for}\ (s,z,\omega)\in[0,T_0]\times A\times\Omega,
%$$
%then
%$$
%N^\psi((0,t]\times B)=N^\varphi((0,t]\times B),\ \forall t\in[0,T_0],\ \forall B\in\mathcal{B}(A).
%$$
%\end{proposition}

\vskip 0.3cm
For each $f\in L^2([0,T],K)$, we introduce the quantity
\begin{align*}
Q_{1}(f)
:=\frac{1}{2}\int_{0}^{T}\norm{f(s)}_{K}^{2}\,\dif s,
\end{align*}
and for each $m>0$, denote
\begin{align*}
 S_1^m:=\Big\{f\in L^{2}([0,T],K):\,Q_1(f)\leq m\Big\}.
\end{align*}
Equipped with the weak topology, $S_1^m$ is a compact subset of $L^2([0,T],K).$
%We will  consider $S_1^m$ endowed with this topology throughout this paper.

%By defining the function
%\begin{align*}
%\ell:[0,\infty)\rightarrow[0,\infty), \qquad
%\ell(x)=x\log x-x +1,
%\end{align*}
For each measurable function
$g\colon [0,T]\times Z\to [0,\infty)$, define
\begin{align*}%\label{L_T}
Q_2(g):=\int_{[0,T]\times Z}\ell\big(g(s,z)\big) \, \nu(\dif z)\dif s,
\end{align*}
where $\ell(x)=x\log x-x +1,\ \ell(0):=1$.
For each $m>0$, denote
\begin{align*}%\label{S_N}
     S_2^m:=\Big\{g:[0,T]\times Z\rightarrow[0,\infty):\,Q_2(g)\leq m\Big\}.
\end{align*}
Any measurable function $g\in S_2^m$ can be identified with a measure $\hat{g}\in M_{FC}([0,T]\times Z)$, defined by
   \begin{align}\label{eq.corres-func-meas}
      \hat{g}(A)=\int_A g(s,z)\, \nu(\dif z)\dif s,\ \forall A\in\mathcal{B}([0,T]\times Z).
   \end{align}
This identification induces a topology on $S_2^m$ under which $S_2^m$ is a compact space (see the Appendix of \cite{Budhiraja-Chen-Dupuis}).
%Throughout this paper, we use this topology on $S_2^m$.

Denote
\begin{eqnarray}\label{eq P22}
S:= \bigcup_{m\in\mathbb{N}}\Big\{S_1^m\times S_2^m\Big\},
\end{eqnarray}
and equip it with the usual product topology.

For any $m\in(0,\infty)$, let $\mathcal{S}_1^m$ be a space of stochastic processes on $\Omega$ defined by
%\begin{eqnarray*}
%\mathcal{L}_{2}=\Big\{\psi\colon [0,T]\times
%\bar{\mathbb{V}}\to U:\text{measurable and} \; \int_{0}^{T}\norm{\psi(s,v)}_{U}^{2}\;ds<\infty,\\
%\text{for $\bar{\mathbb{P}}^{\bar{\mathbb{V}}}$-a.a. } v\in \bar{\mathbb{V}}
%\Big\},
%\end{eqnarray*}
%and introduce for each $N\in\mN$ the space
\begin{align*}
\mathcal{S}_1^m:=\{\varphi\colon [0,T]\times
\Omega\to K:\, {\mathbb{F}}\text{-predictable and} \, \varphi(\cdot,\omega)\in S_1^m
\text{ for $P$-a.e. $\omega\in \Omega$}\}.
\end{align*}

Let $\{Z_n\}_{n\in\mN}$ be a sequence of compact sets satisfying that $Z_n\subseteq Z$
and $ Z_n \nearrow Z$.  For each $n\in\mN$, let
\begin{align}\label{Eq Rn}
     \mathcal{R}_{b,n}
= \Big\{\psi\in \mathcal{R}_+:
\psi(t,z,\omega)\in \begin{cases}
 [\tfrac{1}{n},n], &\text{if }z\in Z_n\\
\{1\}, &\text{if }z\in Z_n^c
\end{cases}
\text{ for all }(t,\omega)\in [0,T]\times \Omega
\Big\},
\end{align}
and $\mathcal{R}_{b}=\bigcup _{n=1}^\infty \mathcal{R}_{b,n}$.
For any $m\in(0,\infty)$, let $\mathcal{S}_2^m$ be a space of stochastic processes on $\Omega$ defined by
\begin{align*}
\mathcal{S}_2^m:=\{\psi\in  \mathcal{R}_{b}:\, \psi(\cdot,\cdot,\omega)\in S_2^m
\text{ for $P$-a.e. $\omega\in \Omega$}\}.
\end{align*}

\section{McKean-Vlasov SDES}
Now we are in the position to introduce the framework of distribution-dependent SDEs on Banach spaces. Our framework is sufficiently general to also include SPDEs.

Let $H$ be a separable Hilbert space with inner product $\langle\cdot,\cdot\rangle_H$  and norm $\|\cdot\|_H$. Let $V$ and $E$ be separable
Banach spaces with norms $\|\cdot\|_V$ and $\|\cdot\|_E$ such that
$$
V\subset H\subset E
$$
continuously and densely.
\vskip 0.3cm
In the setting of SPDEs, people often take $V=H_0^{1,2}(D)$, $H=L^2(D)$ and $E=H_0^{-1,2}(D)$, where $D$ is an open domain in $\mathbb{R}^d$, $H_0^{1,2}(D)$ is the Sobolev space of order one, $H_0^{-1,2}(D)$ is the dual space of $H_0^{1,2}(D)$.

%Let $D([0,T],H)$ be the space of all c\'adl\'ag functions $f:[0,T]\rightarrow H$ endowed with the Skorokhod topology.
\vskip 0.3cm
$\|\cdot\|_V$ is extended to a function on $H$ by setting $\|x\|_V:=\infty$ if $x\in H\setminus V$. Note that this extension is $\cB(H)$-measurable and lower semi-continuous. Hence, the following
path space is well defined:
$$
\mathbb{D}:=\{x\in D([0,T],H):\ \int_0^T\|x(t)\|_V\dif t<\infty\}
$$
endowed with the metric
$$
d(x_1,x_2):=\int_0^T\|x_1(t)-x_2(t)\|_V\dif t+d_T(x_1,x_2),
$$
where $d_T(x_1,x_2)$ is the Skorokhod distance on $D([0,T],H)$. It is easy to see that $(\mathbb{D},d)$ is a Polish space.

Let $\mathcal{L}_2(K,H)$ be the space of all Hilbert-Schmidt operators from $K$ to $H$ equipped with the usual Hilbert-Schmidt norm $\|\cdot\|_{\mathcal{L}_2}$. Denote by $\mathcal{B}_t(\mD)$ the $\sigma$-algebra generated by all maps $\pi_s:\mD\rightarrow H,\ s\in[0,t]$, where $\pi_s(w):=w(s),\ w\in\mD$.

%Then we have, see e.g. Proposition 7.1 Chapter 3 in \cite{EK 1986} (see Page 127),
%\begin{proposition}\label{Prop Zhai cadlag}
%$$\mathcal{B}\Big(D([0,T],H),d_T\Big)=\sigma\{\pi_s,s\in[0,T]\}.$$
%\end{proposition}

\bas\label{as1} Throughout this paper we always assume that, for any fixed $J\in Pr(D([0,T],H))$:
\begin{itemize}
\item $b(\cdot,\cdot,J): [0,T]\times \mD\rightarrow E\text{ is }\mathcal{B}([0,T])\otimes\mathcal{B}(\mD)/\mathcal{B}(E)\text{-measurable};
\label{page 5}
$
\item $\sigma(\cdot,\cdot,J): [0,T]\times \mD\rightarrow \mathcal{L}_2(K,H)\text{ is }\mathcal{B}([0,T])\otimes\mathcal{B}(\mD)/\mathcal{B}(\mathcal{L}_2(K,H))\text{-measurable};$
\item $
G(\cdot,\cdot,J,\cdot): [0,T]\times \mD\times Z\rightarrow H\text{ is }\mathcal{B}([0,T])\otimes\mathcal{B}(\mD)\otimes \mathcal{B}(Z)/\mathcal{B}(H)\text{-measurable};$

\item $b(t,\cdot,J) \text{ is }\mathcal{B}_t(\mD)/\mathcal{B}(E)\text{-measurable},$ for any $t\in[0,T]$;

\item $\sigma(t,\cdot,J)\text{ is }\mathcal{B}_t(\mD)/\mathcal{B}(\mathcal{L}_2(K,H))\text{-measurable},$ for any $t\in[0,T]$;

\item $G(t,\cdot,J,z)\text{ is }\mathcal{B}_t(\mD)/\mathcal{B}(H)\text{-measurable},$ for any $t\in[0,T]$ and $z\in Z$;

\item $\label{page 6}
G(\cdot,\cdot,J,\cdot)\text{ is predictable with respect to }\{\mathcal{B}_t(\mD),\ t\in[0,T]\}.
\footnote{See Definition 3.3 in \cite{Ikeda-Watanabe}, page 61. An $H$-valued function $f(t,x,z)$ defined on $[0,T]\times\mD\times Z$ is called $\{\mathcal{B}_t(\mD),~t\in[0,T]\}$-predictable
if the mapping $(t,x,z)\rightarrow f(t,x,z)$ is $\mathcal{I}/\mathcal{B}(H)$-measurable where $\mathcal{I}$ is the smallest $\sigma$-field
on $[0,T]\times\mD\times Z$ such that all real valued function $g$ having the following properties are measurable:

(i) for each $t>0$, $(x,z)\rightarrow g(t,x,z)$ is $\mathcal{B}_t(\mD)\otimes \mathcal{B}(Z)/\mathcal{B}(\mathbb{R})$-measurable;

(ii) for each $(x,z)$, $t\rightarrow g(t,x,z)$ is left continuous.
}\label{Page 1}
$
\end{itemize}
\eas
%Here $\mathcal{B}_t(\mD)$ denotes the $\sigma$-algebra generated by all maps $\pi_s:\mD\rightarrow H,\ s\in[0,t]$,
%where $\pi_s(w):=w(s),\ w\in\mD$.
%For $t\geq0$ and $w\in\mD$ define the stopped path $w^t$ by
%$$
%w^t(s)=w(s\wedge t), \ s\geq 0.
%$$

%\vskip 0.3cm
%Recall (\ref{Eq Def Nphi}),

%Denote $N^\varphi$ and $\widetilde{N}^\varphi$ by $N^1$ and $\widetilde{N^1}$ respectively, when $\varphi\equiv1$. Proposition \ref{Prop 1} implies that, on $(\Omega,\mathcal{F},\mathbb{F},P)$, $N^1$ is a PRM on $[0,T]\times Z$ with intensity measure $Leb_T\otimes \nu$, and $\widetilde{N^1}$ is the corresponding compensated PRM. \vskip 0.2cm

In this section, we consider the following general distribution-dependent SDEs with jumps
\begin{eqnarray}\label{eq YJ}
\dif Y(t)=b(t,Y,J)\dif t+\sigma(t,Y,J)\dif W(t)+\int_{Z}G(t,Y,J,z)\widetilde{N}^1(\dif t,\dif z),\ 0\le t\le T
\end{eqnarray}
with initial value $Y(0)=h\in H$.

\begin{remark}
 We stress that the abstract formulation of SDEs (\ref{eq YJ}) is general enough to cover many types of SPDEs, such as SPDEs with delay, distribution-dependent SPDEs, McKean-Vlasov SDEs, etc.
\end{remark}

\vskip 0.2cm

We now introduce some definitions related to the solutions of (\ref{eq YJ}).

\begin{definition}\label{def 1}
For a fixed  $J\in Pr(D([0,T],H))$, $Y$ is called a solution of (\ref{eq YJ})
if
\begin{itemize}
\item[(a)] $Y=\{Y(t),t\in[0,T]\}$ is an $\mF$-adapted process with paths in $\mD$,

  \item[(b)] $\int_0^T\|b(t,Y,J)\|_E\dif t + \int_0^T\|\sigma(t,Y,J)\|^2_{\mathcal{L}_2}\dif t + \int_0^T\!\!\int_{Z}\|G(t,Y,J,z)\|^2_H\nu(\dif z)\dif t <\infty,\ P\text{-a.s.,}$
  \item[(c)]  as a stochastic process on $E$, $Y$ satisfies
  \begin{align}\label{eq YJ 0}
  Y(t)=&h+\int_0^tb(s,Y,J)\dif s + \int_0^t\sigma(s,Y,J)\dif W(s)\cr
  &+\int_0^t\!\!\int_{Z}G(s,Y,J,z)\widetilde{N}^1(\dif z,\dif s),\ t\in[0,T],\ P\text{-a.s..}
  \end{align}
\end{itemize}
\end{definition}

\begin{remark}\label{Rem 1.4}
If (a) holds, the assumptions on $b$, $\sigma$ and $G$ imply that

\begin{itemize}
  \item $\{b(s,Y(\omega),J),\ s\in[0,T],\ \omega\in\Omega\}$ is an $E$-valued $\mathbb{F}$-adapted process, %$\mathcal{B}([0,T])\otimes\mathcal{F}$-measurable, ,
  \item $\{\sigma(s,Y(\omega),J),\ s\in[0,T],\ \omega\in\Omega\}$ is a $\mathcal{L}_2(K,H)$-valued $\mathbb{F}$-adapted process, %$\mathcal{B}([0,T])\otimes\mathcal{F}$-measurable,
  \item $\{G(s,Y(\omega),J,z),\ s\in[0,T],\ \omega\in\Omega,\ z\in Z\}$ is a $H$-valued $\mathbb{F}$-predictable process\footnote{%Similar to footnote 1 on page \pageref{Page 1}.
      An $H$-valued function $f(t,x,z)$ defined on $[0,T]\times\Omega\times Z$ is called $\mF$-predictable
if the mapping $(t,x,z)\rightarrow f(t,x,z)$ is $\mathcal{I}_{\mF}/\mathcal{B}(H)$-measurable where $\mathcal{I}_{\mF}$ is the smallest $\sigma$-field
on $[0,T]\times\Omega\times Z$ with respect to which all real valued function $g$ having the following properties are measurable:

(i) for each $t>0$, $(x,z)\rightarrow g(t,x,z)$ is $\mathcal{F}_t\otimes \mathcal{B}(Z)/\mathcal{B}(\mathbb{R})$-measurable,

(ii) for each $(x,z)$, $t\rightarrow g(t,x,z)$ is left continuous. \label{page 2}
}. %$\mathcal{B}([0,T])\otimes\mathcal{F}\otimes\mathcal{B}(Z)$-measurable,.
\end{itemize}
Furthermore, if (b) holds, then $\int_0^tb(s,Y,J)\dif s$, $\int_0^t\sigma(s,Y,J)\dif W(s)$ and $\int_0^t\!\!\int_{Z}G(s,Y,J,z)\widetilde{N}^1(\dif s,\dif z)$ are well-defined as the Lebesgue-Stieltjes integral and the $It\hat{o}$ integrals respectively. The reader is referred to \cite{RSZ 2008} and \cite{HY Zhao} for more details.
\end{remark}

\begin{definition}\label{def 2}
The pathwise uniqueness is said to hold for (\ref{eq YJ}) with the fixed $J\in Pr(D([0,T],H))$, if for any two solutions  $Y_1$ and $Y_2$ of (\ref{eq YJ}),
$$
Y_1(t)=Y_2(t),\ t\in[0,T],\ P\text{-a.s..}
$$
\end{definition}
\vskip 0.4cm
Now consider the McKean-Vlasov equation:
\begin{align}\label{eq X}
\dif X(t)=&b(t,X,Law(X))\dif t+\sigma(t,X,Law(X))\dif W(t)\nonumber\\
&+\int_{Z}G(t,X,Law(X),z)\widetilde{N}^1(\dif z,\dif t)
\end{align}
with initial value $X(0)=h\in H$.

Notice that a stochastic process $X=\{X(t)\}_{0\le t\le T}$ is a solution to equation (\ref{eq X}) if
$X$ is a solution of (\ref{eq YJ}) with $J=Law(X)$.
%Here and in the following, we denote the law of $X$ by $Law(X)$.

\begin{definition}\label{def 4}
The pathwise uniqueness is said to hold for (\ref{eq X}), if for any two solutions $X_1$ and $X_2$ of (\ref{eq X}),
$$
X_1(t)=X_2(t),\ t\in[0,T]\ P\text{-a.s.,}
$$
and hence $Law(X_1)=Law(X_2)$.
\end{definition}
\vskip 0.5cm

Now we state a perturbation result.
\begin{theorem}\label{Thm 01}
Fix  $J\in Pr(D([0,T],H))$. Suppose that $Y$ is a solution of (\ref{eq YJ}), and pathwise uniqueness holds with the fixed $J$. Then there exists a unique map
$\Gamma_J:C([0,T],K)\times M_{FC}\big([0,T]\times Z\big)\rightarrow\mD$ such that
$$
Y=\Gamma_J(W,N^1).
$$
Moreover for any $m\in(0,\infty)$ and $u=(\phi,\psi)\in \mathcal{S}_1^m\times\mathcal{S}_2^m$, letting
\begin{eqnarray}\label{eq def Yu}
Y^u:=\Gamma_J(W+\int_0^\cdot\phi(s)\dif s, N^\psi),
\end{eqnarray}
we have
\begin{itemize}
\item[(a)] $Y^u=\{Y^u(t),t\in[0,T]\}$ is an $\mF$-adapted process with paths in $\mD$,
  \item[(b)] \begin{align*}
  &\int_0^T\|b(t,Y^u,J)\|_E\dif t + \int_0^T\|\sigma(t,Y^u,J)\|^2_{\mathcal{L}_2}\dif t + \int_0^T\|\sigma(t,Y^u,J)\phi(t)\|_H\dif t\\
  &+ \int_0^T\!\!\int_{Z}\|G(t,Y^u,J,z)\|^2_H\psi(t,z)\nu(\dif z)\dif t
  + \int_0^T\!\!\int_{Z}\|G(t,Y^u,J,z)(\psi(t,z)-1)\|_H\nu(\dif z)\dif t\\
  &<\infty,\ P\text{-a.s.,}
  \end{align*}
  \item[(c)]  as a stochastic equation on $E$, $Y^u$ satisfies
  \begin{align}\label{eq YJ u}
  Y^u(t)
  =&h+\int_0^tb(s,Y^u,J)\dif s + \int_0^t\sigma(s,Y^u,J)\dif W(s)+\int_0^t\sigma(s,Y^u,J)\phi(s)\dif s \cr
  &+ \int_0^t\!\!\int_{Z}G(s,Y^u,J,z)\Big(N^\psi(\dif z,\dif s)-\nu(\dif z)\dif s\Big),\ t\in[0,T],\ P\text{-a.s..}
  \end{align}
\end{itemize}

Moreover, $Y^u$ is the unique stochastic process satisfying (a)-(c).
\end{theorem}

%\vskip 0.2cm
%We give some remarks before proving this theorem.
\begin{remark}\label{Rem 1.8}
Note that \eqref{eq YJ u} is equivalent to
  \begin{align*}
  Y^u(t)=
  &h+\int_0^tb(s,Y^u,J)\dif s + \int_0^t\sigma(s,Y^u,J)\dif W(s) \cr
  &+\int_0^t\sigma(s,Y^u,J)\phi(s)
  +\int_0^t\!\!\int_{Z}G(s,Y^u,J,z)\widetilde{N}^\psi(\dif z,\dif s)\cr
  &+\int_0^t\!\!\int_{Z}G(s,Y^u,J,z)\Big(\psi(s,z)-1\Big)\nu(\dif z)\dif s,\ t\in[0,T],\ P\text{-a.s..}
  \end{align*}
  where
\begin{align}\label{eq 1.9}
\int_0^t\!\!\int_{Z}G(s,Y^u,J,z)\widetilde{N}^\psi(\dif z,\dif s)
=\int_0^t\!\!\int_{Z}\!\int_{[0,\infty)}G(s,Y^u,J,z)1_{[0,\psi(s,z)]}(r)\widetilde{N}(\dif z,\dif s,\dif r).
\end{align}
\end{remark}
\vskip 0.4cm
Although the claim in Theorem \ref{Thm 01} is not surprising, its rigorous proof requires the careful use of the Girsanov Theorem for the mixture of Brownian motion and Poisson random measures. We refer the reader to \cite{[Zhai]} for details.
%A related result  for stochastic evolution equation driven by  Gaussian noise can be found in \cite{DFPR 2013}. and \cite{BPZ} for 2-D stochastic
 %Navier-Stokes equation with pure jump.%We refer the readers to {\color{red}[Zhai]} for more details.
\vskip 0.3cm
Applying Theorem \ref{Thm 01} and Remark \ref{Rem 1.8}, we immediately have
\begin{theorem}\label{Thm 02}
Assume  that $X$  is a solution of (\ref{eq X}) with initial value $X(0)=h\in H$, and that the pathwise uniqueness holds for (\ref{eq YJ}) with  $J=Law(X)$.
%, where $Law(X)$
%is the probability measure $P\circ X^{-1}$ on $D([0,T],H)$.
Then
$X=\Gamma_{Law(X)}(W,N^1)$, where $\Gamma_{Law(X)}$ is the map $\Gamma_J$ in Theorem \ref{Thm 01} with $J=Law(X)$.

Moreover for any $m\in(0,\infty)$ and $u=(\phi,\psi)\in \mathcal{S}_1^m\times\mathcal{S}_2^m$, let
$X^u:=\Gamma_{Law(X)}(W+\int_0^\cdot\phi(s)\dif s, N^\psi)$, then we have
\begin{itemize}
  \item[(a)] $X^u=\{X^u(t),t\in[0,T]\}$ is an $\mF$-adapted process with paths in $\mD$,
  \item[(b)]  as a stochastic equation on $E$, $X^u$ satisfies
  \begin{align}\label{eq YJ u00}
  X^u(t)
  =&
  h+\int_0^tb(s,X^u,Law(X))\dif s + \int_0^t\sigma(s,X^u,Law(X))\dif W(s)\cr
  &+\int_0^t\sigma(s,X^u,Law(X))\phi(s)\dif s
  +\int_0^t\!\!\int_{Z}G(s,X^u,Law(X),z)\widetilde{N}^\psi(\dif z,\dif s)\cr
  &+\int_0^t\!\!\int_{Z}G(s,X^u,Law(X),z)\Big(\psi(s,z)-1\Big)\nu(\dif z)\dif s
  ,\ t\in[0,T],\ P\text{-a.s..}
  \end{align}
\end{itemize}
\end{theorem}
\section{Large and moderate deviation principles}
In this section, we will consider the large and moderate deviation principles  of the solutions:
\begin{align}\label{eq X0}
\dif X^\epsilon(t)
=&b_\epsilon(t,X^\epsilon,Law(X^\epsilon))\dif t
+\sqrt{\epsilon}\sigma_\epsilon(t,X^\epsilon,Law(X^\epsilon))\dif W(t)\nonumber\\
&+\epsilon\int_{Z}G_\epsilon(t,X^\epsilon,Law(X^\epsilon),z)\widetilde{N}^{\epsilon^{-1}}
(\dif z,\dif t)
\end{align}
with initial value $X^\epsilon(0)=h\in H$, as $\epsilon \downarrow 0$.
%\vskip 0.2cm
%In the following, we divide into two parts to state the applications on LDP and MDP respctively.
%\vskip 0.2cm
%\begin{center}
%  \textbf{Large deviation principle}
%\end{center}
\subsection{Large deviation principle}
Let us first recall the definition of a rate function and LDP.

Let $\mathcal{E}$ be a Polish space with the Borel $\sigma$-field $\mathcal{B}(\mathcal{E}).$ Recall

\begin{definition}[Rate function] A function $I: \mathcal{E} \rightarrow[0, \infty]$ is called a rate function on $\mathcal{E},$ if for each $M<\infty,$ the level set $\{x \in \mathcal{E}: I(x) \leq M\}$ is a compact subset of $\mathcal{E}$.
\end{definition}

\begin{definition}[Large deviation principle] Let I be a rate function on $\mathcal{E}$. Given a collection $\{\hbar(\epsilon)\}_{\epsilon>0}$ of
positive reals, a family $\left\{\mathbb{X}^{\epsilon}\right\}_{\epsilon>0}$ of $\mathcal{E}$-valued random elements is said to satisfy a LDP on $\mathcal{E}$ with speed $\hbar(\epsilon)$ and rate function $I$ if the following two claims hold.
\begin{itemize}
  \item[(a)]
 (Upper bound) For each closed subset $C$ of $\mathcal{E},$
$$
\limsup_{\epsilon \rightarrow 0} \hbar(\epsilon) \log {P}\left(\mathbb{X}^{\epsilon} \in C\right) \leq - \inf_{x \in C} I(x).
$$
 \item[(b)] (Lower bound) For each open subset $O$ of $\mathcal{E}$
$$
\liminf_{\epsilon \rightarrow 0} \hbar(\epsilon) \log {P}\left(\mathbb{X}^{\epsilon} \in O\right) \geq - \inf_{x \in O} I(x).
$$
\end{itemize}
\end{definition}

%In the following of this section, we
Introduce the hypothesis:
\begin{itemize}
  \item[(S0)] For any fixed $\epsilon>0$ and $J\in Pr(D([0,T],H))$, the maps $b_\epsilon(\cdot,\cdot, J):[0,T]\times\mD\rightarrow E$,
 $\sigma_\epsilon(\cdot,\cdot, J):[0,T]\times\mD\rightarrow \mathcal{L}_2(K,H)$ and $G_\epsilon(\cdot,\cdot, J,\cdot):[0,T]\times\mD\times Z\rightarrow H$ satisfy the Assumption \ref{as1};%  same assumptions on $b,\sigma,G$ respectively stated in Pages \pageref{page 5} and \pageref{page 6};
  \item[(S1)] (\ref{eq X0}) has a unique solution $X^\epsilon$ as stated in Definition \ref{def 4};
  \item[(S2)] Pathwise uniqueness holds for the following SDE with the fixed $J$ replaced by $Law(X^\epsilon)$ as stated in Definition \ref{def 2},
\begin{align}\label{eq X0 1}
\dif Y^\epsilon(t)
=b_\epsilon(t,Y^\epsilon,J)\dif t+\sqrt{\epsilon}\sigma_\epsilon(t,Y^\epsilon,J)\dif W(t)
+\epsilon\int_{Z}G_\epsilon(t,Y^\epsilon,J,z)\widetilde{N}^{\epsilon^{-1}}(\dif z,\dif t)
\end{align}
with initial value $Y^\epsilon(0)=h\in H$.
\end{itemize}

\vskip 0.5cm

Theorem \ref{Thm 02} states
that there exists a map $\Gamma^\epsilon_{Law(X^\epsilon)}$ such that $X^\epsilon=\Gamma^\epsilon_{Law(X^\epsilon)}(\sqrt{\epsilon}W(\cdot),\epsilon N^{\epsilon^{-1}})$. Moreover, for any $m\in(0,\infty)$, $u_\epsilon=(\phi_\epsilon,\psi_\epsilon)\in \mathcal{S}_1^m\times\mathcal{S}_2^m$, let
\begin{equation}\label{sol-control}
Z^{u_\epsilon}:=\Gamma^\epsilon_{Law(X^\epsilon)}(\sqrt{\epsilon}W(\cdot)
+\int_0^\cdot\phi_\epsilon(s)\dif s, \epsilon N^{\epsilon^{-1}\psi_\epsilon}),
\end{equation}
then $Z^{u_\epsilon}$ is the unique solution of the equation:
\begin{align}
%\label{eq X0 2}
  &Z^{u_\epsilon}(t)\nonumber\\
%  =&
%  h+\int_0^tb_\epsilon(s,Z^{u_\epsilon},Law(X^\epsilon))\dif s + \sqrt{\epsilon}\int_0^t\sigma_\epsilon(s,Z^{u_\epsilon},Law(X^\epsilon))\dif W(s)
%  +\int_0^t\sigma_\epsilon(s,Z^{u_\epsilon},Law(X^\epsilon))\phi_\epsilon(s)\dif s \nonumber\\
%  &+\epsilon\int_0^t\!\!\int_{Z}G_\epsilon(s,Z^{u_\epsilon},Law(X^\epsilon),z)
%  \Big(N^{\epsilon^{-1}\psi_\epsilon}(\dif z,\dif s)-\epsilon^{-1}\nu(\dif z)\dif s\Big),\ t\in[0,T],\ P\text{-a.s.}\\
  =&
  h+\int_0^tb_\epsilon(s,Z^{u_\epsilon},Law(X^\epsilon))\dif s + \sqrt{\epsilon}\int_0^t\sigma_\epsilon(s,Z^{u_\epsilon},Law(X^\epsilon))\dif W(s)\nonumber\\
  &+\int_0^t\sigma_\epsilon(s,Z^{u_\epsilon},Law(X^\epsilon))\phi_\epsilon(s)\dif s
  +\epsilon\int_0^t\!\!\int_{Z}G_\epsilon(s,Z^{u_\epsilon},Law(X^\epsilon),z)
  \widetilde{N}^{\epsilon^{-1}\psi_\epsilon}(\dif z,\dif s)\nonumber\\
  &+\int_0^t\!\!\int_{Z}G_\epsilon(s,Z^{u_\epsilon},Law(X^\epsilon),z)
  \Big(\psi_\epsilon(s,z)-1\Big)\nu(\dif z)\dif s,\ t\in[0,T],\ P\text{-a.s..}
  \label{eq X0 2-second}
  \end{align}
\vskip 0.5cm

The next  result is a re-formulation of  Theorem 2.4 in \cite{Budhiraja-Chen-Dupuis}, Theorem 4.2 in \cite{Budhiraja-Dupuis-Maroulas.} in the current setting.
\begin{theorem}\label{Thm zhunze 1}
Assume that (S0)-(S2) hold.
Suppose that there exists a measurable map $\Gamma^0:S\rightarrow\mD$
{
\footnote{
The existence of $\Gamma^0$ implies that there exists a measurable map $\widetilde{\Gamma}^0:C([0,T],K)\times M_{FC}\big([0,T]\times Z\big)\rightarrow\mD$ such that, for any $u=(\phi,\psi)\in S$,
$$
\widetilde{\Gamma}^0(\int_0^\cdot\phi(s)\dif s,\hat{\psi}):={\Gamma}^0(u).
$$
The definition of $\hat{\psi}$ can be found in (\ref{eq.corres-func-meas}).
 }
 }
 such that

(a) For any $m\in(0,\infty)$ and any family $\{u_\epsilon=(\phi_\epsilon,\psi_\epsilon);~ \epsilon>0\}\subset\mathcal{S}_1^m\times\mathcal{S}_2^m$
satisfying that $u_\epsilon$ converges in law as ${S}_1^m\times {S}_2^m$-valued random elements to some element $u=(\phi,\psi)$ as
$\epsilon\rightarrow 0$,  $Z^{u_\epsilon}$ converges in law to $\Gamma^0(\phi,\psi)$.
%Here $(Z^{u_\epsilon},Law(X^\epsilon))$ is the solution of (\ref{eq X0 2}) as stated in Definition \ref{def 1}.

(b) For every $m\in(0,\infty)$, the set
$$
 \Big\{\Gamma^0(\phi,\psi); (\phi,\psi)\in S^m_1\times S^m_2\Big\}
$$
is a compact subset of $\mD$.

Then the family $\{X^\epsilon\}_{\epsilon>0}$ satisfies a LDP in $\mD$ with speed $\epsilon$ and the rate function $I$
given by
\begin{eqnarray}\label{eq rate function 0}
I(g):=\inf_{(\phi,\psi)\in S,\ g=\Gamma^0(\phi,\psi)}\{Q_1(\phi)+Q_2(\psi)\},\ g\in \mD,
\end{eqnarray}
with the convention $\inf\{\emptyset\}=\infty$.
%Here $X^\epsilon$ is the solution of (\ref{eq X0}) as stated in Definition \ref{def 3}.
\end{theorem}

The next theorem provides a convenient, sufficient condition for verifying the assumptions in Theorem \ref{Thm zhunze 1}.
\begin{theorem}\label{Thm zhunze 2}
Assume that (S0)-(S2) hold. Suppose that there exists a measurable map $\Gamma^0:S\rightarrow\mD$ such that

(a) For any $m\in(0,\infty)$, any family $\{u_\epsilon=(\phi_\epsilon,\psi_\epsilon);~ \epsilon>0\}\subset \mathcal{S}_1^m\times\mathcal{S}_2^m$, and any
$\delta>0$,
$$
\lim_{\epsilon\rightarrow 0}P\Big(d(Z^{u_\epsilon},\Gamma^0(\phi_\epsilon,\psi_\epsilon))>\delta\Big)=0.
$$
% Here $(Z^{u_\epsilon},Law(X^\epsilon))$ is the solution of (\ref{eq X0 2}) as stated in Definition \ref{def 1}.

(b) For any $m\in(0,\infty)$ and any family $\{(\phi_n,\psi_n)\in S^m_1\times S^m_2,~n\in\mathbb{N}\}$ satisfying that
$(\phi_n,\psi_n)$ converges to some element $(\phi,\psi)$ in $S^m_1\times S^m_2$ as $n\rightarrow \infty$,
$\Gamma^0(\phi_n,\psi_n)$ converges to $\Gamma^0(\phi,\psi)$ in the space $\mD$.

Then the family $\{X^\epsilon\}_{\epsilon>0}$ satisfies a LDP in $\mD$ with speed $\epsilon$ and the rate function $I$
given by (\ref{eq rate function 0}).
%\begin{eqnarray}\label{eq rate function 00}
%I(g):=\inf_{(\phi,\psi)\in S,\ g=\Gamma^0(\phi,\psi)}\{Q_1(\phi)+Q_2(\psi)\},\ g\in \mD,
%\end{eqnarray}
%with the convention $\inf\{\emptyset\}=\infty$.
%Here $X^\epsilon$ is the solution of (\ref{eq X0}) as stated in Definition \ref{def 3}.
\end{theorem}
\vskip 0.4cm
The proof of Theorem \ref{Thm zhunze 2} is very similar to that of Theorem 3.2 in \cite{MSZ}, so we omit it here.

\subsection{Moderate deviation principle}
%\vskip 0.2cm
%\begin{center}
%  \textbf{Moderate deviation principle}
%\end{center}

Theorem \ref{Thm 02} can also be applied to establish a MDP of the solution $X^\epsilon$ to (\ref{eq X0}) as $\epsilon$ decreases to 0.
%\begin{eqnarray}\label{eq X0 0}
%dX^\epsilon(t)=b_\epsilon(t,X^\epsilon,Law(X^\epsilon))dt+\sqrt{\epsilon}\sigma_\epsilon(t,X^\epsilon,Law(X^\epsilon))dW(t)+\epsilon\int_{Z}G_\epsilon(t,X^\epsilon,Law(X^\epsilon),z)\widetilde{N}^{\epsilon^{-1}}(dz,dt)
%\end{eqnarray}
%with initial value $X^\epsilon(0)=h\in H$.
%\vskip 0.3cm

%Before stating the main result in this part, we need the following assumptions and notations.
\vskip 0.2cm

Introduce the following conditions.
\begin{itemize}
  \item[(S3)] As the parameter $\epsilon$ tends to zero, the solution $X^\epsilon$ of (\ref{eq X0}) will converge in probability (in a suitable path space) to $X^0$ given as the solution of the following
deterministic equation
\begin{eqnarray}\label{eq X0 1}
\dif X^0(t)=b_0(t,X^0,Law(X^0))\dif t
\end{eqnarray}
with initial value $X^0(0)=h\in H$.
  \item[(S4)] (\ref{eq X0 1}) has a unique solution $X^0=\{X^0(t),\ t\in[0,T]\}$.
\end{itemize}

\vskip 0.3cm
Assume $a(\epsilon)>0,~\epsilon>0$ satisfies
\begin{eqnarray}\label{eq a ep}
a(\epsilon)\rightarrow 0,\ \ \ \epsilon/a^2(\epsilon)\rightarrow 0,\ \ \ as\ \epsilon\rightarrow 0.
\end{eqnarray}

Let
\begin{eqnarray*}
M^\epsilon(t)=\frac{1}{a(\epsilon)}(X^\epsilon(t)-X^0(t)),\ t\in[0,T].
\end{eqnarray*}
Then
$M^\epsilon=\{M^\epsilon(t),\ t\in[0,T]\}$ satisfies the following SDEs
\begin{align}\label{Eq MDP 0}
\dif M^\epsilon(t)
=&\frac{1}{a(\epsilon)}\left(b_\epsilon(t,a(\epsilon)M^\epsilon+X^0,Law(X^\epsilon))
-b_0(t,X^0,Law(X^0))\right)\dif t\cr
&+\frac{\sqrt{\epsilon}}{a(\epsilon)}\sigma_\epsilon(t,a(\epsilon)M^\epsilon+X^0,
Law(X^\epsilon))\dif W(t)\cr
&+\frac{\epsilon}{a(\epsilon)}\int_{Z}G_\epsilon(t,a(\epsilon)M^\epsilon+X^0,Law(X^\epsilon),z)
\widetilde{N}^{\epsilon^{-1}}(\dif z,\dif t),
\end{align}
with $M^\epsilon(0)=0$.

Denote
$$
\mathcal{R}:=\{\varphi:[0,T]\times Z\times\Omega\rightarrow \mathbb{R};\varphi\ \text{is}\  (\mathcal{P}\otimes\mathcal{B}(Z))/\mathcal{B}(\mathbb{R})\text{-measurable}\}.
$$
For any given $\epsilon>0$ and $m\in(0,\infty)$, denote
$$\aligned
S^m_{+,\epsilon} &:=\{g:[0,T]\times Z\rightarrow[0,\infty)\,|\,\ Q_2(g)\leq ma^2(\epsilon)\},\\
S^m_\epsilon &:=\{\varphi:[0,T]\times Z\rightarrow \mathbb{R}\,|\,\varphi=(g-1)/a(\epsilon),\ g\in S^m_{+,\epsilon}\},\\
\mathcal{S}^m_{+,\epsilon} &:=\{g\in\mathcal{R}_b\,|\, g(\cdot,\cdot,\omega)\in S^m_{+,\epsilon},\text{ for }P\text{-a.e. }\omega\in\Omega\},\\
\mathcal{S}^m_\epsilon &:=\{\varphi\in\mathcal{R}\,|\,\varphi(\cdot,\cdot,\omega)\in S^m_\epsilon,\text{ for }P\text{-a.e. }\omega\in\Omega\}.
\endaligned$$

Denote $L^2(\nu_T)$ the space of all $\mathcal{B}([0,T])\otimes\mathcal{B}(Z)/ \mathcal{B}(\mathbb{R})$ measurable functions $f$ satisfying that
$$\|f\|^2_2:=\int_0^T\!\!\int_Z|f(s,z)|^2\nu(\dif z)\dif s < +\infty.$$
%Let
%$$L^2(\nu_T):=
%\left\{f:[0,T]\times Z\rightarrow\mathbb{R}
 %  ;~\mathcal{B}([0,T])\otimes\mathcal{B}(Z)\setminus \mathcal{B}(\mathbb{R})\text{-measurable and }\int_0^T\!\!\int_Z|f(s,z)|^2\nu(\dif z)\dif s<\infty\right\},
%$$
%and
%$$\|f\|^2_2:=\int_0^T\!\!\int_Z|f(s,z)|^2\nu(\dif z)\dif s<+\infty.$$
Then $(L^2(\nu_T),\|\cdot\|_2)$ is a Hilbert space. Denote by $B_2(r)$  the ball of radius $r$ centered at $0$ in $L^2(\nu_T)$.
Throughout this paper, $B_2(r)$ is equipped with the weak topology of $L^2(\nu_T)$ and therefore compact.

Suppose $g\in S^m_{+,\epsilon}$. By Lemma 3.2 in \cite{BDG}, there exists a constant $\kappa_2(1)>0$ (independent of $\epsilon$) such that
$\varphi1_{\{|\varphi|\leq1/a(\epsilon)\}}\in B_2(\sqrt{m\kappa_2(1)})$, where $\varphi=(g-1)/a(\epsilon)$.

\vskip 0.2cm
Assume that (S0)--(S2) hold. Recall the map $\Gamma^\epsilon_{Law(X^\epsilon)}$ defined in (4.3). Set
$$
\Upsilon^\epsilon_{Law(X^\epsilon)}\left(\cdot,\cdot\right)
   :=
\frac{1}{a(\epsilon)}\left(\Gamma^\epsilon_{Law(X^\epsilon)}\left(\cdot,\cdot \right)-X^0\right).
$$
Then, by the property of $\Gamma^\epsilon_{Law(X^\epsilon)}$, we have
\begin{itemize}
  \item[(a)] $\Upsilon^\epsilon_{Law(X^\epsilon)}$ is a measurable map from $C([0,T],K)\times M_{FC}\big([0,T]\times Z\big)\rightarrow\mD$ such that
      $$
        M^\epsilon=\Upsilon^\epsilon_{Law(X^\epsilon)}\left(\sqrt{\epsilon}W(\cdot),\epsilon N^{\epsilon^{-1}}\right).
      $$
  \item [(b)] for any $m\in(0,\infty)$, $u_\epsilon=(\phi_\epsilon,\psi_\epsilon)\in \mathcal{S}^m_1\times\mathcal{S}^m_{+,\epsilon}$, let
  $$
    M^{u_\epsilon}
  :=
    \Upsilon^\epsilon_{Law(X^\epsilon)}
       \left(\sqrt{\epsilon}W(\cdot)+a(\epsilon)\int_0^\cdot\phi_\epsilon(s)ds,\epsilon N^{\epsilon^{-1}\psi_\epsilon}\right),
  $$
  then $M^{u_\epsilon}$ is the unique solution of the following SDE: for each $t\in[0,T]$,
  \begin{align}\label{eq MDP 1-second}
  %\label{eq MDP 1}
    M^{u_\epsilon}(t)
   =&\frac{1}{a(\epsilon)}\int_0^t\left(b_\epsilon(s,a(\epsilon)M^{u_\epsilon}
   +X^0,Law(X^\epsilon))-b_0(s,X^0,Law(X^0))\right)\dif s\cr
   &+\frac{\sqrt{\epsilon}}{a(\epsilon)}\int_0^t\sigma_\epsilon(s,a(\epsilon)
   M^{u_\epsilon}+X^0,Law(X^\epsilon))\dif W(s)\cr
   &+\int_0^t\sigma_\epsilon(s,a(\epsilon)M^{u_\epsilon}+X^0,Law(X^\epsilon))
   \phi_\epsilon(s)\dif s\\
  &+\frac{\epsilon}{a(\epsilon)}\int_0^t\!\!\int_{Z}G_\epsilon(s,a(\epsilon)M^{u_\epsilon}+X^0,
    Law(X^\epsilon),z)\widetilde{N}^{\epsilon^{-1}\psi_\epsilon}(\dif z,\dif s)\cr
  &+\frac{1}{a(\epsilon)}\int_0^t\!\!\int_{Z}G_\epsilon(s,a(\epsilon)M^{u_\epsilon}+X^0,Law(X^\epsilon),z)
   \Big(\psi_\epsilon(s,z)-1\Big)\nu(\dif z)\dif s.\nonumber
  \end{align}

%And we do not know the uniqueness of the solution to (\ref{eq MDP 1}) or (\ref{eq MDP 1-second}).
\end{itemize}

%\vskip 0.2cm
%\begin{remark}\label{rem 4}
%Similar to Remark \ref{rem 3}, $
%        M^\epsilon=\Upsilon^\epsilon_{Law(X^\epsilon)}\Big(\sqrt{\epsilon}W(\cdot),\epsilon N^{\epsilon^{-1}}\Big)
%      $ is a solution to (\ref{eq MDP 1}) or (\ref{eq MDP 1-second}),  but the uniqueness to (\ref{eq MDP 1}) or (\ref{eq MDP 1-second}) is not clear. Hence, in the following  Theorems \ref{TH MDP 1} and \ref{TH MDP 2}, we assume that for any solution to (\ref{eq MDP 1}) or (\ref{eq MDP 1-second}), the corresponding condition holds.
%\end{remark}

\vskip 0.5cm

Next we state the  result on MDP which is a re-formulation  of \cite[Theorem 2.3]{BDG} in the present setting.
\begin{theorem}\label{TH MDP 1}
Assume that  (S0)-(S4) hold. Suppose that there exists a measurable map $\Upsilon^0:L^2([0,T],K)\times L^2(\nu_T)\rightarrow\mD$
{
\footnote{
The existence of $\Upsilon^0$ implies that there exists a measurable map $\widetilde{\Upsilon}^0:C([0,T],K)\times L^2(\nu_T)\rightarrow\mD$ such that, for any $u=(\phi,\varphi)\in L^2([0,T],K)\times L^2(\nu_T)$,
$$
\widetilde{\Upsilon}^0(\int_0^\cdot\phi(s)\dif s,\varphi):=\Upsilon^0(u).
$$
%The definition of $\hat{\psi}$ can be found in (\ref{eq.corres-func-meas}).
 }
 }
 such that
\begin{itemize}
  \item[(MDP 1)] for any given $m\in(0,\infty)$, the set
  $$
  \Big\{\Upsilon^0(\phi,\varphi);\ (\phi,\varphi)\in S^m_1\times B_2(m)\Big\}
  $$
  is a compact subset of $\mD$;

  \item[(MDP 2)]for any given $m\in(0,\infty)$ and any family $\{(\phi_\epsilon,\psi_\epsilon);\epsilon>0\}\subset \mathcal{S}^m_1\times\mathcal{S}^m_{+,\epsilon}$ satisfying that $\phi_\epsilon\to\phi$ in $S^m_1$ and
  for some $\beta\in(0,1]$, $\varphi_\epsilon1_{\{|\varphi_\epsilon|\leq \beta/a(\epsilon)\}}\to\varphi$ in $B_2(\sqrt{m\kappa_2(1)})$
  where $\varphi_\epsilon=(\psi_\epsilon-1)/a(\epsilon)$,
  then,
  $$
   M^{u_\epsilon}\Rightarrow \Upsilon^0(\phi,\varphi)\text{ in }\mD,
  $$ where $M^{u_\epsilon}$ is the solution to (\ref{eq MDP 1-second}).
\end{itemize}
Then $\{M^\epsilon(t),\ t\in[0,T]\}_{\epsilon>0}$ satisfies a LDP with speed $\epsilon/a^2(\epsilon)$ and the rate
function $I$ given by
{\small
\begin{align}\label{eq rate MDP 1}
&I(g)\\
:=&\inf_{(\phi,\varphi)\in L^2([0,T],K)\times L_2(\nu_T),\ g=\Upsilon^0(\phi,\varphi)}
     \Big\{\frac{1}{2}\int_0^T\|\phi(s)\|^2_K\dif s+\frac{1}{2}\int_0^T\!\!\int_Z|\varphi(s,z)|^2\nu(\dif z)\dif s\Big\},\ g\in \mD,\nonumber
\end{align}
}
with the convention $\inf\emptyset=\infty$.
\end{theorem}

\vskip 0.5cm
Here is a sufficient condition for verifying the assumptions in Theorem \ref{TH MDP 1}.
\begin{theorem}\label{TH MDP 2}
Assume that  (S0)-(S4) hold. Suppose that there exists a measurable map $\Upsilon^0:L^2([0,T],K)\times L^2(\nu_T)\rightarrow\mD$ such that
\begin{itemize}
  \item[(MDP 1')] Given $m\in(0,\infty)$, for any $(\phi_n,\varphi_n)\in S^m_1\times B_2(m), n\in\mathbb{N}$ and $\phi_n\rightarrow \phi$
  in $S^m_1$, $\varphi_n\rightarrow\varphi$ in $B_2(m)$, as $n\rightarrow\infty$, then
  $$
  \Upsilon^0(\phi_n,\varphi_n)\rightarrow\Upsilon^0(\phi,\varphi)\text{ in }\mD;
  $$

  \item[(MDP 2')]Given $m\in(0,\infty)$, let $\{(\phi_\epsilon,\psi_\epsilon)\}_{\epsilon>0}$ be such that for every $\e>0$,
  $(\phi_\epsilon,\psi_\epsilon)\in \mathcal{S}^m_1\times\mathcal{S}^m_{+,\epsilon}$, and
  for some $\beta\in(0,1]$, $\varphi_\epsilon1_{\{|\varphi_\epsilon|\leq \beta/a(\epsilon)\}}\in B_2(\sqrt{m\kappa_2(1)})$
  where $\varphi_\epsilon=(\psi_\epsilon-1)/a(\epsilon)$, for any $\varpi>0$,
  $$
   \lim_{\epsilon\rightarrow0}
      P\Big(
         d(M^{u_\epsilon}, \Upsilon^0(\phi_\epsilon,\varphi_\epsilon1_{\{|\varphi_\epsilon|\leq \beta/a(\epsilon)\}}))>\varpi\Big)=0.
  $$
\end{itemize}
Then $\{M^\epsilon(t),\ t\in[0,T]\}_{\epsilon>0}$ satisfies a LDP with speed $\epsilon/a^2(\epsilon)$ and the rate
function $I$ defined by (\ref{eq rate MDP 1}).
%$$
%I(g):=\inf_{(\phi,\varphi)\in L^2([0,T],K)\times L_2(\nu_T),\ g=\Upsilon^0(\phi,\varphi)}
%     \Big\{\frac{1}{2}\int_0^T\|\phi(s)\|^2_K\dif s+\frac{1}{2}\int_0^T\!\!\int_Z|\varphi(s,z)|^2\nu(\dif z)\dif s\Big\}.
%$$

\end{theorem}

The proof  of Theorem \ref{TH MDP 2} is also similar to that of Theorem 3.2 in \cite{MSZ} and we omit it here.

\section{Applications}
%Let $\{W_t\}_{t\in[0,1]}$ be an $\mR^d$-valued Brownian motion defined on $(\Omega, \sF, \{\sF_t\}_{t\geq0}, \mP)$,
%which is a complete probability space with filtration. Let ($Z,\mathcal{B}(Z)$) be a locally compact Polish space and $N(\dif t,\dif z)$ be a Poisson random measure defined on $[0,T]\times Z$ with a $\sigma$-finite intensity measure $\dif t\nu(\dif z)$, where $\nu(\dif z)$ is a $\sigma$-finite  measure on $Z$. Assume that $W$ and $N$ are independent.

In this section, we will apply the abstract formulation in Section 4 to establish a LDP and a MDP for MVSDEs in $\mathbb{R}^d$. For this end, we set $K=V=H=E=\mathbb{R}^d,$ $d\in\mathbb{N}$. Then the notations  $C([0,T],K)$, $\Omega$, $W$, $S^m_1$, $\mathcal{S}^m_1$, $\mathcal{L}_2(K,H)$ and many others in Section 4 should be replaced correspondingly. For example $C([0,T],K)$ will be replaced by $C([0,T],\mathbb{R}^d)$, $W$ is a $d$-dimensional standard Brownian motion, $\mathcal{L}_2(K,H)$ is replaced by $\mathcal{L}_2(\mathbb{R}^d,\mathbb{R}^d)=\mathbb{R}^d\otimes \mathbb{R}^d$.

In the Euclidean space $\mathbb{R}^d$, the inner product and norm are denoted by $|\cdot|$ and $\langle\cdot,\cdot\rangle$, respectively. The Dirac measure concentrated at a point $x\in \mathbb{R}^d$ is denoted by $\delta_x$.

Denote by $\mathcal{P}(\mR^d)$ the collection of probability measures on $(\mR^d,\mathcal{B}(\mR^d))$. Define
\begin{align*}
\mathcal{P}_2:=\left\{\mu\in\mathcal{P}(\mR^d):\int_{\mR^d}|y|^2\mu(\dif y)<\infty\right\}.
\end{align*}
Then $\mathcal{P}_2$ is a Polish space equipped with the Wasserstein distance
\begin{align*}
\mW_2(\mu_1,\mu_2):=\inf\limits_{\pi\in\mathfrak{C}(\mu_1,\mu_2)}
\left(\int_{\mR^d\times\mR^d}|x-y|^2\pi(\dif x,\dif y)\right)^{\frac{1}{2}},
\end{align*}
where $\mathfrak{C}(\mu_1,\mu_2)$ is the set of all couplings for $\mu_1$ and $\mu_2$.
\begin{remark}\label{rem W2}
For any $\mathbb{R}^d$-valued random variables $X$ and $Y$,
$$
\mW_2(Law(X),Law(Y))\leq [\mathbb{E}(X-Y)^2]^{\frac{1}{2}}.
$$
\end{remark}

Let $\epsilon>0$. For measurable maps
\begin{align*}
b_\epsilon:[0,T]\times\mR^d\times\mathcal{P}_2\rightarrow\mR^d,\ \
\sigma_\epsilon:[0,T]\times\mR^d\times\mathcal{P}_2\rightarrow\mR^d\otimes\mR^d,
\end{align*}
and
\begin{align*}
G_\epsilon:[0,T]\times\mR^d\times\mathcal{P}_2\times Z\rightarrow\mR^d,
\end{align*}
consider the following MVSDEs on $\mR^d$:
\begin{align}\label{1-1}
   X^\epsilon(t)
=&
  h+\int_0^tb_\epsilon(s,X^\epsilon(s),Law(X^\epsilon(s)))\dif s+\sqrt{\epsilon}\int_0^t\sigma_\epsilon(s,X^\epsilon(s),Law(X^\epsilon(s)))\dif W(s)\nonumber\\
&+
\epsilon\int_0^t\!\!\int_{Z}G_\epsilon(s,X^\epsilon(s-),Law(X^\epsilon(s)),z)\widetilde{N}^{\epsilon^{-1}}(\dif z,\dif s),\ t\in[0,T],
\end{align}
here $h$ is an element of $\mR^d$.

\vskip 0.2cm
 We assume that

% \begin{center}
 (A0) For any $\epsilon>0$, there exists a unique solution $X^\epsilon$ to (\ref{1-1}).
% \end{center}
\vskip 0.5cm
 The aim of this section is to establish the  large and moderate deviation principles for the solutions $\{X^\epsilon,\epsilon>0\}$ to (\ref{1-1}) as $\epsilon$ decreases to 0.

\vskip 0.5cm

Let
\begin{align*}
b:[0,T]\times\mR^d\times\mathcal{P}_2\rightarrow\mR^d,\ \
\sigma:[0,T]\times\mR^d\times\mathcal{P}_2\rightarrow\mR^d\otimes\mR^d,
\end{align*}
and
\begin{align*}
G:[0,T]\times\mR^d\times\mathcal{P}_2\times Z\rightarrow\mR^d,
\end{align*}
be measurable maps.
\vskip 0.5cm

Introduce the following assumptions.
\vskip 0.4cm
 There are $L>0$ and $q\geq 1$ such that for each $t\in[0,T]$, $x,x'\in\mR^d$ and $\mu,\mu'\in\mathcal{P}_2$,
 \begin{itemize}
    \item[(A1)]
\begin{eqnarray*}
  %\begin{align}
  &&\langle x-x',b(t,x,\mu)-b(t,x',\mu)\rangle\leq L|x-x'|^2,\\
  &&|b(t,x,\mu)-b(t,x,\mu')|\leq L\mW_2(\mu,\mu'),\\
  &&|b(t,x,\mu)-b(t,x',\mu)|\leq L(1+|x|^{q-1}+|x'|^{q-1})|x-x'|,\\
  &&\|\sigma(t,x,\mu)-\sigma(t,x',\mu')\|_{\mathcal{L}_2}\leq L(|x-x'|+\mW_2(\mu,\mu')),\\
  &&\int_0^T\left(|b(t,0,\delta_0)|+\|\sigma(t,0,\delta_0)\|^2_{\mathcal{L}_2}\right)\dif  t<\infty.
  %\end{align}
\end{eqnarray*}
\item[(A2)] \begin{eqnarray*}
  %\begin{align}
  %|\sigma(t,x,\mu)-\sigma(t,x',\mu')|\leq L(|x-x'|+\mW_2(\mu,\mu')),\label{A-3}\\
  &&\int_Z|G(t,x,\mu,z)-G(t,x',\mu',z)|^2\nu(\dif z)\leq L(|x-x'|^2+\mW_2^2(\mu,\mu')),\\
  &&\int_0^T\!\!\int_Z|G(t,0,\delta_0,z)|^2\nu(\dif z) \dif t<\infty.
  %\end{align}
\end{eqnarray*}
\item[(A3)] As $\epsilon\downarrow0$, the maps $b_\epsilon$ and $\sigma_\epsilon$ converge uniformly to $b$ and $\sigma$ respectively, that is,  there exist nonnegative constants $\varrho_{b,\epsilon}$ and $\varrho_{\sigma,\epsilon}$ converging to 0  as $\epsilon\downarrow0$ such that
    \begin{eqnarray}
    &&\sup_{(t,x,\mu)\in[0,T]\times\mR^d\times\mathcal{P}_2}
    \Big(
       |b_\epsilon(t,x,\mu)-b(t,x,\mu)|
    \Big)
    \leq
    \varrho_{b,\epsilon},\label{eq Section 3 uniform b}\\
     &&\sup_{(t,x,\mu)\in[0,T]\times\mR^d\times\mathcal{P}_2}
    \Big(
       \|\sigma_\epsilon(t,x,\mu)-\sigma(t,x,\mu)\|_{\mathcal{L}_2}
    \Big)
    \leq
    \varrho_{\sigma,\epsilon}.\label{eq Section 3 uniform sigma}
    \end{eqnarray}
   %\item[(A4)] For any $\epsilon>0$, there exists a solution $X^\epsilon$ to (\ref{1-1}) as stated in Definitions \ref{def 3};

  \item[(A4)] Pathwise uniqueness holds for (\ref{1-1}) as stated in Definition \ref{def 2} with the $J$ replaced by $Law(X^\epsilon)$.
 \end{itemize}

\begin{remark}{
A trivial example for (A0), (A3) and (A4) to hold is that $b_\e=b, \sigma_\e=\sigma, G_\e=G$, i.e. the coefficients do not depend on $\epsilon$. In this particular case, we only need to assume (A1) and (A2).}
\end{remark}

The following result was proved in  Theorem 3.3 in \cite{[RST]}.
\begin{proposition}\label{th1}
Assume that (A1) holds. There exists a unique function $X^0=\{X^0(t),t\in[0,T]\}$ such that
\begin{itemize}
  \item $X^0\in C([0,T],\mathbb{R}^d)$,
  \item  $
    \int_0^T|b(s,X^0(s),Law(X^0(s)))|\dif s <\infty,
        $
  \item $X^0$ satisfies
          \begin{eqnarray}\label{eq 1.2}
   X^0(t)
=
  h+\int_0^tb(s,X^0(s),Law(X^0(s)))\dif s,\ \forall t\in[0,T].
\end{eqnarray}
\end{itemize}
\end{proposition}
\vskip 0.5cm
Note that $X^0$ is a deterministic path and
$
Law(X^0(s))=\delta_{X^0(s)}.
$
In the sequel, we will always use  $X^0$ to denote the unique solution to (\ref{eq 1.2}).

\subsection{Large deviations principle}

\vskip 0.3cm
In order to obtain the LDP, we need the following notations and assumptions.

Set
$$
L^2(\nu)
:=
\{f:Z\rightarrow\mathbb{R}
   |~f \text{ is }\mathcal{B}(Z)/\mathcal{B}(\mathbb{R})\text{-measurable and }\int_Z|f(z)|^2\nu(\dif z)<\infty\},
$$
and
\begin{align*}
\mathcal{H}=\Big\{g:Z\rightarrow\mathbb{R_+}| g\text{ is Borel measurable and} \text{ there exists }c>0\text{ such that }\\
\ \ \ \ \ \ \ \ \ \int_O e^{cg^2(z)}\nu(\dif z)<\infty\text{ for all }O\in\mathcal{B}(Z) \text{ with }\nu(O)<\infty.\Big\}
\end{align*}
\vskip 0.4cm

\begin{itemize}
  \item[(B1)] There exist $L_1, L_2, L_3\in\mathcal{H}\cap L^2(\nu)$ such that for all $t\in[0,T]$, $x,~x'\in\mR^d$, $\mu, ~\mu'\in\mathcal{P}_2$ and $z\in Z$,
\begin{quote}
\begin{align*}
|G(t,x,\mu,z)-G(t,x',\mu',z)|\leq L_1(z)\left(|x-x'|+\mW_2(\mu,\mu')\right),
\end{align*}
\begin{align*}
|G(t,0,\delta_0,z)|\leq L_2(z).
\end{align*}
\end{quote}
and there exists nonnegative constant $\varrho_{G,\epsilon}$ converging to $0$ such that
\begin{align*}
\sup_{(t,x,\mu)\in[0,T]\times\mR^d\times\mathcal{P}_2}
    |G_\epsilon(t,x,\mu,z)-G(t,x,\mu,z)|
    \leq
    \varrho_{G,\epsilon}L_3(z).
\end{align*}
\end{itemize}
\vskip 0.4cm
It is easy to see that (B1) implies  (A2).

Before stating the main result in this subsection, we need the following result.
\vskip 0.3cm
Recalling $S^m_1$, $S^m_2$ and $S$ defined in Section 2, we have
\begin{proposition}\label{prop 5}
Assume that (A1) and (B1) hold. Then for any $u=(\phi,\psi)\in S$, there exists a unique solution $Y^u=\{Y^u(t),t\in[0,T]\}\in C([0,T],\mathbb{R}^d)$  to the following equation:
\begin{align}\label{eq rate LDP 1}
Y^u(t)
=&h+\int_0^tb(s,Y^u(s),Law(X^0(s)))\dif s
+\int_0^t\sigma(s,Y^u(s),Law(X^0(s)))\phi(s)\dif s\nonumber\\
&+\int_0^t\!\!\int_{Z} G(s,Y^u(s),Law(X^0(s)),z)(\psi(s,z)-1)\nu(\dif z)\dif s,\ \ t\in[0,T].
\end{align}
Moreover, for any $m>0$,
\begin{align}\label{eq ll}
\sup_{u=(\phi,\psi)\in S_1^m\times S_2^m}\sup_{t\in[0,T]}|Y^u(t)|<\infty.
\end{align}
\end{proposition}

\begin{proof}

Without loss of generality, we assume that $u=(\phi,\psi)\in S_1^m\times S_2^m$.

\vskip 0.3cm
We first prove that there exists a unique solution to (\ref{eq rate LDP 1}). Set
$$
b_u(t,x):=b(t,x,Law(X^0(t))),
\ \ \ \
\sigma_\phi(t,x):=\sigma(t,x,Law(X^0(t)))\phi(t),
$$
%{\red I advise us to modify the notation of $b_u$ to avoid confusion with $b_\epsilon$, and indeed $b_u$ does not depend on $u$ here.}
and
$$
G_\psi(t,x,z):=G(t,x,Law(X^0(t)),z)(\psi(t,z)-1),\ \
G^Z_\psi(t,x):=\int_ZG_\psi(t,x,z)\nu(\dif z).
$$
By (A1), Remark \ref{rem W2}, and the fact that $X^0\in C([0,T],\mathbb{R}^d)$, we have that for each $t\in[0,T]$, $x,x'\in\mR^d$,
\begin{align}\label{eq bu 1}
&\langle x-x',b_u(t,x)-b_u(t,x')\rangle\leq L|x-x'|^2,\\
 &|b_u(t,x)-b_u(t,x')|\leq L(1+|x|^{q-1}+|x'|^{q-1})|x-x'|,
\end{align}
\begin{align}\label{eq sigm phi}
|\sigma_\phi(t,x)-\sigma_\phi(t,x')|
\leq L|x-x'||\phi(t)|,
\end{align}
\begin{align}\label{eq bu 2}
\int_0^T|b_u(t,0)|\dif t
\leq&\int_0^T|b(t,0,Law(X^0(t)))-b(t,0,\delta_0)|\dif t
+\int_0^T|b(t,0,\delta_0)|\dif t\nonumber\\
\leq&L\int_0^T\mW_2(Law(X^0(t)),\delta_0)\dif t
+\int_0^T|b(t,0,\delta_0)|\dif t\nonumber\\
\leq&L\int_0^T|X^0(t)|\dif t+\int_0^T|b(t,0,\delta_0)|\dif t
<\infty,
\end{align}
and
\begin{align}\label{eq sigm phi 1}
&\int_0^T|\sigma_\phi(t,0)|\dif t\cr
\leq&\int_0^T|\sigma(t,0,Law(X^0(t)))\phi(t)-\sigma(t,0,\delta_0)\phi(t)|\dif t
   +\int_0^T|\sigma(t,0,\delta_0)\phi(t)|\dif t\nonumber\\
\leq&L\int_0^T\mW_2(Law(X^0(t)),\delta_0)|\phi(t)|\dif t+
    \int_0^T|\sigma(t,0,\delta_0)\phi(t)|\dif t\nonumber\\
\leq&L\sup_{t\in[0,T]}|X^0(t)|\Big(\int_0^T|\phi(t)|^2\dif t+T\Big)
  \!+\!\int_0^T\|\sigma(t,0,\delta_0)\|^2_{\mathcal{L}_2}\dif t
  \!+\!\int_0^T|\phi(t)|^2\dif t\cr<&\infty,
\end{align}
where we use the condition $\phi\in S_1^m$.

By (B1), Remark \ref{rem W2}, and the fact that $X^0\in C([0,T],\mathbb{R}^d)$, we get for each $t\in[0,T]$, $x,x'\in\mR^d$,
\begin{align}\label{eq GZ}
&|G^Z_\psi(t,x)-G^Z_\psi(t,x')|\cr
\leq&\int_Z|G(t,x,Law(X^0(t)),z)(\psi(t,z)-1)-G(t,x',Law(X^0(t)),z)(\psi(t,z)-1)|\nu(\dif z)\cr
\leq&\int_ZL_1(z)|\psi(t,z)-1|\nu(\dif z)|x-x'|,
\end{align}
and
\begin{align}\label{eq GZ 0}
&\int_0^T|G^Z_\psi(t,0)|\dif t\cr
\leq&\int_0^T\!\!\int_Z|G(t,0,Law(X^0(t)),z)(\psi(t,z)-1)-G(t,0,\delta_0,z)
(\psi(t,z)-1)|\nu(\dif z)\dif t\cr
&+\int_0^T\!\!\int_Z|G(t,0,\delta_0,z)(\psi(t,z)-1)|\nu(\dif z)\dif t\cr
\leq&\int_0^T\!\!\int_ZL_1(z)\mW_2(Law(X^0(t)),\delta_0)|\psi(t,z)-1|\nu(\dif z)\dif t\cr
&+\int_0^T\!\!\int_ZL_2(z)|\psi(t,z)-1|\nu(\dif z)\dif t\cr
\leq&\sup_{t\in[0,T]}|X^0(t)|\int_0^T\!\!\int_ZL_1(z)|\psi(t,z)-1|\nu(\dif z)\dif t\cr
&+\int_0^T\!\!\int_ZL_2(z)|\psi(t,z)-1|\nu(\dif z)\dif t.
\end{align}
By Lemma 3.4 in \cite{Budhiraja-Chen-Dupuis}, we have
\begin{align}\label{CK}
\sup\limits_{\varphi\in S^m_2}
\int_0^T\!\!\int_Z(L_1(z)+L_2(z)+L_3(z))|\varphi(t,z)-1|\nu(\dif z)\dif t<\infty.
\end{align}
In view of  (\ref{eq bu 1})--(\ref{CK}), using the classical fixed point arguments, we
can deduce that there exists a unique function $Y^u=\{Y^u(t),t\in[0,T]\}\in C([0,T],\mathbb{R}^d)$ which is the solution to (\ref{eq rate LDP 1}).
\vskip 0.3cm
Next we  prove (\ref{eq ll}). By the chain rule,
\begin{align}\label{YU-1}
|Y^u(t)|^2=&|h|^2+2\int_0^t\<b(s,Y^u(s),Law(X^0(s))),Y^u(s)\>\dif s\cr
&+2\int_0^t\<\sigma(s,Y^u(s),Law(X^0(s)))\phi(s),Y^u(s)\>\dif s\cr
&+2\int_0^t\!\!\int_Z\<G(s,Y^u(s),Law(X^0(s)),z)(\psi(s,z)-1),Y^u(s)\>\nu(\dif z)\dif s\cr
:=&|h|^2+I_1(t)+I_2(t)+I_3(t).
\end{align}
By (A1) and Remark \ref{rem W2},
\begin{align}\label{YU-2}
I_1(t)=&2\int_0^t\<b(s,Y^u(s),Law(X^0(s)))-b(s,0,Law(X^0(s))),Y^u(s)\>\dif s\cr
&+2\int_0^t\<b(s,0,Law(X^0(s)))-b(s,0,\delta_0),Y^u(s)\>\dif s\cr
&+2\int_0^t\<b(s,0,\delta_0),Y^u(s)\>\dif s\cr
\leq&2L\int_0^t|Y^u(s)|^2\dif s+2L\int_0^t|Y^u(s)||X^0(s)|\dif s+2\int_0^t|b(s,0,\delta_0)||Y^u(s)|\dif s\cr
\leq&\int_0^t\!\Big(3L+|b(s,0,\delta_0)|\Big)|Y^u(s)|^2\dif s\!+\!L\int_0^t|X^0(s)|^2ds\!+\!\int_0^t\!|b(s,0,\delta_0)|\dif s,
\end{align}
and
%
%By similar arguments, we can obtain
\begin{align}\label{YU-3}
&I_2(t)\cr
=&
2\int_0^t\<\sigma(s,Y^u(s),Law(X^0(s)))\phi(s)-\sigma(s,0,\delta_0)\phi(s),Y^u(s)\>\dif s\cr
&+2\int_0^t\<\sigma(s,0,\delta_0)\phi(s),Y^u(s)\>\dif s\cr
\leq&2L\int_0^t(|Y^u(s)|+|X^0(s)|)|\phi(s)||Y^u(s)|\dif s
+2\int_0^t\|\sigma(s,0,\delta_0)\|_{\mathcal{L}_2}|\phi(s)||Y^u(s)|\dif s\nonumber\\
\leq&L\int_0^t|Y^u(s)|^2\Big(1+|\phi(s)|^2\Big)\dif s
+L\sup_{l\in[0,T]}|X^0(l)|\int_0^t|Y^u(s)|^2+|\phi(s)|^2\dif s\nonumber\\
&+\int_0^t\Big(\|\sigma(s,0,\delta_0)\|^2_{\mathcal{L}_2}+|\phi(s)|^2\Big)\Big(|Y^u(s)|^2+1\Big)\dif s\nonumber\\
\leq&
(L+1)\int_0^t|Y^u(s)|^2\Big(1+\sup_{l\in[0,T]}|X^0(l)|+|\phi(s)|^2+\|\sigma(s,0,\delta_0)\|^2_{\mathcal{L}_2}\Big)\dif s\nonumber\\
&+
(1+L\sup_{l\in[0,T]}|X^0(l)|)\int_0^t\|\sigma(s,0,\delta_0)\|^2_{\mathcal{L}_2}+|\phi(s)|^2\dif s.
\end{align}
By (B1) and Remark \ref{rem W2},
\begin{align}\label{YU-4}
&I_3(t)\nonumber\\
=&
2\int_0^t\!\!\int_Z\<\Big(G(s,Y^u(s),Law(X^0(s)),z)-G(s,0,\delta_0,z)\Big)(\psi(s,z)-1),Y^u(s)\>\nu(\dif z)\dif s\nonumber\\
&+
2\int_0^t\!\!\int_Z\<G(s,0,\delta_0,z)(\psi(s,z)-1),Y^u(s)\>\nu(\dif z)\dif s\nonumber\\
\leq&
2\int_0^t\!\!\int_Z L_1(z)\Big(|Y^u(s)|+|X^0(s)|\Big)|\psi(s,z)-1||Y^u(s)|\nu(\dif z)\dif s\nonumber\\
&+
2\int_0^t\!\!\int_Z L_2(z)|\psi(s,z)-1||Y^u(s)|\nu(\dif z)\dif s\nonumber\\
\leq&
\int_0^t\!\!\int_Z|Y^u(s)|^2\Big(3L_1(z)+L_2(z)\Big)|\psi(s,z)-1|\nu(\dif z)\dif s\\
&+
\sup_{l\in[0,T]}|X^0(l)|^2\int_0^t\!\!\int_ZL_1(z)|\psi(s,z)-1|\nu(\dif z)\dif s
+
\int_0^t\!\!\int_ZL_2(z)|\psi(s,z)-1|\nu(\dif z)\dif s.\nonumber
\end{align}
Set
\begin{align}\label{eq theta 1}
\Theta_1(u)
:=&
|h|^2+\int_0^TL|X^0(s)|^2+|b(s,0,\delta_0)|\dif s\cr
&+(1+L\sup_{l\in[0,T]}|X^0(l)|)\int_0^T\|\sigma(s,0,\delta_0)\|^2_{\mathcal{L}_2}+|\phi(s)|^2\dif s\cr
&+
\sup_{l\in[0,T]}|X^0(l)|^2\int_0^T\!\!\int_ZL_1(z)|\psi(s,z)-1|\nu(\dif z)\dif s\\
&+
\int_0^T\!\!\int_ZL_2(z)|\psi(s,z)-1|\nu(\dif z)\dif s,\nonumber
\end{align}
and
\begin{align}\label{eq theta 2}
\Theta_2(u)
:=&
\int_0^T\Big(3L+|b(s,0,\delta_0)|\Big)\dif s+
\int_0^T\!\!\int_Z\Big(3L_1(z)+L_2(z)\Big)|\psi(s,z)-1|\nu(\dif z)\dif s\nonumber\\
&+(L+1)\int_0^T\Big(1+\sup_{l\in[0,T]}|X^0(l)|+|\phi(s)|^2
+\|\sigma(s,0,\delta_0)\|^2_{\mathcal{L}_2}\Big)\dif s.
\end{align}
It follows from (A1), (B1) and (\ref{CK}) that there exists a constant $\widetilde{C}_m$ such that
\begin{align}\label{eq theta 3}
\sup\limits_{u=(\phi,\psi)\in S^m_1\times S^m_2}\Big(\Theta_1(u)+\Theta_2(u)\Big)\leq \widetilde{C}_m<\infty.
\end{align}
Then by combining (\ref{YU-1})-(\ref{eq theta 3}) together and  using Gronwall's inequality, we get
\begin{align}\label{ub}
\sup\limits_{u=(\phi,\psi)\in S^m_1\times S^m_2}\sup\limits_{t\in[0,T]}|Y^u(t)|^2\leq \sup\limits_{u=(\phi,\psi)\in S^m_1\times S^m_2}\Theta_1(u) e^{\Theta_2(u)}
\leq\widetilde{C}_m e^{\widetilde{C}_m}<\infty,
\end{align}
which completes the proof.

\end{proof}

We now state the main result in this subsection.
\begin{theorem}\label{Th ex 1}
Assume that (A0), (A1), (B1), (A3) and (A4) hold. Then the solutions $\{X^\epsilon,~\e>0\}$ to (\ref{1-1}) satisfy a LDP on $D([0,T],\mathbb{R}^d)$ with speed $\epsilon$ and the good
rate function $I$ given by
\begin{eqnarray}\label{rate ex LDP}
I(g):=\inf\{Q_1(\phi)+Q_2(\psi):u=(\phi,\psi)\in S,~Y^u=g\},~g\in D([0,T],\mathbb{R}^d),
\end{eqnarray}
where for $u=(\phi,\psi)\in S$, $Y^u$ is the unique solution of (\ref{eq rate LDP 1}).
Here we use the convention that the infimum of an empty set is $\infty$.
\end{theorem}

\begin{proof}
We will apply Theorem \ref{Thm zhunze 2}. Proposition \ref{prop 5} allows us to define a map
\begin{eqnarray}\label{def G0}
\Gamma^0:S\ni u=(\phi,\psi) \mapsto Y^u\in D([0,T],\mathbb{R}^d),
\end{eqnarray}
here $Y^u$ is the unique solution of (\ref{eq rate LDP 1}).
%Then we can define a map $\Gamma^0:C([0,T],H)\times M_{FC}\big([0,T]\times Z\big)\rightarrow\mD$ such that, for any $u=(\phi,\psi)\in S$,
%$$
%\Gamma^0(\int_0^\cdot\phi(s)ds,\hat{\psi}):=\overline{\Gamma}^0(u).
%$$

By (\ref{sol-control}) and (\ref{eq X0 2-second}), for any $\epsilon>0$, $m\in(0,\infty)$ and $u_\e=(\phi_\e,\psi_\e)\in\mathcal{S}^m_1\times\mathcal{S}^m_2$, there exists a unique solution
%{\footnote{In fact, the solution is unique. Since the uniqueness to (\ref{EQ4}) or (\ref{EQ4 LDP 2}) is not required(see Remark \ref{rem 3} and Theorem \ref{Thm zhunze 2}), we omit the proof of the uniqueness to (\ref{EQ4}) or (\ref{EQ4 LDP 2}) here. }
%}
 $\{Z^{u_\e}(t)\}_{t\in[0, T]}$ to the following SDE %{\red (here I think the notation $Z^{u_\e}(t)$ should be changed to $Z^{\e,u_\e}(t)$) }
\begin{align}%\label{EQ4}
\dif Z^{u_\e}(t)
%=&
%b_\epsilon(t,Z^{u_\e}(t),Law(X^\e(t)))\dif t+\sqrt{\e}\sigma_\epsilon(t,Z^{u_\e}(t),Law(X^\e(t)))\dif W(t)\nonumber\\
%&+\sigma_\epsilon(t,Z^{u_\e}(t),Law(X^\e(t)))\phi_\e(t)\dif t\\
%&+\e\int_ZG_\epsilon(t,Z^{u_\e}(t-),Law(X^\e(t)),z)\left(N^{\e^{-1}\psi_\e}(\dif z,\dif t)-\e^{-1}\nu(\dif z)\dif t\right),\nonumber\\
=&
b_\epsilon(t,Z^{u_\e}(t),Law(X^\e(t)))\dif t+\sqrt{\e}\sigma_\epsilon(t,Z^{u_\e}(t),Law(X^\e(t)))\dif W(t)\nonumber\\
&+\sigma_\epsilon(t,Z^{u_\e}(t),Law(X^\e(t)))\phi_\e(t)\dif t\label{EQ4 LDP 2}\\
&+\e\int_ZG_\epsilon(t,Z^{u_\e}(t-),Law(X^\e(t)),z)\widetilde{N}^{\e^{-1}\psi_\e}(\dif z,\dif t),\nonumber\\
&+\int_ZG_\epsilon(t,Z^{u_\e}(t),Law(X^\e(t)),z)\left(\psi_\e(t,z)-1\right)\nu(\dif z)\dif t,\nonumber
\end{align}
with the initial data $Z^{u_\e}(0)=h$ and $X^\epsilon$ is the solution to (\ref{1-1}).

According to Theorem \ref{Thm zhunze 2}, to complete the proof of the theorem,
it is sufficient to verify the following two claims:
\begin{itemize}
  \item[(\textbf{LDP1})] For any given $m\in(0,\infty)$, let $u_n=(\phi_n,\psi_n),~n\in\mathbb{N},~u=(\phi,\psi)\in S^m_1\times S^m_2$ be such that
$u_n\rightarrow u$ in $S^m_1\times S^m_2$ as $n\rightarrow\infty$. Then
$$
\lim_{n\rightarrow\infty}\sup_{t\in[0,T]}|{\Gamma}^0(u_n)(t)-{\Gamma}^0(u)(t)|=0.
$$
  \item[(\textbf{LDP2})] For any given $m\in(0,\infty)$, let $\{u_\e=(\phi_\e,\psi_\e),~\e>0\}\subset \mathcal{S}^m_1\times \mathcal{S}^m_2$. Then
\begin{align*}
\lim\limits_{\e\rightarrow0}\mE\left(\sup\limits_{t\in[0,T]}|Z^{u_\e}(t)-{\Gamma}^0(u_\e)(t)|^2\right)=0.
\end{align*}
\end{itemize}

The verification of (\textbf{LDP1}) will be given in Proposition \ref{Yu-con}. (\textbf{LDP2}) will
be established in Proposition \ref{lem LDP 2}.
\end{proof}

\vskip 0.3cm

Recall the map ${\Gamma}^0$ in (\ref{def G0}).

\bp\label{Yu-con}
For any given $m\in(0,\infty)$, let $u_n=(\phi_n,\psi_n),~n\in\mathbb{N},~u=(\phi,\psi)\in S^m_1\times S^m_2$ be such that
$u_n\rightarrow u$ in $S^m_1\times S^m_2$ as $n\rightarrow\infty$. Then
$$
\lim_{n\rightarrow\infty}\sup_{t\in[0,T]}|{\Gamma}^0(u_n)(t)-{\Gamma}^0(u)(t)|=0.
$$
% $$
%Y^{u_n}\rightarrow Y^u\text{ in } C([0,T],\mathbb{R}^d).
%$$

\ep
\begin{proof} Let $Y^u$ be the the solution of (\ref{eq rate LDP 1}) and $Y^{u_n}$ be the solution of (\ref{eq rate LDP 1}) with $u$ replaced by $u_n$. By the definition of ${\Gamma}^0$, $Y^{u_n}={\Gamma}^0(u_n)$ and $Y^{u}={\Gamma}^0(u)$. For simplicity, denote $Y_n=Y^{u_n}$ and $Y=Y^{u}$. Note that $Y, Y_n \in C([0,T],\mathbb{R}^d), \forall n\in\mathbb{N}$.

The proof is divided into two steps.

\vskip 0.3cm
{\bf Step 1}: We first prove that $\{Y_n\}_{n\geq1}$ is pre-compact in $C([0,T],\mR^d)$.
It suffices to show that $\{Y_n\}_{n\geq1}$ is uniformly bounded and equi-continuous in $C([0,T],\mR^d)$.
\vskip 0.2cm

It follows from (\ref{eq ll}) that $\{Y_n\}_{n\geq1}$ is uniformly bounded, i.e.
\begin{align}\label{ubm LDP}
\sup\limits_{n\geq1}\sup\limits_{t\in[0,T]}|Y_n(t)|=:C_m<\infty.
\end{align}
\vskip 0.2cm
Next, we will prove that $\{Y_n\}_{n\geq1}$ is equi-continuous  in $C([0,T],\mR^d)$.
For $t>s$,
\begin{align}\label{eq 5.1}
Y_n(t)-Y_n(s)=&\int_s^tb(r,Y_n(r),Law(X^0(r)))\dif r
+\int_s^t\sigma(r,Y_n(r),Law(X^0(r)))\phi_n(r)\dif r\nonumber\\
&+\int_s^t\!\!\int_ZG(r,Y_n(r),Law(X^0(r)),z)(\psi_n(r,z)-1)\nu(\dif z)\dif r.
\end{align}
By (A1), (\ref{ubm LDP}) and Remark \ref{rem W2} we have
\begin{align}\label{best1}
&\int_s^t|b(r,Y_n(r),Law(X^0(r)))|\dif r\cr
\leq&\int_s^t|b(r,Y_n(r),Law(X^0(r)))-b(r,0,Law(X^0(r)))|\dif r\cr
&+\int_s^t|b(r,0,Law(X^0(r)))-b(r,0,\delta_0)|\dif r
+\int_s^t|b(r,0,\delta_0)|\dif r\cr
\leq&\int_s^tL(1+|Y_n(r)|^{q-1})|Y_n(r)|\dif r+L\int_s^t|X^0(r)|\dif r+\int_s^t|b(r,0,\delta_0)|\dif r\cr
\leq&L\Big(C_m(1+C_m^{q-1})+\sup_{r\in[0,T]}|X^0(r)|\Big)|t-s|+\int_s^t|b(r,0,\delta_0)|\dif r,
\end{align}
%where $C$ is constant which is independent of $t$ and $s$. Similarly,
and
\begin{align}\label{sest1}
&\int_s^t|\sigma(r,Y_n(r),Law(X^0(r)))\phi_n(r)|\dif r\cr
\leq&
    \int_s^t|\sigma(r,Y_n(r),Law(X^0(r)))\phi_n(r)-\sigma(r,0,\delta_0)\phi_n(r)|\dif r
    +
    \int_s^t|\sigma(r,0,\delta_0)\phi_n(r)|\dif r
    \cr
\leq&\int_s^tL(1+|Y_n(r)|+|X^0(r)|)|\phi_n(r)|\dif r
+\!\int_s^t\|\sigma(r,0,\delta_0)\|_{\mathcal{L}_2}|\phi_n(r)|\dif r\cr
\leq&  L(1+C_m+\sup_{r\in[0,T]}|X^0(r)|)\Big(\int_0^T|\phi_n(r)|^2\dif r\Big)^{\frac{1}{2}}|t-s|^{\frac{1}{2}}\cr
&+\left(\int_s^t|\sigma(r,0,\delta_0)|^2\dif r\right)^{\frac{1}{2}}\left(\int_0^T|\phi_n(r)|^2\dif r\right)^{\frac{1}{2}}.
%\leq&\left(|t-s|^{\frac{1}{2}}+\left(\int_0^T|\phi_n(r)|^2\dif r\right)^{\frac{1}{2}}\right).
\end{align}
By (B1), (\ref{ubm LDP}) and Remark \ref{rem W2}, we have
\begin{align}\label{gest1}
&\int_s^t\!\!\int_Z|G(r,Y_n(r),Law(X^0(r)),z)(\psi_n(r,z)-1)|\nu(\dif z)\dif r\cr
\leq&
   \int_s^t\!\!\int_Z|G(r,Y_n(r),Law(X^0(r)),z)(\psi_n(r,z)-1)-G(r,0,\delta_0,z)(\psi_n(r,z)-1)|\nu(\dif z)\dif r\cr
   &+
   \int_s^t\!\!\int_Z|G(r,0,\delta_0,z)(\psi_n(r,z)-1)|\nu(\dif z)\dif r
   \cr
\leq&\int_s^t\!\!\int_ZL_1(z)(|Y_n(r)|+|X^0(r)|)|\psi_n(r,z)-1|\nu(\dif z)\dif r
+\int_s^t\!\!\int_ZL_2(z)|\psi_n(r,z)-1|\nu(\dif z)\dif r\cr
\leq& (C_m+\sup_{r\in[0,T]}|X^0(r)|+1)\int_s^t\!\!\int_Z(L_1(z)+L_2(z))|\psi_n(r,z)-1|\nu(\dif z)\dif r.
\end{align}

 To show that the right side of (\ref{gest1}) is uniformly small, we need the following result, see  \cite[(3.3) of Lemma 3.1]{Zhai-Zhang} or  \cite[ Remark 2]{Yang-Zhai-Zhang}, or  \cite[(3.5) of Lemma 3.4]{Budhiraja-Chen-Dupuis}.
\begin{center}
  \emph{For every ${\theta}>0$, there exists some $\beta>0$ such that for any $ A\in\mathcal{B}([0,T])$ with ${Leb}_T({A})\leq \beta$,
\begin{eqnarray}\label{eq3 small {A}}
    \sup_{i=1,2,3} \, \sup_{\varphi\in S^m_2} \int_{A}\int_{{Z}}L_i(z)|\varphi(s,z)-1|\nu(\dif z)\,\dif s\leq {\theta}.%\ \ \ \ i=1,2,3.
\end{eqnarray}}
\end{center}

On the other hand, (A1) implies that, for every ${\theta}>0$, there exists some $\beta>0$ such that if $ A\in\mathcal{B}([0,T])$ satisfies ${Leb}_T({A})\leq \beta$, then
\begin{align}\label{eq3 small}
\int_A |b(r,0,\delta_0)|\dif r+\int_A\|\sigma(r,0,\delta_0)\|^2_{\mathcal{L}^2}\dif r\leq {\theta}.%\ \ \ \ i=1,2,3.
\end{align}
Combining (\ref{eq 5.1})-(\ref{eq3 small}) together, we can deduce that $\{Y_n\}_{n\geq1}$ is equi-continuous on $[0,T]$. So $\{Y_n\}_{n\geq1}$ is pre-compact in $C([0,T],\mR^d)$.

\vskip 0.3cm

\textbf{Step 2}: Let $\gamma$ be a limit of some subsequence of $\{Y_n\}_{n\geq1}$. We will show that $\gamma=Y$, completing
the proof of the proposition.  Without loss of generality, we simply assume
\begin{eqnarray}\label{eq sup LDP}
\lim_{n\rightarrow\infty}\sup_{t\in[0,T]}|\gamma(t)-Y_n(t)|=0.
\end{eqnarray}
\vskip 0.2cm

First, we note that
\begin{align}\label{Gama}
\sup\limits_{t\in[0,T]}|\gamma(t)|\leq\sup\limits_{n\geq1}\sup\limits_{t\in[0,T]}|Y_n(t)|
=C_m<\infty.
\end{align}
By (A1), (\ref{eq sup LDP}) and (\ref{Gama}),
\begin{align*}
&\int_0^T|b(s,Y_n(s),Law(X^0(s)))-b(s,\gamma(s),Law(X^0(s)))|\dif s\cr
\leq &L\int_0^T(1+|Y_n(s)|^{q-1}+|\gamma(s)|^{q-1})|Y_n(s)-\gamma(s)|\dif s\cr
\leq& LT(1+2C_m^{q-1})\sup\limits_{s\in[0,T]}|Y_n(s)-\gamma(s)|\rightarrow0, \ \ n\rightarrow\infty.
\end{align*}
Hence, for each $t\in[0,T]$,
\begin{align}\label{bcon}
&\int_0^tb(s,Y_n(s),Law(X^0(s)))\dif s
\rightarrow\int_0^tb(s,\gamma(s),Law(X^0(s)))\dif s,\ \  n\rightarrow\infty.
\end{align}
Due to (A1), $X^0\in C([0,T],\mathbb{R}^d)$, Remark \ref{rem W2} and (\ref{Gama}), it is not difficult to prove that
$$\int_0^T\|\sigma(s,\gamma(s),Law(X^0(s)))\|^2_{\mathcal{L}_2}\dif s<\infty.$$
  Since $\phi_n$ converges to $\phi$ weakly in $L^2([0, T],\mR^{d})$, for any $e\in \mathbb{R}^d$,
\begin{align*}
\int_0^t\langle\sigma(s,\gamma(s),Law(X^0(s)))\phi_n(s),e\rangle\dif s\rightarrow\int_0^t\langle\sigma(s,\gamma(s),Law(X^0(s)))\phi(s),e\rangle\dif s, \ \ n\rightarrow\infty.
\end{align*}
Hence,
\begin{align}\label{eq si 1}
\int_0^t\sigma(s,\gamma(s),Law(X^0(s)))\phi_n(s)\dif s\rightarrow\int_0^t\sigma(s,\gamma(s),Law(X^0(s)))\phi(s)\dif s, \ \ n\rightarrow\infty.
\end{align}
By (A1) and (\ref{eq sup LDP}) we have
\begin{align}\label{eq si 2}
&\int_0^T|\sigma(s,Y_n(s),Law(X^0(s)))\phi_n(s)-\sigma(s,\gamma(s),Law(X^0(s)))\phi_n(s)|\dif s\nonumber\\
\leq& L\int_0^T|Y_n(s)-\gamma(s)||\phi_n(s)|\dif s\nonumber\\
\leq& LT^{\frac{1}{2}}\sup\limits_{s\in[0,T]}|Y_n(s)-\gamma(s)|
\sup\limits_{i\geq1}\left(\int_0^T|\phi_i(s)|^2\dif s\right)^{\frac{1}{2}}\rightarrow0,\ \  \text{as} \ n\rightarrow\infty
\end{align}
where  we use $\phi_i\in S^m_1$, i.e. $\frac{1}{2}\int_0^T|\phi_i(s)|^2ds\leq m$, in the last inequality.

Therefore, (\ref{eq si 1}) and (\ref{eq si 2}) imply that
\begin{align}\label{scon}
&\int_0^t\sigma(s,Y_n(s),Law(X^0(s)))\phi_n(s)\dif s
\rightarrow\int_0^t\sigma(s,\gamma(s),Law(X^0(s)))\phi(s)\dif s, \ \  n\rightarrow\infty.
\end{align}
Applying (B1), (\ref{CK}) and (\ref{eq sup LDP}), we have
\begin{align}\label{eq uu}
&\int_0^t\!\!\int_Z|G(s,Y_n(s),Law(X^0(s)),z)
-G(s,\gamma(s),Law(X^0(s)),z)||(\psi_n(s,z)-1)|\nu(\dif z)\dif s\cr
\leq&\int_0^t\!\!\int_Z|Y_n(s)-\gamma(s)|L_1(z)|\psi_n(s,z)-1|\nu(\dif z)\dif s\cr
\leq&\sup\limits_{i\geq 1}\int_0^T\!\!\int_ZL_1(z)|\psi_i(s,z)-1|\nu(\dif z)\dif s\sup\limits_{s\in[0,T]}|Y_n(s)-\gamma(s)|\rightarrow0, \ \ n\rightarrow\infty.
\end{align}

By (B1), Remark \ref{rem W2},  (\ref{CK}) and (\ref{Gama}), we see that
\begin{eqnarray*}
\int_0^t\!\!\int_Z|G(s,\gamma(s),Law(X^0(s)),z)|^2\nu(\dif z)\dif s<\infty.
\end{eqnarray*}

Combining the above inequality with $\psi_n\rightarrow\psi$ in $S^m_2$, using Lemma 3.11 in \cite{Budhiraja-Chen-Dupuis}, we deduce that
\begin{align*}
&\lim_{n\rightarrow\infty}\int_0^t\!\!\int_Z G(s,\gamma(s),Law(X^0(s)),z)(\psi_n(s,z)-1)\nu(\dif z)\dif s\cr
&=\int_0^t\!\!\int_Z G(s,\gamma(s),Law(X^0(s)),z)(\psi(s,z)-1)\nu(\dif z)\dif s.
\end{align*}

This, together with (\ref{eq uu}), yields that
\begin{align}\label{gcon}
&\lim_{n\rightarrow\infty}\int_0^t\!\!\int_ZG(s,Y_n(s),Law(X^0(s)),z)(\psi_n(s,z)-1)\nu(\dif z)\dif s\cr
&=\int_0^t\!\!\int_ZG(s,\gamma(s),Law(X^0(s)),z)(\psi(s,z)-1)\nu(\dif z)\dif s.
\end{align}

Recall that  $Y_n$ is the solution of (\ref{eq rate LDP 1}) with $u$ replaced by $u_n$:
 \begin{align*}
   Y_n(t)
=&
  h+\int_0^tb(s,Y_n(s),Law(X^0(s)))\dif s
+
\int_0^t\sigma(s,Y_n(s),Law(X^0(s)))\phi_n(s)\dif s\nonumber\\
&+
\int_0^t\int_{Z} G(s,Y_n(s),Law(X^0(s)),z)(\psi_n(s,z)-1)\nu(\dif z)\dif s,\ \ t\in[0,T].
\end{align*}
Letting $n\rightarrow\infty$ and taking into account   (\ref{eq sup LDP}), (\ref{bcon}), (\ref{scon}) and (\ref{gcon}), we see that  $\gamma$ is a solution to (\ref{eq rate LDP 1}), and the uniqueness of the solutions of  (\ref{eq rate LDP 1}) implies that $\gamma=Y$, which completes the proof.
%Since the solution exists uniquely, so $\gamma=Y^u$.

%The proof of \textbf{Step 2} is complete.

%The proof of this proposition is complete.

\end{proof}

\vskip 0.3cm
To  verify (\textbf{LDP2}), we need the following result.
\vskip 0.2cm

%For any $\epsilon\in(0,1)$, recalling that $\{Y^\e_t\}$ denotes the solution of Equ. (\ref{EQ2}), we have the following estimate.
\begin{lemma}\label{yey0}
There exists some $\e_0>0$ and a constant $C_T$ independent of $\e$ such that
\begin{align}
\mE\left(\sup\limits_{t\in[0,T]}|X^\e(t)-X^0(t)|^2\right)\leq C_T\Big(\e+\varrho^2_{b,\e}+\e \varrho_{\sigma,\e}^2+\e \varrho_{G,\e}^2\Big),\ \ \forall \e\in(0,\e_0]
\end{align}
where $\varrho_{b,\epsilon}$ and $\varrho_{\sigma,\epsilon}$ are the constants given in (A3), and $\varrho_{G,\e}$ in (B1).
\end{lemma}
\begin{proof}
By  It\^o's formula,
\begin{align*}
&|X^\e(t)-X^0(t)|^2\cr
=&2\int_0^t\<b_\epsilon(s,X^\e(s),Law(X^\e(s)))-b(s,X^0(s),Law(X^0(s))),X^\e(s)-X^0(s)\>\dif s\cr
&+2\sqrt{\e}\int_0^t\<\sigma_\epsilon(s,X^\e(s),Law(X^\e(s))),X^\e(s)-X^0(s)\>\dif W(s)\cr
&+2\e\int_0^t\!\!\int_Z\<G_\epsilon(s,X^\e(s-),Law(X^\e(s)),z),X^\e(s-)-X^0(s-)\>\widetilde{N}^{\e^{-1}}(\dif z,\dif s)\cr
&+\e\int_0^t\|\sigma_\epsilon(s,X^\e(s),Law(X^\e(s)))\|^2_{{\mathcal{L}_2}}\dif s\cr
&+\e^2\int_0^t\!\!\int_Z|G_\epsilon(s,X^\e(s-),Law(X^\e(s)),z)|^2N^{\e^{-1}}(\dif z,\dif s)\cr
=:&
 I_1(t)+I_2(t)+I_3(t)+I_4(t)+I_5(t).
\end{align*}

By (A1), (A3) and Remark \ref{rem W2}, we have
\begin{align}\label{eq I1}
     &I_1(t)\nonumber\\
\leq&
     2\int_0^t|\<b_\epsilon(s,X^\e(s),Law(X^\e(s)))-b(s,X^\e(s),Law(X^\e(s))),X^\e(s)-X^0(s)\>|\dif s\nonumber\\
     &+
     2\int_0^t\<b(s,X^\e(s),Law(X^\e(s)))-b(s,X^0(s),Law(X^\e(s))),X^\e(s)-X^0(s)\>\dif s\nonumber\\
    &+
    2\int_0^t|\<b(s,X^0(s),Law(X^\e(s)))-b(s,X^0(s),Law(X^0(s))),X^\e(s)-X^0(s)\>|\dif s\nonumber\\
\leq&
     2\varrho_{b,\epsilon}\int_0^T|X^\e(s)-X^0(s)|\dif s
     +
     2L\int_0^T|X^\e(s)-X^0(s)|^2\dif s\nonumber\\
    &+
    2L\int_0^T\mW_2(Law(X^\e(s)),Law(X^0(s)))|X^\e(s)-X^0(s)|\dif s\nonumber\\
\leq&
     (3L+1)\int_0^T|X^\e(s)-X^0(s)|^2\dif s
    +
    L\int_0^T\mW_2^2(Law(X^\e(s)),Law(X^0(s)))\dif s
    +
    \varrho^2_{b,\epsilon} T
    \nonumber\\
\leq&
     (3L+1)\int_0^T|X^\e(s)-X^0(s)|^2\dif s
    +
    L\int_0^T\mathbb{E}(|X^\e(s)-X^0(s)|^2)\dif s
    +
    \varrho^2_{b,\epsilon} T.
\end{align}
Hence,
\begin{eqnarray}\label{eq I1 1}
     \mathbb{E}(\sup_{t\in[0,T]}I_1(t))
\leq
     (4L+1)\mathbb{E}\int_0^T|X^\e(s)-X^0(s)|^2\dif s+\varrho^2_{b,\epsilon} T.
\end{eqnarray}
%Hence,
%\begin{align}\label{eq I1 1}
%&|X^\e(t)-X^0(t)|^2\cr
%\leq&I_2(t)+I_3(t)+I_4(t)+I_5(t)
%+(4L+1)\mathbb{E}\int_0^T|X^\e(s)-X^0(s)|^2\dif s+\varrho^2_{b,\epsilon} T.
%\end{align}
Also by (A1), (A3) and Remark \ref{rem W2},
\begin{align}\label{eq I4}
&  \mathbb{E}\left(\sup_{t\in[0,T]}|I_4(t)|\right)\nonumber\\
=&
     \e \mathbb{E}\int_0^T\|\sigma_\epsilon(s,X^\e(s),Law(X^\e(s)))\|^2_{\mathcal{L}_2}\dif s\cr
\leq&
     C\e \mathbb{E}\int_0^T\|\sigma_\epsilon(s,X^\e(s),Law(X^\e(s)))-\sigma(s,X^\e(s),Law(X^\e(s)))\|^2_{\mathcal{L}_2}\dif s\cr
      &+
     C\e \mathbb{E}\int_0^T\|\sigma(s,X^\e(s),Law(X^\e(s)))-\sigma(s,X^0(s),Law(X^0(s)))\|^2_{\mathcal{L}_2}\dif s\cr
     &+
     C\e \int_0^T\|\sigma(s,X^0(s),Law(X^0(s)))\|^2_{\mathcal{L}_2}\dif s\cr
\leq&
     C\e \varrho_{\sigma,\e}^2T
     +
     CL^2\e \mathbb{E}\int_0^T\left(|X^\e(s)-X^0(s)|^2+\mW_2^2(Law(X^\e(s)),Law(X^0(s)))\right)\dif s\cr
     &+
     C\e \int_0^T\|\sigma(s,X^0(s),Law(X^0(s)))\|^2_{\mathcal{L}_2}\dif s\\
\leq&
     C\e \varrho_{\sigma,\e}^2T
     +
     CL^2T\e \mathbb{E}\left(\sup\limits_{s\in[0,T]}|X^\e(s)-X^0(s)|^2\right)\nonumber\\
     &+
     C\e \int_0^T\|\sigma(s,X^0(s),Law(X^0(s)))\|^2_{\mathcal{L}_2}\dif s.\nonumber
\end{align}
By Burkholder-Davis-Gundy's inequality, Young's inequality and (\ref{eq I4}), we have
\begin{align}\label{eq I2}
&\mathbb{E}\left(\sup_{t\in[0,T]}|I_2(t)|\right)\nonumber\\
\leq&
    C\sqrt{\e}\mathbb{E}\left(\int_0^T\|\sigma_\e(s,X^\e(s),Law(X^\e(s)))\|^2_{\mathcal{L}_2}|X^\e(s)-X^0(s)|^2\dif s\right)^{\frac{1}{2}}\nonumber\\
\leq&
     \frac{1}{5}\mathbb{E}\left(\sup\limits_{s\in[0,T]}|X^\e(s)-X^0(s)|^2\right)
     +
     C\e \mathbb{E}\int_0^T\|\sigma_\e(s,X^\e(s),Law(X^\e(s)))\|^2_{\mathcal{L}_2}\dif s\\
\leq&
     (\frac{1}{5}+CL^2T\e)\mathbb{E}
     \left(\sup\limits_{s\in[0,T]}|X^\e(s)-X^0(s)|^2\right)+
     C\e \varrho_{\sigma,\e}^2T\nonumber\\
     &+
     C\e \int_0^T\|\sigma(s,X^0(s),Law(X^0(s)))\|^2_{\mathcal{L}_2}\dif s.\nonumber
\end{align}
It follows from Remark \ref{rem W2} and (B1) that
\begin{align}\label{eq I5}
&\mathbb{E}\left(\sup_{t\in[0,T]}|I_5(t)|\right)\nonumber\\
=&
    \e \mathbb{E}\left(\int_0^T\!\!\int_Z|G_\e(s,X^\e(s),Law(X^\e(s)),z)|^2\nu(dz)\dif s\right)\nonumber\\
\leq&
C\e \mathbb{E}\left(\int_0^T\!\!\int_Z|G_\e(s,X^\e(s),Law(X^\e(s)),z)-G(s,X^\e(s),Law(X^\e(s)),z)|^2
   \nu(\dif z)\dif s\right)\nonumber\\
   &+
   C\e \mathbb{E}\left(\int_0^T\!\!\int_Z|G(s,X^\e(s),Law(X^\e(s)),z)-G(s,X^0(s),Law(X^0(s)),z)|^2
   \nu(\dif z)\dif s\right)\nonumber\\
    &+
   C\e \int_0^T\!\!\int_Z|G(s,X^0(s),Law(X^0(s)),z)|^2\nu(\dif z)\dif s\nonumber\\
\leq&
   C\e \varrho_{G,\e}^2\int_ZL_3^2(z)\nu(dz) T
   +
   C\e \int_0^T\!\!\int_Z|G(s,X^0(s),Law(X^0(s)),z)|^2\nu(\dif z)\dif s\nonumber\\
   &+
   C\e \mathbb{E}\left(\int_0^T\!\!\int_ZL_1^2(z)\Big(|X^\e(s)-X^0(s)|^2
   +\mW_2^2(Law(X^\e(s)),Law(X^0(s)))\Big)\nu(\dif z)\dif s\right)\nonumber\\
\leq&
  C\e \varrho_{G,\e}^2\int_ZL_3^2(z)\nu(dz) T
   +C\e \int_0^T\!\!\int_Z|G(s,X^0(s),Law(X^0(s)),z)|^2\nu(\dif z)\dif s\nonumber\\
    &+C\e T\int_ZL_1^2(z)\nu(dz) \mathbb{E}\Big(\sup_{s\in[0,T]}\!\!|X^\e(s)-X^0(s)|^2\Big).
\end{align}

By Burkholder-Davis-Gundy's inequality, Young's inequality and (\ref{eq I5}), one can obtain
\begin{align}\label{eq I3}
&  \mathbb{E}\left(\sup_{t\in[0,T]}|I_3(t)|\right)\nonumber\\
\leq&
    C\e \mathbb{E}\left(\int_0^T\!\!\int_Z|G_\e(s,X^\e(s-),Law(X^\e(s)),z)|^2|X^\e(s-)-X^0(s-)|^2
    N^{\e^{-1}}(\dif z,\dif s)\right)^{\frac{1}{2}}\nonumber\\
\leq&
     \frac{1}{5}\mathbb{E}\left(\sup\limits_{s\in[0,T]}|X^\e(s)-X^0(s)|^2\right)
     +
    C\e \mathbb{E}\left(\int_0^T\!\!\int_Z|G_\e(s,X^\e(s),Law(X^\e(s)),z)|^2\nu(\dif z)\dif s)\right)\nonumber\\
\leq&
   (\frac{1}{5}+C\e T\int_ZL_1^2(z)\nu(\dif z)) \mathbb{E}\Big(\sup_{s\in[0,T]}\!\!|X^\e(s)-X^0(s)|^2\Big)\\
    &+
   C\e \int_0^T\!\!\int_Z|G(s,X^0(s),Law(X^0(s)),z)|^2\nu(\dif z)\dif s
   +
   C\e \varrho_{G,\e}^2\int_ZL_3^2(z)\nu(\dif z) T.\nonumber
\end{align}

Based on the above estimates, it holds
\begin{align}\label{eq cc}
&\Big(\frac{3}{5}-CL^2T\e-C\e T\int_ZL_1^2(z)\nu(\dif z)\Big)\mE\left(\sup\limits_{t\in[0,T]}|X^\e(t)-X^0(t)|^2\right)\nonumber\\
\leq&
(4L+1)\mathbb{E}\int_0^T|X^\e(s)-X^0(s)|^2\dif s
+
C\e \int_0^T\|\sigma(s,X^0(s),Law(X^0(s)))\|^2_{\mathcal{L}_2}\dif s\nonumber\\
&+
C\e \int_0^T\!\!\int_Z|G(s,X^0(s),Law(X^0(s)),z)|^2\nu(\dif z)\dif s\nonumber\\
&+
\varrho_{b,\e}^2T
+
C\e\varrho_{\sigma,\e}^2T
+
C\e \varrho_{G,\e}^2\int_ZL_3^2(z)\nu(dz) T.
\end{align}
By (A1), (B1) and using the fact that $\int_Z(L_1^2(z)+L_2^2(z))\nu(\dif z)<\infty$, we can prove that
\begin{align}\label{eq Ct}
\int_0^T\!\!\!\big(\|\sigma(s,\!X^0(s),\!Law(X^0(s)))\|^2_{\mathcal{L}_2}\!
+\!\!\int_Z\!|G(s,X^0(s),Law(X^0(s)),z)|^2\nu(\dif z)\big)\dif s
<\infty,
\end{align}
and there exists $\e_0>0$ small enough such that, for any $\e\in(0,\e_0]$,
\begin{eqnarray}\label{eq Ct LDP sat}
\frac{3}{5}-CL^2T\e-C\e T\int_ZL_1^2(z)\nu(\dif z)\geq \frac{1}{5}.
\end{eqnarray}
Hence, by (\ref{eq cc}), (\ref{eq Ct}), (\ref{eq Ct LDP sat}) and Gronwall's inequality, there exists some constant $C_T>0$ such that, for any $\e\in(0,\e_0]$,
\begin{align*}
\mE\left(\sup\limits_{t\in[0,T]}|X^\e(t)-X^0(t)|^2\right)\leq C_T(\e+\varrho_{b,\e}^2+\e \varrho_{\sigma,\e}^2+\e \varrho_{G,\e}^2),
\end{align*}
which is the desired result.
%Here $\int_ZL_3^2(z)\nu(\dif z)<\infty$ has been used.

%The proof of this lemma is complete.
\end{proof}

\vskip 0.3cm
Next we will verify (\textbf{LDP2}).
\vskip 0.3cm

\begin{proposition}\label{lem LDP 2}
For any given $m\in(0,\infty)$, let $\{u_\e=(\phi_\e,\psi_\e),~\e>0\}\subset \mathcal{S}^m_1\times \mathcal{S}^m_2$. Then, for the solution $Z^{u_\e}$ to (\ref{EQ4 LDP 2}),
\begin{align*}
\lim\limits_{\e\rightarrow0}\mE\left(\sup\limits_{t\in[0,T]}|Z^{u_\e}(t)-{\Gamma}^0(u_\e)(t)|^2\right)=0.
\end{align*}
\end{proposition}
\begin{proof}
 Let $Y^{u_\e}$ be the solution of (\ref{eq rate LDP 1}) with $u$ replaced by $u_\e$. Then  ${\Gamma}^0(u_\e)=Y^{u_\e}$.

By (\ref{EQ4 LDP 2}) and (\ref{eq rate LDP 1}), we have
\begin{align*}
&Z^{u_\e}(t)-Y^{u_\e}(t)\cr
=&\int_0^t\left(b_\e(s,Z^{u_\e}(s),Law(X^\e(s)))-b(s,Y^{u_\e}(s),Law(X^0(s)))\right)\dif s\cr
&+\sqrt{\e}\int_0^t\sigma_\e(s,Z^{u_\e}(s),Law(X^\e(s)))\dif W(s)\cr
&+\int_0^t\left(\sigma_\e(s,Z^{u_\e}(s),Law(X^\e(s)))-\sigma(s,Y^{u_\e}(s),Law(X^0(s)))\right)
\phi_\e(s)\dif s\cr
&+\int_0^t\!\!\int_Z\left(G_\e(s,Z^{u_\e}(s),Law(X^\e(s)),z)-G(s,Y^{u_\e}(s),Law(X^0(s)),z)\right)
(\psi_\e(s,z)-1)\nu(\dif z)\dif s\cr
&+\e\int_0^t\!\!\int_ZG_\e(s,Z^{u_\e}(s-),Law(X^\e(s)),z)\widetilde{N}^{\e^{-1}\psi_\e}(\dif z,\dif s).
\end{align*}
By It\^o's formula,
\begin{align}\label{eq 5.26}
&|Z^{u_\e}(t)-Y^{u_\e}(t)|^2\nonumber\\
=&
2\int_0^t\<Z^{u_\e}(s)-Y^{u_\e}(s), b_\e(s,Z^{u_\e}(s),Law(X^\e(s)))-b(s,Y^{u_\e}(s),Law(X^0(s)))\>\dif s\nonumber\\
&
+2\sqrt{\e}\int_0^t\<Z^{u_\e}(s)-Y^{u_\e}(s), \sigma_\e(s,Z^{u_\e}(s),Law(X^\e(s)))\dif W(s)\>\nonumber\\
&
+2\int_0^t\<Z^{u_\e}(s)-Y^{u_\e}(s),
\left(\sigma_\e(s,Z^{u_\e}(s),Law(X^\e(s)))-\sigma(s,Y^{u_\e}(s),Law(X^0(s)))\right)\phi_\e(s)\>
\dif s\nonumber\\
&
+2\int_0^t\!\!\int_Z\<Z^{u_\e}(s)-Y^{u_\e}(s), \nonumber\\
&\ \ \ \ \ \ \ \ \left(G_\e(s,Z^{u_\e}(s),Law(X^\e(s)),z)-G(s,Y^{u_\e}(s),Law(X^0(s)),z)\right)
(\psi_\e(s,z)-1)\>\nu(\dif z)\dif s\nonumber\\
&
+2\e\int_0^t\!\!\int_Z\<Z^{u_\e}(s-)-Y^{u_\e}(s-),
G_\e(s,Z^{u_\e}(s-),Law(X^\e(s)),z)\>\widetilde{N}^{\e^{-1}\psi_\e}(\dif z,\dif s)\nonumber\\
&
+\e\int_0^t\|\sigma_\e(s,Z^{u_\e}(s),Law(X^\e(s)))\|_{\mathcal{L}_2}^2\dif s\nonumber\\
&+\e^2\int_0^t\!\!\int_Z|G_\e(s,Z^{u_\e}(s-),Law(X^\e(s)),z)|^2N^{\e^{-1}\psi_\e}(\dif z,\dif s)\nonumber\\
=:&J_1(t)+J_2(t)+J_3(t)+J_4(t)+J_5(t)+J_6(t)+J_7(t).
\end{align}

By (B1), Remark \ref{rem W2} and (\ref{CK}),
\begin{align*}
&J_4(t)\\
\leq&
2\int_0^t\!\!\int_Z |Z^{u_\e}(s)-Y^{u_\e}(s)|
   |G_\e(s,Z^{u_\e}(s),Law(X^\e(s)),z)-G(s,Z^{u_\e}(s),Law(X^\e(s)),z)|\nonumber\\
               &\ \ \ \ \ \ \ \ \ \ \cdot|\psi_\e(s,z)-1|\nu(\dif z)\dif s\nonumber\\
&+
C\int_0^t\!\!\int_Z\left(|Z^{u_\e}(s)-Y^{u_\e}(s)|^2
+|Z^{u_\e}(s)-Y^{u_\e}(s)|\mW_2(Law(X^\e(s)),Law(X^0(s)))\right)\nonumber\\
 &\ \ \ \ \ \ \ \ \ \ \cdot L_1(z)|\psi_\e(s,z)-1|\nu(\dif z)\dif s\cr
\leq&
C\int_0^t\!\!\int_Z|Z^{u_\e}(s)-Y^{u_\e}(s)|^2(L_1(z)+L_3(z))|\psi_\e(s,z)-1|\nu(\dif z)\dif s\\
&+
C\mE\Big(\sup_{s\in[0,T]}|X^\e(s)-X^0(s)|^2\Big)\sup_{\varphi\in S^m_2}\int_0^T\!\!\int_ZL_1(z)|\varphi(s,z)-1|\nu(\dif z)\dif s\\
&+
\varrho_{G,\e}^2\sup_{\varphi\in S^m_2}\int_0^T\!\!\int_ZL_3(z)|\varphi(s,z)-1|\nu(\dif z)\dif s\\
\leq&
C\int_0^t\!\!\int_Z|Z^{u_\e}(s)-Y^{u_\e}(s)|^2(L_1(z)+L_3(z))|\psi_\e(s,z)-1|\nu(\dif z)\dif s\\
&+
C\Big(\mE\Big(\sup_{s\in[0,T]}|X^\e(s)-X^0(s)|^2\Big)
+
\varrho_{G,\e}^2\Big).
\end{align*}

Set
$$
J:=\sup_{t\in[0,T]}\Big(J_1(t)+J_2(t)+J_3(t)+J_5(t)+J_6(t)+J_7(t)\Big).
$$
Plugging the above inequality into (\ref{eq 5.26}), by Gronwall's inequality and (\ref{CK}), we arrive at
\begin{align}\label{eq 5.27}
&\sup_{t\in[0,T]}|Z^{u_\e}(t)-Y^{u_\e}(t)|^2\nonumber\\
\leq&
C\exp\{C\sup\limits_{\varphi\in S^m_2}\int_0^T\!\!\int_Z(L_1(z)+L_3(z))|\varphi(s,z)-1|)\nu(\dif z)\dif s\}\nonumber\\
&\times\Big(
    \mE\big(\sup_{s\in[0,T]}|X^\e(s)-X^0(s)|^2\big)
   +
   \varrho_{G,\e}^2
   +
   J
   \Big)\nonumber\\
\leq&
C\Big(\mE\Big(\sup_{s\in[0,T]}|X^\e(s)-X^0(s)|^2\Big)+\varrho_{G,\e}^2+J\Big).
\end{align}

%To deduce the above last inequality, we have used (\ref{CK}).
%Hence, Lemma \ref{yey0} implies that

Using (A1), Lemma \ref{yey0}, (A3), (B1) and Burkholder-Davis-Gundy's inequality, it follows from (\ref{eq 5.27}) that
\begin{align}\label{eq zz}
&\mE\left(\sup\limits_{t\in[0,T]}|Z^{u_\e}(t)-Y^{u_\e}(t)|^2\right)\nonumber\\
\leq& C(\e+\varrho_{b,\e}^2+\e\varrho_{\sigma,\e}^2+\e\varrho_{G,\e}^2+\varrho_{G,\e}^2)
+
C\varrho_{b,\e}\mE\int_0^T|Z^{u_\e}(s)-Y^{u_\e}(s)|\dif s\nonumber\\
&+
C\mE\int_0^T\left(|Z^{u_\e}(s)-Y^{u_\e}(s)|^2+ |Z^{u_\e}(s)-Y^{u_\e}(s)|\mW_2(Law(X^\e(s)),Law(X^0(s)))\right)\dif s\nonumber\\
&
+C\sqrt{\e}\mE\left(\int_0^T|Z^{u_\e}(s)-Y^{u_\e}(s)|^2
\|\sigma_\e(s,Z^{u_\e}(s),Law(X^\e(s)))\|_{\mathcal{L}_2}^2\dif s\right)^{\frac{1}{2}}\nonumber\\
&
+
C\varrho_{\sigma,\e}\mE\int_0^T|Z^{u_\e}(s)-Y^{u_\e}(s)||\phi_\e(s)|\dif s\nonumber\\
&
+C\mE\int_0^T|Z^{u_\e}(s)-Y^{u_\e}(s)|
\|\sigma(s,Z^{u_\e}(s),Law(X^\e(s)))-\sigma(s,Y^{u_\e}(s),Law(X^0(s)))\|_{\mathcal{L}_2}
|\phi_\e(s)|\dif s\nonumber\\
&
+C\e\mE\left(\int_0^T\!\!\int_Z|Z^{u_\e}(s-)-Y^{u_\e}(s-)|^2|
G_\e(s,Z^{u_\e}(s-),Law(X^\e(s)),z)|^2N^{\e^{-1}\psi_\e}(\dif z,\dif s)\right)^{\frac{1}{2}}\nonumber\\
&
+
C\e\mE\int_0^T\|\sigma_\e(s,Z^{u_\e}(s),Law(X^\e(s)))\|_{\mathcal{L}_2}^2\dif s\cr
&+
C\e\mE\int_0^T\!\!\int_Z|G_\e(s,Z^{u_\e}(s),Law(X^\e(s)),z)|^2|\psi_\e(s,z)|\nu(\dif z)\dif s\nonumber\\
:=&C(\e+\varrho_{b,\e}^2+\e\varrho_{\sigma,\e}^2+\e\varrho_{G,\e}^2+\varrho_{G,\e}^2)\nonumber\\
&+I_1+I_2+I_3+I_4+I_5+I_6+I_7+I_8, \ \ \ \forall \e\in(0,\e_0],
\end{align}
where $\e_0$ is the constant appearing in Lemma \ref{yey0}.

\vskip 0.2cm
In the sequel, let $\epsilon\in(0,\e_0]$.
\vskip 0.2cm

Keeping in mind that $\phi_\e\in\mathcal{S}^m_1$, we have
\begin{align}\label{eq LDP I1 I4}
I_1+I_4
\leq&
C\int_0^T\mE\Big(\sup\limits_{r\in[0,s]}|Z^{u_\e}(r)-Y^{u_\e}(r)|^2\Big)\dif s
+
C\varrho_{b,\e}^2T+C\varrho_{\sigma,\e}^2\mE\int_0^T|\phi_\e(s)|^2\dif s\nonumber\\
\leq&
C\int_0^T\mE\Big(\sup\limits_{r\in[0,s]}|Z^{u_\e}(r)-Y^{u_\e}(r)|^2\Big)\dif s
+
C\varrho_{b,\e}^2T+C\varrho_{\sigma,\e}^2m.
\end{align}
By Young's inequality and Lemma \ref{yey0},% for any $\e\in(0,\e_0]$,
\begin{align}\label{eq E I1}
I_2\leq&C\int_0^T\mE\Big(\sup\limits_{r\in[0,s]}|Z^{u_\e}(r)-Y^{u_\e}(r)|^2\Big)\dif s
+C\int_0^T\mW_2^2(Law(X^\e(s)),Law(X^0(s)))\dif s\nonumber\\
\leq&C\int_0^T\mE\Big(\sup\limits_{r\in[0,s]}|Z^{u_\e}(r)-Y^{u_\e}(r)|^2\Big)\dif s
+C\mE\left(\sup\limits_{s\in[0,T]}|X^\e(s)-X^0(s)|^2\right)\nonumber\\
\leq&C\int_0^T\mE\Big(\sup\limits_{r\in[0,s]}|Z^{u_\e}(r)-Y^{u_\e}(r)|^2\Big)\dif s+C(\e+\varrho_{b,\e}^2+\e\varrho_{\sigma,\e}^2+\e\varrho_{G,\e}^2).
\end{align}
Using (A3) and (\ref{eq ll}), we have
\begin{align}\label{eq E I25}
&I_3+I_7\cr
\leq&
\frac{1}{10}\mE\Big(\sup\limits_{t\in[0,T]}|Z^{u_\e}(t)-Y^{u_\e}(t)|^2\Big)+
C\e\mE\int_0^T\|\sigma_\e(s,Z^{u_\e}(s),Law(X^\e(s)))\|_{\mathcal{L}_2}^2\dif s\cr
\leq&
\frac{1}{10}\mE\Big(\sup\limits_{t\in[0,T]}|Z^{u_\e}(t)-Y^{u_\e}(t)|^2\Big)\cr
&+
C\e\mE\int_0^T\|\sigma_\e(s,Z^{u_\e}(s),Law(X^\e(s)))-\sigma(s,Z^{u_\e}(s),Law(X^\e(s)))\|_{\mathcal{L}_2}^2\dif s\cr
&+
C\e\int_0^T\mE\|\sigma(s,Z^{u_\e}(s),Law(X^\e(s)))
-\sigma(s,Y^{u_\e}(s),Law(X^0(s)))\|_{\mathcal{L}_2}^2\dif s\cr
&+
C\e\int_0^T\mE\|\sigma(s,Y^{u_\e}(s),Law(X^0(s)))-\sigma(s,0,\delta_0)\|_{\mathcal{L}_2}^2\dif s
+C\e\int_0^T\|\sigma(s,0,\delta_0)\|_{\mathcal{L}_2}^2\dif s\cr
\leq&
\frac{1}{10}\mE\Big(\sup\limits_{t\in[0,T]}|Z^{u_\e}(t)-Y^{u_\e}(t)|^2\Big)
+
C\e\varrho_{\sigma,\e}^2T
+
C\e\int_0^T\mE\Big(|Z^{u_\e}(s)-Y^{u_\e}(s)|^2\Big)\dif s\cr
&+
C\e\Big\{\mE\Big(\sup\limits_{s\in[0,T]}|X^\e(s)-X^0(s)|^2\Big)
+\sup\limits_{u\in S^m_1\times S^m_2}\Big(\sup\limits_{t\in[0,T]}|Y^u(t)|^2\Big)\cr
&\ \ \ \ \ \ \ \ \ \ +\sup\limits_{t\in[0,T]}|X^0(t)|^2+\int_0^T\|\sigma(s,0,\delta_0\|_{\mathcal{L}_2}^2\dif s\Big\}\nonumber\\
\leq&
(\frac{1}{10}+C\e)\mE\Big(\sup\limits_{t\in[0,T]}|Z^{u_\e}(t)-Y^{u_\e}(t)|^2\Big)\!
+\!
C(\e+\e^2+\varrho_{b,\e}^2+\e\varrho_{\sigma,\e}^2+\e\varrho_{G,\e}^2).
\end{align}

Again by (A1), Young's inequality and Lemma \ref{yey0}, using the fact that  $\phi_\e\in\mathcal{S}^m_1$, we have
\begin{align}\label{eq E I3}
I_5
\leq&C\mE\Bigg\{\sup\limits_{t\in[0,T]}|Z^{u_\e}(t)-Y^{u_\e}(t)|
\Big(\int_0^T|\phi_\e(s)|^2\dif s\Big)^{\frac{1}{2}}\Big)
\nonumber\\
&\Big(\int_0^T
\|\sigma(s,Z^{u_\e}(s),Law(X^\e(s)))-\sigma(s,Y^{u_\e}(s),Law(X^0(s)))\|_{\mathcal{L}_2}^2\dif s\Big)^{\frac{1}{2}}\Bigg\}\nonumber\\
\leq&
  \frac{1}{10}\mE\Big(\sup\limits_{t\in[0,T]}|Z^{u_\e}(t)-Y^{u_\e}(t)|^2\Big)
+C\int_0^T\mE|Z^{u_\e}(s)-Y^{u_\e}(s)|^2\dif s\nonumber\\
&+C\int_0^T\mW_2^2(Law(X^\e(s)),Law(X^0(s)))\dif s\nonumber\\
\leq&
  \frac{1}{10}\mE\Big(\sup\limits_{t\in[0,T]}|Z^{u_\e}(t)-Y^{u_\e}(t)|^2\Big)
+
C\int_0^T\mE\Big(\sup\limits_{r\in[0,s]}|Z^{u_\e}(r)-Y^{u_\e}(r)|^2\Big)\dif s\nonumber\\
&+
C(\e+\varrho_{b,\e}^2+\e\varrho_{\sigma,\e}^2+\e\varrho_{G,\e}^2).
\end{align}

By H\"older's inequality, (B1), (\ref{eq ll}) and Lemma \ref{yey0},%  for any $\e\in(0,\e_0]$,
{\small
\begin{align}\label{eq E I46}
&I_6+I_8\cr
\leq&
\frac{1}{10}\mE\Big(\sup\limits_{t\in[0,T]}|Z^{u_\e}(t)-Y^{u_\e}(t)|^2\Big)
+
C\e\mE\int_0^T\!\!\int_Z|G_\e(s,Z^{u_\e}(s),Law(X^\e(s)),z)|^2|\psi_\e(s,z)|\nu(\dif z)\dif s\cr
\leq&
\frac{1}{10}\mE\Big(\sup\limits_{t\in[0,T]}|Z^{u_\e}(t)-Y^{u_\e}(t)|^2\Big)\cr
&+
C\e\mE\int_0^T\!\!\int_Z|G_\e(s,Z^{u_\e}(s),Law(X^\e(s)),z)-G(s,Z^{u_\e}(s),Law(X^\e(s)),z)|^2|\psi_\e(s,z)|\nu(\dif z)\dif s\cr
&+
C\e\mE\int_0^T\!\!\int_Z|G(s,Z^{u_\e}(s),Law(X^\e(s)),z)-G(s,Y^{u_\e}(s),Law(X^0(s)),z)|^2|\psi_\e(s,z)|\nu(\dif z)\dif s\cr
&+
C\e\mE\int_0^T\!\!\int_Z|G(s,Y^{u_\e}(s),Law(X^0(s)),z)-G(s,0,\delta_0,z)|^2|\psi_\e(s,z)|\nu(\dif z)\dif s\cr
&+
  C\e\mE\int_0^T\!\!\int_Z|G(s,0,\delta_0,z)|^2|\psi_\e(s,z)|\nu(\dif z)\dif s\cr
\leq&
\frac{1}{10}\mE\Big(\sup\limits_{t\in[0,T]}|Z^{u_\e}(t)-Y^{u_\e}(t)|^2\Big)
+
C\e \varrho_{G,\e}^2\Theta_m
\cr
&+
C\e \Theta_m
\Bigg\{
    \mE\Big(\sup\limits_{t\in[0,T]}|Z^{u_\e}(t)-Y^{u_\e}(t)|^2\Big)
   +
   \mE\Big(\sup\limits_{t\in[0,T]}|X^\e(t)-X^0(t)|^2\Big)\cr
&+
  \sup\limits_{u\in S^m_1\times S^m_2}\Big(\sup\limits_{t\in[0,T]}|Y^u(t)|^2\Big)
+\sup\limits_{t\in[0,T]}|X^0(t)|^2
+
1
\Bigg\}\cr
\leq&(\frac{1}{10}+C\e \Theta_m)\mE\Big(\sup\limits_{t\in[0,T]}|Z^{u_\e}(t)-Y^{u_\e}(t)|^2\Big)
+
C\e \Theta_m(1+\e+\varrho_{b,\e}^2+\e\varrho_{\sigma,\e}^2+\e\varrho_{G,\e}^2+\varrho_{G,\e}^2).
\end{align}
}
Here
$$
\Theta_m:=\sup_{\varphi\in S^m_2}\int_0^T\!\!\int_{{Z}}(L^2_1(z)+L^2_2(z)+L^2_3(z))(\varphi(s,z)+1)\nu(\dif z)\,\dif s<\infty.
$$
See (3.3)  in \cite[Lemma 3.4]{Budhiraja-Chen-Dupuis}.
Combining (\ref{eq zz})--(\ref{eq E I46}) together, we arrive at
\begin{align*}
&(\frac{7}{10}-C\e-C\e\Theta_m)\mE\Big(\sup\limits_{t\in[0,T]}|Z^{u_\e}(t)-Y^{u_\e}(t)|^2\Big)\\
\leq&
C\e+C\int_0^T\mE\Big(\sup\limits_{r\in[0,s]}|Z^{u_\e}(r)-Y^{u_\e}(r)|^2\Big)\dif s+C(\varrho_{b,\e}^2+\varrho_{\sigma,\e}^2+\varrho_{G,\e}^2).
\end{align*}
Hence, by Gronwall's inequality we have
\begin{eqnarray*}
\lim_{\e\rightarrow0}\mE\Big(\sup\limits_{t\in[0,T]}|Z^{u_\e}(t)-Y^{u_\e}(t)|^2\Big)=0,
\end{eqnarray*}
which completes the proof.
\end{proof}

\subsection{Moderate deviation principle}

For any $t\in[0,T]$ and $\mu\in \mathcal{P}_2$,
 let $b'_2(t,x,\mu)$ denote the  derivative of $b(t,x,\mu)$ with respect to the variable $x$. In order to obtain the MDP for the solution $\{X^\epsilon,\epsilon>0\}$ to (\ref{1-1}), we give the following assumption.
\vskip 0.2cm

(B2) There are $L', q'\geq0$ such that for each $x, x'\in\mR^d$,
\begin{align}\label{gbg}
|b'_2(s,x,Law(X^0(s)))-b'_2(s,x',Law(X^0(s)))|\leq L'(1+|x|^{q'}+|x'|^{q'})|x-x'|,
\end{align}
and
$\int_0^T|b'_2(t,X^0(t),Law(X^0(t)))|\dif t<\infty$.

Before stating the main result, we first present the following result.
\bp\label{lem 3.8}
Assume that (A1), (B1) and (B2) hold. Then for any fixed $m\in(0,\infty)$ and $u=(\phi,\varphi)\in {S}^m_1\times B_2(m)$, there is a unique solution $K^u=\{K^u(t),t\in[0,T]\}\in C([0,T],\mathbb{R}^d)$ to the following equation,
\begin{equation}\label{Y(U)}
\begin{cases}
&\dif K^u(t)=b'_2(t,X^0(t),Law(X^0(t)))K^u(t)\dif t
+\sigma(t,X^0(t),Law(X^0(t)))\phi(t)\dif t\\
&\ \ \ \ \ \ \ \ \ \ \ \ \ \ \ \ \ \ +\int_ZG(t,X^0(t),Law(X^0(t)),z)\varphi(t,z)\nu(\dif z)\dif t\\
&K^u(0)=0.
\end{cases}
\end{equation}
Moreover,
\begin{eqnarray}\label{eq E M 1}
\sup_{u\in {S}^m_1\times B_2(m)}\sup_{t\in[0,T]}|K^u(t)|=:\Xi_m<\infty.
\end{eqnarray}

\ep
\begin{proof}

By using (\ref{eq Ct}) and the fact that $u\in {S}^m_1\times B_2(m)$, we have
\begin{align}\label{eq 11}
&\int_0^T|\sigma(t,X^0(t),Law(X^0(t)))\phi(t)|\dif t\cr
\leq&\left(\int_0^T|\sigma(t,X^0(t),Law(X^0(t)))|^2\dif t\right)^{\frac{1}{2}}\left(\int_0^T|\phi(t)|^2\dif t\right)^{\frac{1}{2}}\cr
\leq&\left(\int_0^T|\sigma(t,X^0(t),Law(X^0(t)))|^2\dif t\right)^{\frac{1}{2}}\left(2m\right)^{\frac{1}{2}}\cr
<&\infty,
%\leq&K\left(\int_0^T\left(|X^0(t)|^2+\mW_2(Law(X^0(t)),\delta_0)^2+|\sigma(t,0,\delta_0)|^2
%\right)\dif t\right)^{\frac{1}{2}}<\infty,
\end{align}
and
\begin{align}\label{eq 12}
&\int_0^T\!\!\int_Z|G(t,X^0(t),Law(X^0(t)),z)\varphi(t,z)|\nu(\dif z)\dif t\cr
\leq&\left(\int_0^T\!\!\int_Z|G(t,X^0(t),Law(X^0(t)),z)|^2\nu(\dif z)\dif t\right)^{\frac{1}{2}}\left(\int_0^T\!\!\int_Z|\varphi(t,z)|^2\nu(\dif z)\dif t\right)^{\frac{1}{2}}\cr
\leq&\left(\int_0^T\!\!\int_Z|G(t,X^0(t),Law(X^0(t)),z)|^2\nu(\dif z)\dif t\right)^{\frac{1}{2}}m\cr
<&\infty.
%\leq&K\left(\int_0^T\left(|X^0(t)|^2+\mW_2(Law(X^0(t)),\delta_0)^2+\int_Z|G(t,0,\delta_0,z)|^2\nu(\dif z)\right)\dif t\right)^{\frac{1}{2}}<\infty.
\end{align}
With these two estimates above it is standard to show that  the linear equation (\ref{Y(U)}) has a unique solution $\{K^u(t)\}_{t\in[0, T]}$. The estimate (\ref{eq E M 1})follows by using Gronwall's inequaity.

\end{proof}

Recall $a(\e)$ in (\ref{eq a ep}).
For any $\e\in(0,1)$, define
\begin{align*}
M^\e(t)=\frac{1}{a(\e)}(X^\e(t)-X^0(t)),\ \ t\in[0,T].
\end{align*}
Due to  (\ref{1-1}) and (\ref{eq 1.2}), $M^\e$ satisfies
\begin{align}
M^\e(t)
=&\frac{1}{a(\e)}\int_0^t\left(b_\e(s,a(\e)M^\e(s)+X^0(s),Law(X^\e(s)))
-b(s,X^0(s),Law(X^0(s)))\right)\dif s\nonumber\\
&+\frac{\sqrt{\e}}{a(\e)}\int_0^t\sigma_\e(s,a(\e)M^\e(s)+X^0(s),Law(X^\e(s)))\dif W(s)\nonumber\\
&+\frac{\e}{a(\e)}\int_0^t\!\!\int_ZG_\e(s,a(\e)M^\e(s-)+X^0(s-),Law(X^\e(s)),z)
\widetilde{N}^{\e^{-1}}(\dif z,\dif s).
\end{align}

We introduce the following assumption.

\vskip 0.5cm

(B3) $\lim_{\epsilon\rightarrow 0}\frac{\varrho_{b,\e}}{a(\e)}=0$, where $\varrho_{b,\e}$ is given in (A3).

%{\red give a remark about why we need this technical assumption???}
\vskip 0.3cm
Now we state the main result in this subsection.
\begin{theorem}\label{Th ex 2}
Assume that (A0), (A1), (A3), (A4), (B1), (B2) and (B3) hold. Then $\{M^\epsilon,~\e>0\}$  satisfies a LDP on $D([0,T],\mathbb{R}^d)$ with speed $\epsilon/a^2(\e)$ and the good
rate function $I$ given by for any $g\in D([0,T],\mathbb{R}^d)$
\begin{align*}
&I(g):=&\inf_{
    \{u=(\phi,\varphi)\in L^2([0,T],\mathbb{R}^d)\times L_2(\nu_T),K^u=g\}
           }
    \Big\{
      \frac{1}{2}\int_0^T\!\!|\phi(s)|^2\dif s
      +\frac{1}{2}\int_0^T\!\!\!\int_Z|\varphi(s,z)|^2\nu(\dif z)\dif s
    \Big\},
\end{align*}
where for $u=(\phi,\varphi)\in L^2([0,T],\mathbb{R}^d)\times L_2(\nu_T)$, $K^u$ is the unique solution of (\ref{Y(U)}).
Here we use the convention that the infimum of an empty set is $\infty$.
\end{theorem}

\begin{proof}
%We will apply the criterions in Theorem \ref{TH MDP 2} to prove this theorem.

By Proposition \ref{lem 3.8}, we can define a map
\begin{eqnarray}\label{def R}
 {\Upsilon}^0:L^2([0,T],\mathbb{R}^d)\times L_2(\nu_T)\ni u=(\phi,\varphi) \mapsto K^u\in D([0,T],\mathbb{R}^d),
\end{eqnarray}
where $K^u$ is the unique solution of (\ref{Y(U)}).
%Then we can define a map $\Upsilon^0:C([0,T],H)\times L_2(\nu_T)\rightarrow\mD$ such that, for any $u=(\phi,\varphi)\in L^2([0,T],\mathbb{R}^d)\times L_2(\nu_T)$,
%$$
%\Upsilon^0(\int_0^\cdot\phi(s)ds,\varphi):= {\Upsilon}^0(u).
%$$

For any $\epsilon>0$, $m\in(0,\infty)$ and $u_\e=(\phi_\e,\psi_\e)\in\mathcal{S}^m_1\times\mathcal{S}^m_{+,\e}$, recall that $\{M^{u_\e}(t)\}_{t\in[0, T]}$ (see (\ref{eq MDP 1-second}))
%{\footnote{In fact, the solution is unique. Since the uniqueness to (\ref{F-EQU}) is not required(see Remark \ref{rem 4} and Theorem \ref{TH MDP 2}), we omit the proof of the uniqueness to (\ref{F-EQU}) here. }}
 is the solution to the following SDE
\begin{equation}\label{F-EQU}
\begin{cases}
\dif M^{u_\e}(s)=&\!\!\!\frac{1}{a(\e)}\left(b_\e(s,a(\e)M^{u_\e}(s)+X^0(s),Law(X^\e(s)))
-b(s,X^0(s),Law(X^0(s)))\right)\dif s\cr
&+
\frac{\sqrt{\e}}{a(\e)}\sigma_\e(s,a(\e)M^{u_\e}(s)+X^0(s),Law(X^\e(s)))\dif W(s)\cr
&+\sigma_\e(s,a(\e)M^{u_\e}(s)+X^0(s),Law(X^\e(s)))\phi_\e(s)\dif s\\
&+
\frac{\e}{a(\e)}\int_ZG_\e(s,a(\e)M^{u_\e}(s-)+X^0(s-),Law(X^\e(s)),z)
\widetilde{N}^{\e^{-1}\psi_\e}(\dif z,\dif s)\cr
&+
\frac{1}{a(\e)}\int_ZG_\e(s,a(\e)M^{u_\e}(s)+X^0(s),Law(X^\e(s)),z)
(\psi_\e(s,z)-1)\nu(\dif z)\dif s,\\
M^{u_\e}(0)=0.
\end{cases}
\end{equation}
%with initial data $Z^{u_\e}(0)=h$ and $X^\epsilon$ is the solution to (\ref{1-1}).

According to Theorem \ref{TH MDP 2},
it is sufficient to verify the following two claims:
\begin{itemize}
  \item[(\textbf{MDP1})] For any given $m\in(0,\infty)$, let $u_n=(\phi_n,\varphi_n),~n\in\mathbb{N},~u=(\phi,\psi)\in S^m_1\times B_2(m)$ be such that
$u_n\rightarrow u$ in $S^m_1\times B_2(m)$ as $n\rightarrow\infty$. Then
$$
\lim_{n\rightarrow\infty}\sup_{t\in[0,T]}| {\Upsilon}^0(u_n)(t)- {\Upsilon}^0(u)(t)|=0.
$$
  \item[(\textbf{MDP2})] For any given $m\in(0,\infty)$, let $\{u_\e=(\phi_\e,\psi_\e),~\e>0\}\subset \mathcal{S}^m_1\times \mathcal{S}^m_{+,\e}$, and
  for some $\beta\in(0,1]$, $\varphi_\epsilon1_{\{|\varphi_\epsilon|\leq \beta/a(\epsilon)\}}\in B_2(\sqrt{m\kappa_2(1)})$
  where $\varphi_\epsilon=(\psi_\epsilon-1)/a(\epsilon)$. Set
\begin{eqnarray}\label{eq uuu}
\widetilde{u}_\e:=(\phi_\epsilon,\varphi_\epsilon1_{\{|\varphi_\epsilon|\leq \beta/a(\epsilon)\}}).
\end{eqnarray}
 Then for any  $\varpi>0$,
 $$\lim_{\e\rightarrow0}P\Big(\sup\limits_{t\in[0,T]}|M^{u_\e}(t)- {\Upsilon}^0(\widetilde{u}_\e)(t)|>\varpi)=0.$$
\end{itemize}

The verification of (\textbf{MDP1}) and (\textbf{MDP2}) will be given in Propositions \ref{YUm-YU} and \ref{Lemma 6.7} respectively.
%(\textbf{MDP2}) will be established in Proposition . The proof of Theorem \ref{Th ex 2} will be  complete after proving these two propositions.
\end{proof}

\vskip 0.3cm
Next proposition is the  verification of (\textbf{MDP1}).
\vskip 0.2cm

\bp\label{YUm-YU}
For any given $m\in(0,\infty)$, let $u_n=(\phi_n,\varphi_n),~n\in\mathbb{N},~u=(\phi,\psi)\in S^m_1\times B_2(m)$ be such that
$u_n\rightarrow u$ in $S^m_1\times B_2(m)$ as $n\rightarrow\infty$. Then
$$
\lim_{n\rightarrow\infty}\sup_{t\in[0,T]}| {\Upsilon}^0(u_n)(t)- {\Upsilon}^0(u)(t)|=0.
$$
\ep
\begin{proof}
Recall that $K^u={\Upsilon}^0(u)$,$K^{u_n}={\Upsilon}^0(u_n)$ are the corresponding solutions to (\ref{Y(U)}). We need
to prove the following result:
$$
\lim_{n\rightarrow\infty}\sup_{t\in[0,T]}|K^{u_n}(t)-K^u(t)|=0.
$$

The proof is similar to that of Proposition \ref{Yu-con} and we just give a sketch here.
We first show that  $\{K^{u_n}\}_{n\geq1}$ is pre-compact in $C([0,T],\mR^d)$.
%For any $m\geq1$ , note that $\Gamma^0(u_m)$ is the solution to the equation:
%\begin{align}\label{Yum}
%Y^{u_m}_t=&y+\int_0^tb(s,Y^{u_m}_s,\delta_{Y^0_s})\dif s+\int_0^t\sigma(s,Y^{u_m}_s,\delta_{Y^0_s})h_m(s)\dif s\cr
%&+\int_0^t\!\!\int_ZG(s,Y^{u_m}_s,\delta_{Y^0_s},z)(\varphi_m(s,z)-1)\nu(\dif z)\dif s,\ \ t\in[0,T].
%\end{align}
 (\ref{eq E M 1}) implies that $\{K^{u_n}\}_{n\geq1}$ is uniformly bounded, i.e.
\begin{align}\label{ubm MDP}
C_m:=\sup\limits_{n\geq1}\sup\limits_{t\in[0,T]}|K^{u_n}(t)|<\infty.
\end{align}
\vskip 0.3cm
For any $s, t\in[0,T]$ with $s<t$, by (\ref{eq 11}), (\ref{eq 12}) and (\ref{ubm MDP}),
\begin{align*}
&|K^{u_n}(t)-K^{u_n}(s)|\cr
\leq&C_m\int_s^t|b'_2(r,X^0(r),Law(X^0(r)))|\dif r
+\left(2m\right)^{\frac{1}{2}}\left(\int_s^t|\sigma(r,X^0(r),Law(X^0(r)))|^2\dif r\right)^{\frac{1}{2}}\cr
&+m\left(\int_s^t\!\!\int_Z|G(r,X^0(r),Law(X^0(r)),z)|^2\nu(\dif z)\dif r\right)^{\frac{1}{2}}.
\end{align*}
This, together with (\ref{eq Ct}) and (B2), implies that $\{K^{u_n}\}_{n\geq1}$ is equi-continuous  in $C([0,T],\mR^d)$.

\vskip 0.2cm
Hence, $\{K^{u_n}\}_{n\geq1}$ is pre-compact in $C([0,T],\mR^d)$.

\vskip 0.3cm
Let $\widetilde{K}$ be any limit of some subsequence of $\{K^{u_n}\}_{n\geq1}$ in $C([0,T],\mR^d)$. Using the  similar arguments as in the proof of Proposition \ref{Yu-con}, we can show $\widetilde{K}=K^u$, which completes the proof.

\end{proof}

%Given $m\in(0,\infty)$, let $\{u_\e=(\phi_\epsilon,\psi_\epsilon)\}_{\epsilon>0}$ be such that for every $\epsilon>0$,
%  $(\phi_\epsilon,\psi_\epsilon)\in \mathcal{S}^m_1\times\mathcal{S}^m_{+,\epsilon}$, and
%  for some $\beta\in(0,1]$, $\varphi_\epsilon1_{\{|\varphi_\epsilon|\leq \beta/a(\epsilon)\}}\in B_2(\sqrt{m\kappa_2(1)})$
%  where $\varphi_\epsilon=(\psi_\epsilon-1)/a(\epsilon)$. Set
%\begin{eqnarray}\label{eq uuu}
%\tilde{u}_\e:=(\phi_\epsilon,\varphi_\epsilon1_{\{|\varphi_\epsilon|\leq \beta/a(\epsilon)\}}).
%\end{eqnarray}
%Then
%$\tilde{u}_\e\in S^m_1\times B_2(\sqrt{m\kappa_2(1)})\ P$-a.s..

%Now consider the following equation:

\vskip 0.2cm
In order to  verify (\textbf{MDP2}), we need the following two lemmas.
\vskip 0.2cm

The following lemma is taken from  Lemma 4.2, Lemma 4.3 and Lemma 4.7 in \cite{Budhiraja-Dupuis-Ganguly}
\footnote{Note: The reference \cite{BDG} is the published version of \cite{Budhiraja-Dupuis-Ganguly}, and the paper \cite{BDG} considered a little more general assumptions than those in \cite{Budhiraja-Dupuis-Ganguly}, see (2.13) in \cite{BDG} and (2.13) in \cite{Budhiraja-Dupuis-Ganguly}. Hence some of a priori estimates are different between \cite{Budhiraja-Dupuis-Ganguly} and \cite{BDG}, for example,  Lemma 4.2, Lemma 4.3, Lemma 4.7 and Lemma 4.8 of \cite{BDG} are different with Lemma 4.2, Lemma 4.3, Lemma 4.7 and Lemma 4.8 of \cite{Budhiraja-Dupuis-Ganguly}. In this paper, we use the same assumption with those in \cite{Budhiraja-Dupuis-Ganguly}, and hence, we use the a priori estimates in \cite{Budhiraja-Dupuis-Ganguly}}.
\begin{lemma}\label{lem 234}
Fix $m\in(0,\infty)$.

\begin{itemize}
  \item[(a)] There exists $\varsigma_m\in(0,\infty)$ such that for all $I\in\mathcal{B}([0,T])$ and $\e\in(0,\infty)$,
\begin{eqnarray}\label{eq lem 4.2}
\sup_{\psi\in{S}_{+,\e}^m}\int_{Z\times I}\Big(L_1^2(y)+L_2^2(y)+L_3^2(z)\Big)\psi(y,s)\nu(\dif y)\dif s
\leq
\varsigma_m(a^2(\e)+Leb_T(I)).
\end{eqnarray}
  \item[(b)] There exist $\Gamma_m, \rho_m:(0,\infty)\to(0,\infty)$ such that $\Gamma_m(s)\downarrow 0$ as $s\uparrow \infty$, and for all $I\in\mathcal{B}([0,T])$ and
$\e,\beta\in(0,\infty)$,
\begin{align}\label{eq lem 4.3 1}
&\sup_{\varphi\in{S}^m_\e}\int_{Z\times I}\Big(L_1(z)+L_2(z)+L_3(z)\Big)|\varphi(y,s)|1_{\{|\varphi|\geq \beta/a(\e)\}}(y,s)\nu(\dif y)\dif s
\nonumber\\
\leq&
\Gamma_m(\beta)(1+\sqrt{Leb_T(I)}),
\end{align}
and
\begin{align}\label{eq lem 4.3 2}
&\sup_{\varphi\in{S}^m_\e}\int_{Z\times I}\!\!\Big(L_1(z)+L_2(z)+L_3(z)\Big)|\varphi(y,s)|\nu(\dif y)\dif s\cr
\leq&
\rho_m(\beta)\sqrt{Leb_T(I)}+\Gamma_m(\beta)a(\e).
\end{align}
  \item[(c)] For any $\beta>0$,
\begin{eqnarray}\label{eq lem 4.7}
\lim_{\e\to 0}\sup_{\varphi\in{S}^m_\e}\int_{Z\times [0,T]}\Big(L_1(z)+L_2(z)+L_3(z)\Big)|\varphi(y,s)|1_{\{|\varphi|>\beta/a(\e)\}}(y,s)\nu(\dif y)\dif s
=
0.
\end{eqnarray}
\end{itemize}
\end{lemma}

%
%
%\begin{lemma}\label{lem 4.2}
%Let $f\in L^2(\nu)\cap\mathcal{H}$ and fix $m>0$.
%Then there exists $\varsigma_f>0$ such that for any measurable subset $I$ of $[0,T]$ and for all $\e>0$,
%\begin{eqnarray}\label{eq lem 4.2}
%\sup_{\psi\in{S}_{+,\e}^m}\int_{Z\times I}f^2(y)\psi(y,s)\nu(dy)ds
%\leq
%\varsigma_f(a^2(\e)+Leb_T(I)).
%\end{eqnarray}
%\end{lemma}
%
%\begin{lemma}\label{lem 4.3}
%Let $f\in L^2(\nu)\cap\mathcal{H}$ and $I$ be a measurable subset of $[0,T]$. Fix $m>0$. Then
%there exists $\Gamma_f, \rho_f:(0,\infty)\to(0,\infty)$ such that $\Gamma_f(s)\downarrow 0$ as $s\uparrow \infty$, and for all
%$\e,\beta\in(0,\infty)$,
%$$
%\sup_{\varphi\in{S}^m_\e}\int_{Z\times I}|f(y)\varphi(y,s)|1_{\{|\varphi|\geq \beta/a(\e)\}}\nu(dy)ds
%\leq
%\Gamma_f(\beta)(1+\sqrt{Leb_T(I)}),
%$$
%and
%$$
%\sup_{\varphi\in{S}^m_\e}\int_{Z\times I}|f(y)\varphi(y,s)|\nu(dy)ds
%\leq
%\rho_f(\beta)\sqrt{Leb_T(I)}+\Gamma_f(\beta)a(\e).
%$$
%
%\end{lemma}
%
%\begin{lemma}\label{lem 4.7}
%Let $f\in L^2(\nu)\cap\mathcal{H}$ and suppose $f\geq 0$. Then for any $\beta>0$,
%\begin{eqnarray}\label{eq lem 4.7}
%\lim_{\e\to 0}\sup_{\varphi\in{S}^m_\e}\int_{Z\times [0,T]}|f(y)\varphi(y,s)|1_{\{|\varphi|>\beta/a(\e)\}}\nu(dy)ds
%=
%0.
%\end{eqnarray}
%\end{lemma}
\vskip 0.5cm

\bl\label{lem 6.6}
 Let $M^{u_\e}$ be the solution to (\ref{F-EQU}). Then  there exists $\kappa_0>0$ such that
\begin{eqnarray}\label{eq sup M}
\sup_{\e\in(0,\kappa_0]}\mathbb{E}\sup_{t\in[0,T]}|M^{u_\e}(t)|^2<\infty.
\end{eqnarray}
\el
\begin{proof}
By It\^o's formula, for any $t\in[0,T]$,
\begin{align}\label{eq MDP s}
&|M^{u_\e}(t)|^2\cr
=&
  \frac{2}{a(\e)}\int_0^t\<b_\e(s,a(\e)M^{u_\e}(s)+X^0(s),Law(X^\e(s)))
       -b(s,X^0(s),Law(X^0(s))),M^{u_\e}(s)\>\dif s\cr
&
+
\frac{2\sqrt{\e}}{a(\e)}\int_0^t\<M^{u_\e}(s),\sigma_\e(s,a(\e)M^{u_\e}(s)+X^0(s),Law(X^\e(s)))\dif  W(s)\>\cr
&
+
\frac{\e}{a^2(\e)}\int_0^t\|\sigma_\e(s,a(\e)M^{u_\e}(s)+X^0(s),Law(X^\e(s)))\|_{\mathcal{L}_2}^2\dif s\cr
&
+
2\int_0^t\<\sigma_\e(s,a(\e)M^{u_\e}(s)+X^0(s),Law(X^\e(s)))\phi_\e(s),M^{u_\e}(s)\>\dif s\cr
&+\frac{2\e}{a(\e)}\int_0^t\!\!\int_Z\<G_\e(s,a(\e)M^{u_\e}(s-)+X^0(s-),Law(X^\e(s)),z),
M^{u_\e}(s-)\>\widetilde{N}^{\e^{-1}\psi_\e}(\dif z,\dif s)\cr
&+\frac{\e^2}{a^2(\e)}\int_0^t\!\!\int_Z|G_\e(s,a(\e)M^{u_\e}(s-)+X^0(s-),Law(X^\e(s)),z)|^2
N^{\e^{-1}\psi_\e}(\dif z,\dif s)\cr
&+\frac{2}{a(\e)}\int_0^t\!\!\int_Z\<G_\e(s,a(\e)M^{u_\e}(s)+X^0(s),Law(X^\e(s)),z)
(\psi_\e(s,z)-1),M^{u_\e}(s)\>\nu(\dif z)\dif s\cr
=:&I_1(t)+I_2(t)+I_3(t)+I_4(t)+I_5(t)+I_6(t)+I_7(t).
\end{align}

(\ref{eq a ep}), (A3), (B1) and (B3) imply that there exists $\e_1>0$ such that
\begin{eqnarray}\label{eq xianz ep}
\frac{\e}{a^2(\e)}\vee a(\e)\vee \varrho_{b,\e}\vee\varrho_{\sigma,\e}\vee\varrho_{G,\e}\vee \frac{\varrho_{b,\e}}{a(\e)}\in(0,\frac{1}{2}],\ \ \forall\e\in(0,\e_1].
\end{eqnarray}
Recall the constant $\e_0$ in Lemma \ref{yey0}. Set $\e_2=\e_0\wedge\e_1\wedge \frac{1}{2}$.

Note that we
always use $C$ to denote a generic constant which may change from line to line and is independent of $\epsilon$.
\vskip 0.2cm
By (A1), (A3), Lemma \ref{yey0} and (\ref{eq Ct}), for any $\e\in(0,\e_2]$,
\begin{align}\label{I1}
I_1(t)
\leq&
\frac{2\varrho_{b,\e}}{a(\e)}\int_0^t|M^{u_\e}(s)|\dif s\cr
&+
2L\int_0^t|M^{u_\e}(s)|^2\dif s
+\frac{2}{a(\e)}\int_0^t|M^{u_\e}(s)|\mW_2(Law(X^\e(s)),Law(X^0(s)))\dif s\cr
\leq&C\int_0^t|M^{u_\e}(s)|^2\dif s+C,
%\leq&C\int_0^t|M^{u_\e}(s)|^2\dif s+C,
\end{align}
\begin{align}\label{I3}
&I_3(t)\cr
\leq&
\frac{C\e\varrho_{\sigma,\e}^2}{a^2(\e)}
+
C\e\int_0^t|M^{u_\e}(s)|^2\dif s\cr
&+\frac{C\e}{a^2(\e)}\left(\int_0^T\mW^2_2(Law(X^\e(s)),Law(X^0(s)))+
\|\sigma(s,X^0(s),Law(X^0(s)))\|_{\mathcal{L}_2}^2\dif s\right)\cr
\leq&C\int_0^t|M^{u_\e}(s)|^2\dif s+C,
\end{align}
and
\begin{align}\label{I4}
&I_4(t)\cr
\leq&
2\varrho_{\sigma,\e}\int_0^t|\phi_\e(s)||M^{u_\e}(s)|ds\cr
&+
C\int_0^t
   \left(a(\e)|M^{u_\e}(s)|+\mW_2(Law(X^\e(s)),Law(X^0(s)))
          +
           \|\sigma(s,X^0(s),Law(X^0(s)))\|_{\mathcal{L}_2}
   \right)\cr
   &\ \ \ \ \ \ \ \ \ |\phi_\e(s)||M^{u_\e}(s)|\dif s\cr
\leq&
Ca(\e)\int_0^t|M^{u_\e}(s)|^2|\phi_\e(s)|\dif s+C\Big(\Big(\mathbb{E}(\sup_{s\in[0,T]}|X^\e(s)-X^0(s)|^2)\Big)^{1/2}+1\Big)\int_0^t|\phi_\e(s)||M^{u_\e}(s)|\dif s\cr
&+
C\int_0^t\|\sigma(s,X^0(s),Law(X^0(s)))\|_{\mathcal{L}_2}|\phi_\e(s)||M^{u_\e}(s)|\dif s\cr
\leq&
C\int_0^t|M^{u_\e}(s)|^2(|\phi_\e(s)|^2+1)\dif s\cr
&+C\int_0^T\|\sigma(s,X^0(s),Law(X^0(s)))\|_{\mathcal{L}_2}^2\dif s
+C\int_0^T|\phi_\e(s)|^2\dif s\cr
\leq&
C\int_0^t|M^{u_\e}(s)|^2(|\phi_\e(s)|^2+1)\dif s+C,
\end{align}
where the last inequality follows from the fact that  $\phi_\e\in\mathcal{S}^m_1$.
\vskip 0.2cm

Recall $\varphi_\e(s,z)=\frac{\psi_\e(s,z)-1}{a(\e)}$. By (B1) and Lemma \ref{yey0}, for any $\e\in(0,\e_2]$,
\begin{align}\label{I7}
&I_7(t)\nonumber\\
=&2\int_0^t\!\!\int_Z\<G_\e(s,a(\e)M^{u_\e}(s)+X^0(s),Law(X^\e(s)),z)
\frac{\psi_\e(s,z)-1}{a(\e)},M^{u_\e}(s)\>\nu(\dif z)\dif s\nonumber\\
\leq&
2\int_0^t\!\!\int_Z\varrho_{G,\e}L_3(z)|\varphi_\e(s,z)||M^{u_\e}(s)|\nu(\dif z)\dif s\nonumber\\
&+2\int_0^t\!\!\int_Z\<\Big(G(s,a(\e)M^{u_\e}(s)+X^0(s),Law(X^\e(s)),z)-G(s,0,\delta_0,z)\Big)
\varphi_\e(s,z),M^{u_\e}(s)\>\nu(\dif z)\dif s\nonumber\\
&+
2\int_0^t\!\!\int_Z\<G(s,0,\delta_0,z)
\varphi_\e(s,z),M^{u_\e}(s)\>\nu(\dif z)\dif s\nonumber\\
\leq&
C\int_0^t\!\!\int_ZL_1(z)\left(a(\e)|M^{u_\e}(s)|+|X^0(s)|+\mW_2(Law(X^\e(s)),\delta_0)\right)
|\varphi_\e(s,z)||M^{u_\e}(s)|\nu(\dif z)\dif s\nonumber\\
+&C\int_0^t\!\!\int_Z(L_2(z)+L_3(z))|\varphi_\e(s,z)||M^{u_\e}(s)|\nu(\dif z)\dif s\\
\leq&C\int_0^t\!\!\int_Z(L_1(z)+L_2(z)+L_3(z))|\varphi_\e(s,z)|\nu(\dif z)|M^{u_\e}(s)|^2\dif s\nonumber\\
&+
C\int_0^T\!\!\int_Z(L_1(z)+L_2(z)+L_3(z))|\varphi_\e(s,z)|\nu(\dif z)\dif s.\nonumber
\end{align}
To deduce the last inequality, the following facts have been used
\begin{itemize}
  \item $X^0\in C([0,T],\mathbb{R}^d)$,
  \item $\mW_2(Law(X^\e(s)),\delta_0)\leq \mW_2(Law(X^\e(s)),Law(X^0(s)))+|X^0(s)|$.
\end{itemize}
%{\red Here I think it should be written more clearly!!}

Set
$$
D_\e:=\int_0^T(|\phi_\e(s)|^2+1)\dif s
+\int_0^T\!\!\int_Z(L_1(z)+L_2(z)+L_3(z))|\varphi_\e(s,z)|\nu(\dif z)\dif s.
$$

%By substituting (\ref{I1}), (\ref{I3}), (\ref{I4}), (\ref{I7}) back into (\ref{eq MDP s}) and applying Gronwall's inequality, we obtain
By substituting (\ref{I1})-(\ref{I7}) back into (\ref{eq MDP s}) and applying Gronwall's inequality, we obtain
\begin{align}\label{eq dd}
|M^{u_\e}(t)|^2
\leq e^{C D_\e}
\Bigg\{CD_\e+\sup_{s\in[0,T]}|I_2(s)+I_5(s)+I_6(s)|\Bigg\},\ \forall\e\in(0,\e_2],\ t\in[0,T].
\end{align}
Since $(\phi_\e,\varphi_\e)\in{S}^m_1\times{S}^m_\e\ P$-a.s., we have
\begin{eqnarray}\label{eq dd1}
\frac{1}{2}\int_0^T|\phi_\e(s)|^2\dif s\leq m,\ \ P\text{-a.s.} \ \forall \e\in(0,\e_2].
\end{eqnarray}

Hence, by (\ref{eq lem 4.3 2}), (\ref{eq dd}) and (\ref{eq dd1}), there exists a constant $\Lambda\in(0,\infty)$ such that for each $\e\in(0,\e_2]$,

\begin{align}\label{F}
&\mE\Big(\sup\limits_{t\in[0,T]}|M^{u_\e}(t)|^2\Big)\nonumber\\
\leq&\Lambda \Big\{1+\mE\Big(\sup\limits_{t\in[0,T]}|I_2(t)|\Big)
+\mE\Big(\sup\limits_{t\in[0,T]}|I_5(t)|\Big)
+\mE\Big(\sup\limits_{t\in[0,T]}|I_6(t)|\Big)\Big\}.
\end{align}
By Burkholder-Davis-Gundy's inequality, (A1), (A3), Young's inequality, Lemma \ref{yey0}, (\ref{eq Ct}) and (\ref{I3}), for any $\e\in(0,\e_2]$,
\begin{align}\label{I2}
&\mE\Big(\sup\limits_{t\in[0,T]}|I_2(t)|\Big)\cr
\leq&
\frac{C\sqrt{\e}}{a(\e)}\mE\left(\int_0^T|M^{u_\e}(s)|^2
\|\sigma_\e(s,a(\e)M^{u_\e}(s)+X^0(s),Law(X^\e(s)))\|_{\mathcal{L}_2}^2\dif s\right)^{\frac{1}{2}}\cr
\leq&
\frac{C\sqrt{\e}}{a(\e)}\mE\Big(\sup\limits_{s\in[0,T]}|M^{u_\e}(s)|^2\Big)
+
\frac{C\sqrt{\e}}{a(\e)}\mE\int_0^T\|\sigma_\e(s,a(\e)M^{u_\e}(s)+X^0(s),Law(X^\e(s)))\|_{\mathcal{L}_2}^2\dif s\cr
\leq&\frac{C\sqrt{\e}}{a(\e)}\mE\Big(\sup\limits_{s\in[0,T]}|M^{u_\e}(s)|^2\Big)
+
\frac{C\sqrt{\e}\varrho_{\sigma,\e}^2}{a(\e)}
+
C\sqrt{\e}a(\e)\mathbb{E}\int_0^T|M^{u_\e}(s)|^2\dif s\cr
&+\frac{C\sqrt{\e}}{a(\e)}\left(\int_0^T\mW^2_2(Law(X^\e(s)),Law(X^0(s)))+
\|\sigma(s,X^0(s),Law(X^0(s)))\|_{\mathcal{L}_2}^2\dif s\right)\cr
\leq&
C\Big(\frac{\sqrt{\e}}{a(\e)}+\sqrt{\e}a(\e)\Big)
\mE\Big(\sup\limits_{s\in[0,T]}|M^{u_\e}(s)|^2\Big)
+C.
\end{align}
Similarly, using (B1), we have  for any $\e\in(0,\e_2]$,
\begin{align}\label{I6-5}
&\Lambda\Big(\mE\Big(\sup\limits_{t\in[0,T]}|I_5(t)|\Big)+\mE\Big(\sup\limits_{t\in[0,T]}|I_6(t)|\Big)\Big)\cr
\leq&
\frac{C\e}{a(\e)}\mE\left(\int_0^T\!\!\!\int_Z|G_\e(s,a(\e)M^{u_\e}(s)+X^0(s),Law(X^\e(s)),z)|^2
|M^{u_\e}(s)|^2N^{\e^{-1}\psi_\e}(\dif z,\dif s)\right)^\frac{1}{2}\cr
&+\frac{\e}{a^2(\e)}\mE\int_0^T\!\!\!\int_Z|G_\e(s,a(\e)M^{u_\e}(s)+X^0(s),Law(X^\e(s)),z)|^2
\psi_\e(s,z)\nu(\dif z)\dif s\cr
\leq&
\frac{1}{10}\mE\Big(\sup\limits_{s\in[0,T]}|M^{u_\e}(s)|^2\Big)\cr
&+
\frac{C\e}{a^2(\e)}\mE\int_0^T\!\!\!\int_Z|G_\e(s,a(\e)M^{u_\e}(s)+X^0(s),Law(X^\e(s)),z)|^2
\psi_\e(s,z)\nu(\dif z)\dif s\cr
\leq&
\frac{1}{10}\mE\Big(\sup\limits_{s\in[0,T]}|M^{u_\e}(s)|^2\Big)\cr
&+
\frac{C\e}{a^2(\e)}\mE\int_0^T\!\!\!\int_Z\Big(a^2(\e)|M^{u_\e}(s)|^2+|X^0(s)|^2+\mE(|X^\e(s)|^2)\Big)L_1^2(z)
\psi_\e(s,z)\nu(\dif z)\dif s\cr
&+
\frac{C\e}{a^2(\e)}\mE\int_0^T\!\!\!\int_Z|G(s,0,\delta_0,z)|^2
\psi_\e(s,z)\nu(\dif z)\dif s
+
\frac{C\e\varrho_{G,\e}^2}{a^2(\e)}\mE\int_0^T\!\!\!\int_ZL_3^2(z)
\psi_\e(s,z)\nu(\dif z)\dif s
\cr
\leq&
    \frac{1}{10}\mE\Big(\sup\limits_{s\in[0,T]}|M^{u_\e}(s)|^2\Big)
     +
     C\e\sup\limits_{\psi\in{S}^K_{+,\e}}
            \int_0^T\!\!\!\int_ZL_1^2(z)\psi(s,z)\nu(\dif z)\dif s
                \mE\Big(\sup\limits_{s\in[0,T]}|M^{u_\e}(s)|^2\Big)\cr
       &\ \ \
       +
       C\sup\limits_{\psi\in{S}^K_{+,\e}}
            \int_0^T\!\!\!\int_Z(L_1^2(z)+L_2^2(z)+L_3^2(z))\psi(s,z)\nu(\dif z)\dif s\cr
         &\ \ \ \ \ \ \ \cdot\Big(
             1+\sup_{s\in[0,T]}|X^0(s)|^2+\mathbb{E}\Big(\sup_{s\in[0,T]}|X^\e(s)-X^0(s)|^2\Big)
          \Big)\cr
\leq&
\frac{1}{10}\mE\Big(\sup\limits_{s\in[0,T]}|M^{u_\e}(s)|^2\Big)
+
C\e\sup\limits_{\psi\in{S}^K_{+,\e}}
\int_0^T\!\!\!\int_ZL_1^2(z)\psi(s,z)\nu(\dif z)\dif s
\mE\Big(\sup\limits_{s\in[0,T]}|M^{u_\e}(s)|^2\Big)\cr
&+
C\sup\limits_{\psi\in{S}^K_{+,\e}}
\int_0^T\!\!\!\int_Z(L_1^2(z)+L_2^2(z)+L_3^2(z))\psi(s,z)\nu(\dif z)\dif s.
\end{align}
By (\ref{F})--(\ref{I6-5}), we have for $\forall \e\in(0,\e_2]$,
\begin{align*}
&\Big(\frac{9}{10}-C\frac{\sqrt{\e}}{a(\e)}+C\sqrt{\e}a(\e)-C\e\sup\limits_{\psi\in{S}^K_{+,\e}}
\int_0^T\!\!\!\int_ZL_1^2(z)\psi(s,z)\nu(\dif z)\dif s\Big)\mE\Big(\sup\limits_{s\in[0,T]}|M^{u_\e}(s)|^2\Big)\\
\leq&
C\Big(1+\sup\limits_{\psi\in{S}^K_{+,\e}}
\int_0^T\!\!\!\int_Z(L_1^2(z)+L_2^2(z)+L_3^2(z))\psi(s,z)\nu(\dif z)\dif s\Big).
\end{align*}
In view of  (\ref{eq lem 4.2}) and (\ref{eq a ep}), this yields that there exists $\kappa_0>0$ such that
\begin{eqnarray*}
\sup_{\e\in(0,\kappa_0]}\mE\Big(\sup\limits_{s\in[0,T]}|M^{u_\e}(s)|^2\Big)
<
\infty.
\end{eqnarray*}

The proof of this lemma is completed.
\end{proof}

\vskip 0.3cm
The verification of (\textbf{MDP2}) is given in the next proposition.
\vskip 0.2cm
Recall $\widetilde{u}_\e$ in (\ref{eq uuu}).
\begin{proposition}\label{Lemma 6.7}
For any $\varpi>0$,
 $\lim_{\e\rightarrow0}P\Big(\sup\limits_{t\in[0,T]}|M^{u_\e}(t)-K^{\widetilde{u}_\e}(t)|>\varpi\Big)=0.$
\end{proposition}

\begin{proof}
For each fixed $\e>0$ and $j\in\mN$, define a stopping time
\begin{align*}
\tau_\e^j=\inf\{t\geq0:|M^{u_\e}(t)|\geq j\}\wedge T.
\end{align*}
By  Lemma \ref{lem 6.6}, we have
\begin{eqnarray}\label{eq stop}
P(\tau_\e^j< T)\leq \frac{\mathbb{E}\Big(\sup_{t\in[0,T]}|M^{u_\e}(t)|^2\Big)}{j^2}\leq \frac{C}{j^2},\ \ \forall \e\in(0,\kappa_0]
\end{eqnarray}
where $\kappa_0$ is the same as in Lemma \ref{lem 6.6}.
\vskip 0.3cm

Let $Q^\e(s)=M^{u_\e}(s)-K^{\widetilde{u}_\e}(s)$ for each $s\in[0,T]$. Notice that the corresponding equations $M^{u_\e}$ and $K^{\widetilde{u}_\e}$ satisfied are distribution-independent SDEs. By It\^o's formula, we have
{\small
\begin{align}\label{F-Y}
&|Q^\e(t\wedge\tau_\e^j)|^2\cr
=&
2\int_0^{t\wedge\tau_\e^j}\!\!\!
\<\frac{1}{a(\e)}\Big(b_\e(s,a(\e)M^{u_\e}(s)+X^0(s),Law(X^\e(s)))-b(s,X^0(s),Law(X^0(s)))\Big)\nonumber\\
&
\ \ \ \ \ \ \ \ \ \ \ \ -b'_2(s,X^0(s),Law(X^0(s)))K^{\widetilde{u}_\e}(s), Q^\e(s)\>\dif s\cr
&
+
2\frac{\sqrt{\e}}{a(\e)}\int_0^{t\wedge\tau_\e^j}\!\!\!\<Q^\e(s),
\sigma_\e(s,a(\e)M^{u_\e}(s)+X^0(s),Law(X^\e(s)))\dif  W(s)\>\cr
&
+
\frac{\e}{a^2(\e)}\int_0^{t\wedge\tau_\e^j}\!\!\!\|
\sigma_\e(s,a(\e)M^{u_\e}(s)+X^0(s),Law(X^\e(s)))\|^2_{\mathcal{L}_2}\dif  s\cr
&
+
2\int_0^{t\wedge\tau_\e^j}\!\!\!\!
\<(\sigma_\e(s,a(\e)M^{u_\e}(s)+X^0(s),Law(X^\e(s)))\!-\!\sigma(s,X^0(s),Law(X^0(s))))\phi_\e(s), Q^\e(s)\>\dif s\cr
&
+
\frac{2\e}{a(\e)}\int_0^{t\wedge\tau_\e^j}\!\!\!
\int_Z\<G_\e(s,a(\e)M^{u_\e}(s-)+X^0(s-),Law(X^\e(s)),z),
Q^\e(s-)\>\widetilde{N}^{\e^{-1}\psi_\e}(\dif z,\dif s)\cr
&
+
\frac{\e^2}{a^2(\e)}\int_0^{t\wedge\tau_\e^j}
\!\!\int_Z|G_\e(s,a(\e)M^{u_\e}(s-)+X^0(s-),Law(X^\e(s)),z)|^2
N^{\e^{-1}\psi_\e}(\dif z,\dif s)\cr
&
+
2\int_0^{t\wedge\tau_\e^j}\!\!\!\int_Z\
   \langle G_\e(s,a(\e)M^{u_\e}(s)+X^0(s),Law(X^\e(s)),z)\varphi_\e(s,z)\nonumber\\
&
    \ \ \ \ \ \ \ \ \ \ -
   G(s,X^0(s),Law(X^0(s)),z)\varphi_\e(s,z)1_{\{|\varphi_\e|\leq\beta/a(\e)\}}(s,z),
    Q^\e(s)\rangle\nu(\dif z)\dif s\cr
=&I_1(t)+I_2(t)+I_3(t)+I_4(t)+I_5(t)+I_6(t)+I_7(t).
\end{align}
}

Due to (\ref{eq E M 1}) and the fact that $\tilde{u}_\e\in S^m_1\times B_2(\sqrt{m\kappa_2(1)})\ $, there exists some $\Omega^0\in\mathcal{F}$ with $P(\Omega^0)=1$ such that
\begin{align}\label{eq 3.82}
\kappa:=\sup\limits_{\e\in(0,\kappa_0]}
\sup\limits_{\omega\in\Omega^0,t\in[0,T]}|K^{\widetilde{u}_\e}(t)(\omega)|<\infty.
\end{align}

Recall the constant $\e_2$ appearing in (\ref{eq xianz ep}). Set
$\e_3=\e_2\wedge\kappa_0.$ By the mean value theorem and (B2), for any $\epsilon\in(0,\e_3]$, there exists $\theta_\e(s)\in[0,1]$ such that
{\small
\begin{align*}
&I_{1,1}(t)\\
:=&
\int_0^{t\wedge\tau_\e^j}\<\frac{b(s,a(\e)M^{u_\e}(s)+X^0(s),Law(X^0(s)))-b(s,X^0(s),Law(X^0(s)))}{a(\e)}\\
&\ \ \ \ \ \
-b'_2(s,X^0(s),Law(X^0(s)))M^{u_\e}(s),Q^\e(s)\>\dif s\cr
\leq&
\int_0^{t\wedge\tau_\e^j}\!\!\!\!|b'_2(s,a(\e)M^{u_\e}(s)\theta_\e(s)\!
+\!X^0(s),Law(X^0(s)))\!-\!b'_2(s,X^0(s),Law(X^0(s)))||M^{u_\e}(s)||Q^\e(s)|\dif s\cr
\leq&
La(\e)\int_0^{t\wedge\tau_\e^j}(1+|a(\e)M^{u_\e}(s)\theta_\e(s)+X^0(s)|^{q'}+|X^0(s)|^{q'})
|M^{u_\e}(s)|^2|Q^\e(s)|\dif s\\
\leq&
C_ja(\e),
\end{align*}
}
where
$$
C_j=L\Big(1+|j+\sup_{s\in[0,T]}|X^0(s)||^{q'}+\sup_{s\in[0,T]}|X^0(s)|^{q'}\Big)j^2(j+\kappa)T,
$$
which is independent of $\e$.

In the sequel, $C_j$ will denote generic  constants which are  independent of $\e$.

\vskip 0.2cm
Hence, by Lemma \ref{yey0}, (A1) and the above inequality, for any $\epsilon\in(0,\e_3]$,
{\small
\begin{align}\label{I-1}
&I_1(t)\nonumber\\
=&2I_{1,1}(t)\nonumber\\
&+2\int_0^{t\wedge\tau_\e^j}\!\!\!\!\<\frac{b_\e(s,a(\e)M^{u_\e}(s)+X^0(s),Law(X^\e(s)))
-b(s,a(\e)M^{u_\e}(s)+X^0(s),Law(X^\e(s)))}{a(\e)}, Q^\e(s)\>\dif s\nonumber\\
&+2\int_0^{t\wedge\tau_\e^j}\!\!\!\!\<b'_2(s,X^0(s),Law(X^0(s)))M^{u_\e}(s)
-b'_2(s,X^0(s),Law(X^0(s)))K^{\widetilde{u}_\e}(s),Q^\e(s)\>\dif s\nonumber\\
&+2\int_0^{t\wedge\tau_\e^j}\!\!\!\<\frac{b(s,a(\e)M^{u_\e}(s)+X^0(s),Law(X^\e(s)))
-b(s,a(\e)M^{u_\e}(s)+X^0(s),Law(X^0(s)))}{a(\e)},Q^\e(s)\>\dif s\nonumber\\
\leq&
2\frac{\varrho_{b,\e}}{a(\e)}\int_0^{t\wedge\tau_\e^j}|Q^\e(s)|\dif s
+
\frac{2L}{a(\e)}\int_0^{t\wedge\tau_\e^j}(\mE|X^\e(s)-X^0(s)|^2)^{\frac{1}{2}}|Q^\e(s)|\dif s\nonumber\\
&
+C_ja(\e)+2\int_0^t|b'_2(s\wedge\tau_\e^j,X^0(s\wedge\tau_\e^j),Law(X^0(s\wedge\tau_\e^j)))|
|Q^\e({s\wedge\tau_\e^j})|^2\dif s\nonumber\\
\leq&
C_j\left(\frac{\varrho_{b,\e}}{a(\e)}+\frac{(\e+\varrho_{b,\e}^2+\e\varrho_{\sigma,\e}^2
+\e\varrho_{G,\e}^2)^{\frac{1}{2}}}{a(\e)}+a(\e)\right)\nonumber\\
&+
2\int_0^t|b'_2(s\wedge\tau_\e^j,X^0(s\wedge\tau_\e^j),Law(X^0(s\wedge\tau_\e^j)))|
|Q^\e({s\wedge\tau_\e^j})|^2\dif s.
\end{align}
}
\vskip 0.2cm
Inserting the inequality (\ref{I-1}) into (\ref{F-Y}), and using Gronwall's inequality,
we deduce that, for any $\epsilon\in(0,\e_3]$ and $t\in[0,T]$,
\begin{align}\label{eq F-Y 1}
&\sup_{t\in[0,T]}|Q^\e(t\wedge\tau_\e^j)|^2\nonumber\\
\leq&\exp\Big\{2\int_0^T|b'_2(s,X^0(s),Law(X^0(s)))|\dif s\Big\}\nonumber\\
&\times\left\{C_j\big(\frac{\varrho_{b,\e}}{a(\e)}+\frac{(\e+\varrho_{b,\e}^2+\e\varrho_{\sigma,\e}^2
+\e\varrho_{G,\e}^2)^{\frac{1}{2}}}{a(\e)}+a(\e)\big)
+\sum_{i=2}^7\sup_{s\in[0,T]}|I_i(s)|\right\}.%\nonumber\\
%\leq&
%C\left\{C_j\big(\frac{\varrho_{b,\e}}{a(\e)}+\frac{(\e+\varrho_{b,\e}^2+\e\varrho_{\sigma,\e}^2
%+\e\varrho_{G,\e}^2)^{\frac{1}{2}}}{a(\e)}+a(\e)\big)
%+\sum_{i=2}^7\sup_{s\in[0,T]}|I_i(s)|\right\}.
\end{align}
Set
$$
\Lambda:=\exp\Big\{2\int_0^T|b'_2(s,X^0(s),Law(X^0(s)))|\dif s\Big\}.
$$
By Burkholder-Davis-Gundy's inequality and (\ref{eq Ct}), using similar arguments as in the proofs of  (\ref{I3}) and (\ref{I2}), one can obtain for each $\e\in(0,\e_3]$,
\begin{align}\label{eq I2-I3}
&\Lambda\Big(\mathbb{E}\Big(\sup_{t\in[0,T]}|I_2(t)|\Big)
+\mathbb{E}\Big(\sup_{t\in[0,T]}|I_3(t)|\Big)\Big)\nonumber\\
\leq&
   \frac{1}{10} \mathbb{E}\Big(\sup_{t\in[0,T]}|Q^\e(t\wedge\tau_\e^j)|^2\Big)\nonumber\\
&+
  \frac{C\e}{a^2(\e)}\mathbb{E}\Big(\int_0^{T\wedge\tau_\e^j}\|
                            \sigma_\e(s,a(\e)M^{u_\e}(s)+X^0(s),Law(X^\e(s)))\|^2_{\mathcal{L}_2}\dif  s\Big)\nonumber\\
\leq&
   \frac{1}{10} \mathbb{E}\Big(\sup_{t\in[0,T]}|Q^\e(t\wedge\tau_\e^j)|^2\Big)
   +
   \frac{C\e \varrho_{\sigma,\e}^2}{a^2(\e)}
   \nonumber\\
&+
  \frac{C\e}{a^2(\e)}\mathbb{E}\Big(\int_0^{T\wedge\tau_\e^j}
  |M^{u_\e}(s)|^2+\mathbb{E}(|X^\e(s)-X^0(s)|^2)\dif  s\Big)\nonumber\\
&+
  \frac{C\e}{a^2(\e)}\int_0^{T}\|
                            \sigma(s,X^0(s),Law(X^0(s)))\|^2_{\mathcal{L}_2}\dif  s\Big)\nonumber\\
\leq&
   \frac{1}{10} \mathbb{E}\Big(\sup_{t\in[0,T]}|Q^\e(t\wedge\tau_\e^j)|^2\Big)\nonumber\\
&+
 C_j\frac{\e}{a^2(\e)}(1+\varrho_{\sigma,\e}^2+\e+\varrho_{b,\e}^2
 +\e\varrho_{\sigma,\e}^2+\e\varrho_{G,\e}^2).
\end{align}

By (A1), (\ref{eq 3.82}) and Lemma \ref{yey0}, remembering that $\phi_\e\in \mathcal{S}^m_1$ , we have for any $\epsilon\in(0,\e_3]$,
\begin{align}\label{I-4}
&\Lambda\mathbb{E}\Big(\sup_{t\in[0,T]}|I_4(t)|\Big)\nonumber\\
\leq&
C\varrho_{\sigma,\e}\mathbb{E}\int_0^{T\wedge\tau_\e^j}|\phi_\e(s)||Q^\e(s)|\dif s\nonumber\\
&+
C\mathbb{E}\int_0^{T\wedge\tau_\e^j}(|a(\e)M^{u_\e}(s)|+\mW_2(Law(X^\e(s)),Law(X^0(s))))
|\phi_\e(s)||Q^\e(s)|\dif s\nonumber\\
\leq&C_j(\varrho_{\sigma,\e}+a(\e)+(\e+\varrho_{b,\e}^2
+\e\varrho_{\sigma,\e}^2+\e\varrho_{G,\e}^2)^{\frac{1}{2}})
\mathbb{E}\Big(
\int_0^T|\phi_\e(s)|^2\dif s
\Big)^{\frac{1}{2}}\nonumber\\
\leq&C_j(\varrho_{\sigma,\e}+a(\e)+(\e+\varrho_{b,\e}^2
+\e\varrho_{\sigma,\e}^2+\e\varrho_{G,\e}^2)^{\frac{1}{2}}).
\end{align}

By Burkholder-Davis-Gundy's inequality, (\ref{eq lem 4.2}) and (\ref{eq xianz ep}), using similar arguments as in the proof of (\ref{I6-5}), one has for $\forall \e\in(0,\e_3]$,
\begin{align}\label{eq I5-6}
&\Lambda\Big(\mathbb{E}\Big(\sup_{t\in[0,T]}|I_5(t)|\Big)
+\mathbb{E}\Big(\sup_{t\in[0,T]}|I_6(t)|\Big)\Big)\nonumber\\
\leq&
   \frac{1}{10} \mathbb{E}\Big(\sup_{t\in[0,T]}|Q^\e(t\wedge\tau_\e^j)|^2\Big)\nonumber\\
&+\frac{C\e}{a^2(\e)}\mathbb{E}\Big(\int_0^{T\wedge\tau_\e^j}\!\!\int_Z
   |G_\e(s,a(\e)M^{u_\e}(s)+X^0(s),Law(X^\e(s)),z)|^2\psi_\e(s,z)\nu(\dif z)\dif s\Big)\cr
\leq&
   \frac{1}{10}\mathbb{E}\Big(\sup_{t\in[0,T]}|Q^\e(t\wedge\tau_\e^j)|^2\Big)
   +
   \frac{C\e\varrho_{G,\e}^2}{a^2(\e)}\mathbb{E}\Big(\int_0^{T\wedge\tau_\e^j}
   \!\!\!\!\int_ZL_3^2(z)\psi_\e(s,z)\nu(\dif z)\dif s\Big)\cr
&+
   \frac{C\e}{a^2(\e)}\mathbb{E}\Big(\int_0^{T\wedge\tau_\e^j}
   \!\!\!\!\int_Z|G(s,0,\delta_0,z)|^2\psi_\e(s,z)\nu(\dif z)\dif s\Big)\cr
&+
\frac{C\e}{a^2(\e)}\mathbb{E}\Big(\int_0^{T\wedge\tau_\e^j}\!\!\!\!
\int_Z\!\!\Big(|a(\e)M^{u_\e}(s)+X^0(s)|^2+\mathbb{E}(|X^\e(s)|^2)\Big)
L^2_1(z)\psi_\e(s,z)\nu(\dif z)\dif s\Big)\cr
\leq&
   \frac{1}{10} \mathbb{E}\Big(\sup_{t\in[0,T]}|Q^\e(t\wedge\tau_\e^j)|^2\Big)\nonumber\\
&+
\frac{C_j\e}{a^2(\e)}\sup_{\psi\in S^m_{+,\e}}\Big(\int_0^T\!\!\!\int_Z\!\!\Big(L^2_1(z)+L^2_2(z)+L^2_3(z)\Big)\psi(s,z)\nu(\dif z)\dif s\Big)\cr
\leq&
   \frac{1}{10} \mathbb{E}\Big(\sup_{t\in[0,T]}|Q^\e(t\wedge\tau_\e^j)|^2\Big)
+
\frac{C_j\e}{a^2(\e)}.
\end{align}

Note that
\begin{align*}
&I_7(t)\\
=&
2\int_0^{t\wedge\tau_\e^j}\!\!\!\!\!\!\int_Z\
\langle \big(G_\e(s,a(\e)M^{u_\e}(s)+X^0(s),Law(X^\e(s)),z)\cr
   &\ \ \ \ \ \ \ \ \ \ \ \ \ \ \ \ -
   G(s,a(\e)M^{u_\e}(s)+X^0(s),Law(X^\e(s)),z)\big)\varphi_\e(s,z),
  Q^\e(s)\rangle\nu(\dif z)\dif s\cr
+&
2\int_0^{t\wedge\tau_\e^j}\!\!\!\!\!\!\int_Z\!\!
   \langle\big(G(s,a(\e)M^{u_\e}(s)+X^0(s),Law(X^\e(s)),z)\cr
   &\ \ \ \ \ \ \ \ \ \ \ \ \ \ \ \ -
   G(s,X^0(s),Law(X^0(s)),z)\big)\varphi_\e(s,z), Q^\e(s)\rangle\nu(\dif z)\dif s\cr
+&
2\int_0^{t\wedge\tau_\e^j}\!\!\!\!\!\int_Z\!\!
   \langle
   \big(G(s,X^0(s),Law(X^0(s)),z)\!-\!G(s,0,\delta_0,z)\big)
   \varphi_\e(s,z)1_{\{|\varphi_\e|>\beta/a(\e)\}}(s,z),Q^\e(s)\rangle\nu(\dif z)\dif s\cr
+&
2\int_0^{t\wedge\tau_\e^j}\!\!\!\!\int_Z\!\!
   \langle
   G(s,0,\delta_0,z)\varphi_\e(s,z)1_{\{|\varphi_\e|>\beta/a(\e)\}}(s,z),
Q^\e(s)\rangle\nu(\dif z)\dif s.
\end{align*}
Hence, by (B1), (\ref{eq 3.82}), Remark \ref{rem W2} and Lemma \ref{yey0},
\begin{align}\label{I-7-1}
&\Lambda\mathbb{E}\Big(\sup_{t\in[0,T]}|I_7(t)|\Big)\nonumber\\
\leq&
C\varrho_{G,\e}\mathbb{E}\Big(
\int_0^{T\wedge\tau_\e^j}\!\!\!\int_ZL_3(z)|Q^\e(s)||\varphi_\e(s,z)|\nu(\dif z)\dif s
\Big)\cr
&+
C\mathbb{E}\Big(
\int_0^{T\wedge\tau_\e^j}\!\!\!\int_ZL_1(z)
\left(|a(\e)M^{u_\e}(s)|+\Big(\mathbb{E}(X^\e(s)-X^0(s))^2\Big)^{\frac{1}{2}}\right)
|Q^\e(s)||\varphi_\e(s,z)|\nu(\dif z)\dif s\Big)\cr
&
+
C\mathbb{E}\Big(
\int_0^{t\wedge\tau_\e^j}\!\!\!\int_ZL_1(z)|X^0(s)|
|Q^\e(s)||\varphi_\e(s,z)|1_{\{|\varphi_\e|>\beta/a(\e)\}}(s,z)\nu(\dif z)\dif s
\Big)\cr
&
+
C\mathbb{E}\Big(
\int_0^{t\wedge\tau_\e^j}\!\!\!\int_Z|Q^\e(s)|L_2(z)
|\varphi_\e(s,z)|1_{\{|\varphi_\e|>\beta/a(\e)\}}(s,z)\nu(\dif z)\dif s
\Big)\cr
\leq&
C_j(\varrho_{G,\e}+a(\e)+(\e+\varrho_{b,\e}^2+\e\varrho_{\sigma,\e}^2
+\e\varrho_{G,\e}^2)^{\frac{1}{2}})\sup_{\varphi\in S^m_{\e}}\int_0^{T}\!\!\!\int_Z(L_1(z)+L_3(z))
|\varphi(s,z)|\nu(\dif z)\dif s\cr
&
+
C_j\sup_{\varphi\in S^m_{\e}}\int_0^{T}\!\!\!\int_Z\Big(L_1(z)+L_2(z)\Big)
|\varphi(s,z)|1_{\{|\varphi(s,z)|>\frac{\beta}{a(\e)}\}}(s,z)\nu(\dif z)\dif s\cr
\leq&
C_j(\varrho_{G,\e}+a(\e)+(\e+\varrho_{b,\e}^2+\e\varrho_{\sigma,\e}^2
+\e\varrho_{G,\e}^2)^{\frac{1}{2}})\cr
&+
C_j\sup_{\varphi\in S^m_{\e}}\int_0^{T}\!\!\!\int_Z\Big(L_1(z)+L_2(z)\Big)
|\varphi(s,z)|1_{\{|\varphi(s,z)|>\frac{\beta}{a(\e)}\}}(s,z)\nu(\dif z)\dif s.
\end{align}

By combining (\ref{eq xianz ep}) and (\ref{F-Y})-(\ref{I-7-1}) together, we obtain for any $\epsilon\in(0,\e_3]$,
\begin{align}\label{eq F-Y 4}
&\frac{8}{10}\mathbb{E}\Big(\sup_{t\in[0,T]}|Q^\e_{t\wedge\tau_\e^j}|^2\Big)\nonumber\\
\leq&
C_j\Big\{\varrho_{G,\e}+\varrho_{\sigma,\e}+\frac{\varrho_{b,\e}}{a(\e)}
+\frac{(\e+\varrho_{b,\e}^2+\e\varrho_{\sigma,\e}^2
+\e\varrho_{G,\e}^2)^{\frac{1}{2}}}{a(\e)}
+\frac{\e}{a^2(\e)}\\
&+a(\e)+(\e+\varrho_{b,\e}^2
+\e\varrho_{\sigma,\e}^2+\e\varrho_{G,\e}^2)^{\frac{1}{2}}\nonumber\\
&+\sup_{\varphi\in S^m_{\e}}\int_0^{T}\!\!\!\int_Z\big(L_1(z)+L_2(z)\big)
|\varphi(s,z)|1_{\{|\varphi(s,z)|>\frac{\beta}{a(\e)}\}}(s,z)\nu(\dif z)\dif s\Big\}.\nonumber
\end{align}

In view of (B3), (\ref{eq a ep}) and (\ref{eq lem 4.7}) it follows that
\begin{align}\label{eq F-Y 3}
\lim_{\e\rightarrow 0}\mathbb{E}\Big(\sup_{t\in[0,T]}|M^{u_\e}(t\wedge\tau_\e^j)-K^{\widetilde{u}_\e}(t\wedge\tau_\e^j)|^2\Big)
=\lim_{\e\rightarrow 0}\mathbb{E}\Big(\sup_{t\in[0,T]}|Q^\e(t\wedge\tau_\e^j)|^2\Big)=
0.
\end{align}
\vskip 0.3cm

Now for any $\varpi>0$, $\epsilon\in(0,\e_3]$ and $j\in\mathbb{N}$,  we have
 %$\lim_{\e\rightarrow0}P\Big(\sup\limits_{t\in[0,T]}|M^{u_\e}(t)-K^{\widetilde{u}_\e}(t)|>\varpi)=0.$
\begin{align*}
&P\Big(\sup\limits_{t\in[0,T]}|M^{u_\e}(t)-K^{\widetilde{u}_\e}(t)|\geq \varpi\Big)\\
\leq&
  P\left(
     \Big(\sup\limits_{t\in[0,T]}|M^{u_\e}(t\wedge\tau_\e^j)-
       K^{\widetilde{u}_\e}(t\wedge\tau_\e^j)|\geq \varpi\Big)
      \bigcap
       \Big(
            \tau_\e^j\geq T
        \Big)
   \right)
  +
   P(\tau_\e^j< T)\cr
\leq&\frac{1}{\varpi^2}\mE\Big(\sup\limits_{t\in[0,T]}
|M^{u_\e}(t\wedge\tau_\e^j)-K^{\widetilde{u}_\e}(t\wedge\tau_\e^j)|^2\Big)
+
\frac{C}{j^2}.
\end{align*}
By letting $\varepsilon\rightarrow 0$ first and then $j\rightarrow \infty$,  we get
\begin{align*}
\lim\limits_{\e\rightarrow0}
P\left(\sup\limits_{t\in[0,T]}|M^{u_\e}(t)-K^{\widetilde{u}_\e}(t)|\geq \varpi\right)=0,
\end{align*}
which is the desired result.
\end{proof}

\end{document}